\documentclass[11pt]{amsart}
\usepackage{amssymb,graphics,epsfig,overpic}

\newtheorem{thm}{Theorem}[section]

\newtheorem{lem}[thm]{Lemma}
\newtheorem{cor}[thm]{Corollary}
\newtheorem{prop}[thm]{Proposition}

\theoremstyle{definition}

\newtheorem{example}[thm]{Example}
\newtheorem{note}[thm]{Note}

\newcommand{\R}{\mathbf{R}}

\newcommand{\ol}{\overline}
\newcommand{\C}{\mathcal{C}}
\newcommand{\p}{\partial}
\newcommand{\f}{f|_{\p M}}

\renewcommand{\P}{\mathcal{P}}

\renewcommand{\S}{\mathbf{S}}
\renewcommand{\l}{\langle}
\renewcommand{\r}{\rangle}
\renewcommand{\(}{\left(}
\renewcommand{\)}{\right)}
\renewcommand{\tilde}{\widetilde}

\DeclareMathOperator{\inte}{int}
\DeclareMathOperator{\conv}{conv}

\DeclareMathOperator{\cl}{cl}

\DeclareMathOperator{\distort}{distort}
\DeclareMathOperator{\length}{length}

\DeclareMathOperator{\rank}{rank}

\title[Boundary Torsion and Convex Caps]{Boundary Torsion and Convex Caps\\ of Locally Convex Surfaces}

\author{Mohammad Ghomi}
\address{School of Mathematics, Georgia Institute of Technology,
Atlanta, GA 30332}
\email{ghomi@math.gatech.edu}
\urladdr{www.math.gatech.edu/$\sim$ghomi}

\date{\today \,(Last Typeset)}
\thanks{Research of the  author was supported in part by NSF Grant DMS--1308777, and Simons Collaboration Grant 279374.}

\begin{document}

\vspace*{-0.5in}

\begin{abstract} 
We prove that the torsion of  any closed space curve which bounds a simply connected locally convex surface  vanishes at least  $4$ times. 
This answers a question of Rosenberg related to a problem of Yau on characterizing the boundary of positively curved disks in Euclidean space. Furthermore, our result generalizes the  $4$ vertex theorem of Sedykh for convex space curves, and thus constitutes  a far reaching extension of the classical $4$ vertex theorem. The proof involves  studying  the arrangement of convex caps in a locally convex surface, and yields a Bose type formula for these objects.
\end{abstract}

\maketitle

\vspace{-0.3in}

\tableofcontents

\section{Introduction}

When does a closed curve $\Gamma$ in Euclidean space $\R^3$ bound a disk of positive curvature? This question, which appears in Yau's list of open problems \cite[\#26]{yau:problems2},  has been investigated by a number of authors \cite{ghomi:stconvex, rosenberg:constant, gluck&pan}. In particular, in 1993  Rosenberg \cite[Sec. 1.3]{rosenberg:constant} asked whether a necessary condition is that $\Gamma$ have (at least) $4$ \emph{vertices}, i.e., points where the torsion vanishes (or the first three derivatives of $\Gamma$ are linearly dependent). Here we show that the answer is yes. To state our main result,
let $M$ denote a topological disk, with   boundary $\p M$. A \emph{locally convex immersion}  $f\colon M\to\R^3$ is a continuous locally one-to-one map which sends a neighborhood $U$ of each point $p$ of $M$ into the boundary of a convex body $K$ in $\R^3$. We say that $f$ is \emph{locally nonflat} along $\p M$ provided that for no point $p\in\p M$, $f(U)$  lies in a plane. 
Further, $\f$ is of regularity class $\C^{k\geq 1}$ if  $f\colon \partial M\to\R^3$ is $k$-times continuously differentiable, and its differential is nonvanishing.
 The torsion $\tau$ of $\f$ \emph{changes sign} (at least) $n$ times, provided that there are $n$ points cyclically arranged in $\p M$ where the sign of $\tau$ alternates.

\begin{thm}[Main Theorem, First Version]\label{thm:main}
Let $f\colon M\to\R^3$ be a locally convex immersion.
Suppose that $\f$ is $\C^3$, has no inflections, and $f$ is locally nonflat along $\p M$. Then 
  either $\tau$ vanishes identically, or  else it changes sign $4$ times.
\end{thm}

Note that if $f$ is locally flat at a point $p\in\p M$, then $\tau$ vanishes throughout a neighborhood of $p$, which yields infinitely many vertices. Thus the above theorem ensures the existence of $4$ boundary vertices for all locally convex immersions $f\colon M\to\R^3$ for which $\tau$ is well defined ($f$ is not required to have any degree of regularity in the interior of $M$). In particular, 
Theorem \ref{thm:main} generalizes a well-known result of Sedykh \cite{sedykh:vertex} who had established the existence of $4$ vertices in the case where $f$ embeds $M$ into the boundary of a single convex body. In this sense, Theorem \ref{thm:main} is a substantial extension of  the classical four vertex theorem for planar curves. Indeed  points of vanishing torsion of a space curve are natural generalizations of critical points of curvature of a planar curve, e.g., see \cite[Note 1.5]{ghomi:verticesC}. We should also note that  the local nonflatness assumption along $\p M$ here  is necessary to ensure that the sign of $\tau$ behaves in the claimed manner, see Example \ref{ex:dumbbell}.

\begin{cor}\label{cor:main}
Let $f\colon M\to\R^3$ be a $\C^3$ immersion with nonnegative curvature. Suppose that $f$ has positive curvature on $\p M$.  Then either $\tau$ vanishes identically, or  else it changes sign $4$ times.
\end{cor}
\begin{proof}
Since $f$  has positive curvature on $\p M$,  $\f$ is automatically devoid of inflections, and thus $\tau$ is well-defined. Further $f$ is locally nonflat along $\p M$. So, to apply Theorem \ref{thm:main}, we just need to check that $f$ is locally convex. This has been shown in \cite[Prop. 3.3]{alexander&ghomi:chp}. Indeed, even though \cite[Prop. 3.3]{alexander&ghomi:chp}, which holds in $\R^n$, has been stated subject to $f$ being $\C^\infty$, the only regularity requirement used there is that $f$ be $\C^{m+1}$, where $m=\dim(M)$, in order to invoke Sard's theorem \cite[Sec. 3.4.3]{federer:book} in \cite[Lem. 3.1]{alexander&ghomi:chp}. Thus, since $m=2$ in the present case, we only need $f$ to be $\C^3$ in order to ensure local convexity.
\end{proof}

Thus we obtain a new necessary condition for the existence of a positively curved disk spanning a given curve. Only one other nontrivial obstruction, involving the self-linking number of the curve, has been known up to now \cite{rosenberg:constant}, see Example \ref{ex:negative}. The above result is of interest also for studying the Plateau problem for surfaces of constant curvature \cite{smith:lecturenotes}, since it has been shown in recent years
by Guan and Spruck \cite{guan&spruck:chp} and Trudinger and Wang \cite{trudinger&wang}
that if a curve bounds a surface of positive curvature, then it bounds a surface of constant positive curvature, see also Smith \cite{smith2012}.  Further we should point out that there exists a disk of everywhere zero curvature smoothly embedded in $\R^3$ whose torsion changes sign only twice, and  has only two vertices, see Example \ref{ex:dumbbell2} and the paper of R{\o}gen \cite{rogen:flat}. Thus, in Corollary \ref{cor:main}, the assumption that the curvature be positive along $\p M$ is essential. The requirement that the curvature be nonnegative in the interior of $M$ is necessary as well, see Example \ref{ex:negative}. Finally note that the $\C^3$ assumption on $f$   was used in the above proof only to ensure local convexity; therefore, if $f$ has everywhere positive curvature,  then $\C^2$ regularity in the interior of $M$ is sufficient.

The rest of this paper is devoted to the proof of Theorem \ref{thm:main}, although our methods yield a substantial refinement of it (Theorem \ref{thm:main2}). The main inspiration here is a proof of the classical four vertex theorem for planar curves which dates back to H. Kneser \cite{hkneser} in 1922, and Bose \cite{bose} in 1932. As we will describe in Section \ref{subsec:nested}, this proof involves studying circles of largest radius which may be inscribed at each point of a simple closed planar curve. A key idea here, as has been pointed out by Umehara \cite{umehara2}, is to focus on the purely topological nature of the partition which these inscribed circles induce on the curve.
As we will outline in Section \ref{subsec:outline}, this argument  may be adapted to the setting of Theorem \ref{thm:main}, where the role of  inscribed circles will be played by  (convex) caps in our locally convex surface. Accordingly, we will embark on an extensive study of these caps and develop a number of their fundamental properties  needed in this work, which may also be of independent interest.

Convex caps play a major role in the seminal works of Alexandrov  \cite{alexandrov:intrinsic}    and Pogorelov \cite{pogorelov:book} on the isometric embeddings of convex surfaces, as well as in other fundamental results in this area such as the works by van Heijenoort \cite{vH:convex}, Sacksteder \cite{sacksteder:convex}, and Volkov \cite[Sec. 12.1]{alexandrov:polyhedra}. In these studies, however, the underlying Riemannian manifolds are assumed to be complete, or nearly complete, as in the works of Greene and Wu \cite{greene&wu:rigidity,greene&wu:rigidityII}. In the present work, on the other hand, we study caps of manifolds with boundary, as in the author's previous work with Alexander \cite{alexander&ghomi:chp} and Alexander and Wong \cite{agw}. Thus we will be prompted to refine a number of old results in this area, and establish some new ones, such as Theorem \ref{thm:singularcap} and the corresponding Bose formula, which is a culmination of other results on uniqueness, extension, distortion, and convergence of caps developed below.

The study of special points of curvature and torsion of closed curves has generated a vast and multifaceted literature since the
works  of  Mukhopadhyaya \cite{mukhopadhyaya} and A. Kneser \cite{kneser} on vertices of planar curves were published in 1910--1912, although aspects of these investigations  may be traced even further back to the study of inflections by  M\"{o}bius \cite{mobius} and Klein \cite{klein}, see \cite{ghomi:verticesB, ghomi:verticesC}. The first version of the four vertex theorem for space curves, which was concerned with curves lying on smooth strictly convex surfaces, was stated by Mohrmann \cite{mohrmann} in 1917, and proved by Barner and Flohr \cite{barner&flohr} in 1958. This result was  finally extended to curves lying on the boundary of any convex body by Sedykh \cite{sedykh:vertex} in 1994, after partial results by other authors \cite{ballesteros&fuster, bisztriczky}, see also Romero-Fuster and Sedykh \cite{fuster&sedykh} for further refinements. 

Among various applications of four vertex theorems, we mention a paper of Berger and Calabi et al.\@ \cite{rmbc} on physics of floating bodies, and recent work of Bray and Jauregui \cite{bray&jauregui} in general relativity. See also the works of Arnold \cite{arnold:plane, arnold:sphere} for relations with contact geometry, the book of Ovsienko and Tabachnikov \cite{ovsienko&tabachnikov} for  projective geometric aspects,
Angenent \cite{angenent:inflection} for connections with mean curvature flow, which are also discussed in \cite{ghomi:verticesC}, and Ivanisvili et al.\@ \cite{ivanisvili1, ivanisvili2} for applications  to the study of Bellman functions.
 Other references and more background on four vertex theorems may be found in \cite{gluck:notices, thorbergsson&umehara, ghomi:verticesA, pak:book, ghys&tabachnikov}.

\section{Preliminaries}\label{sec:prelim}

As has been pointed out by Umehara \cite{umehara2}, and further studied by Umehara and Thorbergsson \cite{thorbergsson&umehara}, there is a purely  topological partition of the circle $\S^1$, called an \emph{intrinsic circle system}, which plays a fundamental role in a variety of four vertex theorems in geometry. Here we review how a weaker notion, which we call a \emph{nested partition}, quickly  leads to various generalizations of the classical four vertex theorem for planar curves, and then outline how this argument may be adapted to proving theorem \ref{thm:main}, which  will then motivate the rest of this work. 

\subsection{Bose's formula for nested partitions of $\S^1$}\label{subsec:nested}

A \emph{partition} $\P$  of $\S^1$ is a collection of its nonempty subsets, or \emph{parts} $P$,  which cover it and are pairwise disjoint. We say that $\P$ is \emph{nontrivial} if it contains more than one part. For each  $p\in \S^1$, let $[p]$ denote the  part of  $\P$ which contains $p$. Further, for each part $P\in\P$, let $P':=\S^1-P$ denote its complement.
We say that  $\P$  is \emph{nested} provided that for every $P\in \P$ and $p\in P'$, $[p]$ lies   in a (connected) component of $P'$.  This definition is motivated by the following example: 

\begin{example}\label{ex:partition}
Let $M$ be a topological disk, $f\colon M\to\S^2$ be a locally one-to-one continuous map into the sphere, and suppose that $\f$ is  $\C^2$. Let us say that $C\subset M$ is a \emph{circle} provided that $f(C)$ is a circle, and $f$ is injective on $C$. Since $\f$ is $\C^2$, there exists a circle passing through each point $p\in \p M$. Further, one of these circles, which we denote by $C_p$, is \emph{maximal}, i.e., it is not contained inside any other circle passing through $p$. Existence of $C_p$ follows quickly from Blaschke's selection principle \cite[Thm. 1.8.6]{schneider:book} when $f$ is one-to-one, or $f(\p M)$ is a simple curve; the general case, which is more subtle, follows for instance from Proposition \ref{prop:converge} below.
Now set 
$$
[p]:=C_p\cap \p M.
$$ 
Then $\P:=\{[p]\}_{p\in\p M}$ is a  partition of $\p M\simeq\S^1$, since maximal circles are unique. We claim that $\P$ is nested. 
 \begin{figure}[h]
\begin{overpic}[height=1.0in]{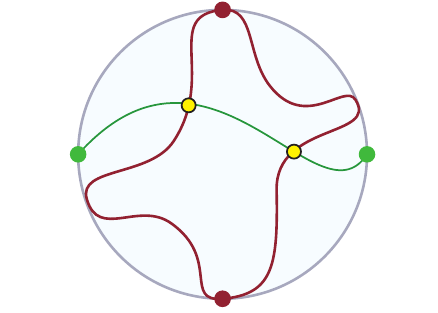}
\put(78,10){\Small $M$}
\put(42,19){\Small $C$}
\put(8,35){\Small $p'$}
\put(88,35){\Small $q'$}
\put(48,-4){\Small $q$}
\put(48,74){\Small $p$}
\end{overpic}
\caption{}
\label{fig:circles}
\end{figure}
Suppose, towards a contradiction, that it is not. Then there is a pair $C$, $C'$ of distinct maximal circles in $M$  such that $C'$ has points $p'$, $q'$ in different components of $\p M-C$, see Figure \ref{fig:circles}. 
Further there are points $p$, $q\in C\cap\p M$ which separate $p'$ and $q'$ in $\p M$. Consequently, each component of $C'-\{p',q'\}$ must intersects both components of $C-\{p,q\}$, since $M$ is simply connected. Thus $C\cap C'$ and consequently $f(C)\cap f(C')$ contain at least $4$ points. Since $f(C)$ and $f(C')$ are circles, it follows then that $f(C)=f(C')$. Thus $C=C'$ which is the desired contradiction.
\end{example}

A part $P\in\mathcal{P}$ has \emph{rank} $n$ provided   that it has a total of $n$ components. If $\rank(P)=1$, then we say that $P$ is \emph{singular}, and if $\rank(P)\geq 3$, then  $P$ is called \emph{triangular}. Our first lemma below gives a lower bound for the number of singular parts of $\P$ in terms of the number of triangular ones. To describe this estimate, let 
$$
\mathrm{Tri}(\P):=\{P\in\P\mid \rank(P)\geq 3\},
$$
denote the set of triangular parts of $\P$. Obviously, one way to measure the size of $\mathrm{Tri}(\P)$ would be by its cardinality, $\#\mathrm{Tri}(\P)$; however, following  Umehara \cite{umehara2}, see also Romero-Fuster and Sedykh \cite{fuster&sedykh}, we employ a more refined measure:
\begin{equation}\label{eq:T}
\mathcal{T}:=\sum_{P\in\mathrm{Tri}(\P)}\big(\rank(P)-2\big)\;\geq\; \#\mathrm{Tri}(\P).
\end{equation}
This notion, which we call the \emph{triangularity} of $\P$, goes back  to the remarkable work of Bose \cite{bose}.
 Next let $\mathcal{S}$ denote the cardinality of the set of singular parts of $\P$:
\begin{equation}\label{eq:S}
\mathcal{S}:=\# \{P\in\P\mid \rank(P)=1\}.
\end{equation}
The first statement in the following lemma, and its proof, is a close variation on \cite[Lem.\@ 1.1]{umehara2}, which according to Umehara \cite{umehara2} dates back to H. Kneser's work \cite{hkneser}. Further the inequality \eqref{eq:bose} below is a ``Bose type formula", which dates back to  \cite{bose}, and is similar to a formula proved in
\cite[Thm.\@ 2.7]{umehara2}, see Note \ref{note:umehara} below.

\begin{lem}\label{lem:partition}
Let $\P$ be a 
 nested nontrivial partition of $\S^1$ and $P\in \P$. Then  each component of $P'$ contains a singular part of $\P$. In particular,
\begin{equation}\label{eq:bose}
 \mathcal{S}\geq\mathcal{T}+2.
 \end{equation}
 \end{lem}
\begin{proof}
Since $\P$ is nontrivial, $P'\neq\emptyset$. Let
 $I_0$ be a component of $P'$, of  length $\ell_0$. Let $m_0$ be the midpoint of $I_0$. If $[m_0]$ is connected we are done, since $[m_0]\subset I_0$ by the nested property of $\P$. Otherwise, $I_0-[m_0]$ has a component $I_1$ of length $\ell_1\leq \ell_0/2$ such that $\cl (I_1)\subset I_0$, where $\cl (I_1)$ denotes the closure of $I_1$ in $\S^1$. Let $m_1$ be the midpoint of $I_1$. Again if $[m_1]$ is connected, we are done, since $[m_1]\subset I_1\subset I_0$. Otherwise  $I_1-[m_1]$ contains a component $I_2$ of length $\ell_2\leq \ell_1/2\leq\ell_0/4$ such that $\cl (I_2)\subset I_1$. Continuing this procure inductively, we either reach a stage where $[m_k]$ is connected for some $k$, in which case we are done; or else we produce an infinite sequence of closed nested intervals $\cl (I_1)\supset \cl (I_2)\supset\dots$ such that $I_k$ has length $\ell_k\leq \ell_0/(2^k)$. Then $I_k$ converge to a point $m_\infty\in\cl (I_1)\subset I_0$. 
 By the nested property, $[m_\infty]\subset I_k$ for all $k$. Consequently $[m_\infty]=\{m_\infty\}$. In particular $[m_\infty]$ is connected, which completes the proof of the first statement of the lemma.

Next, to establish \eqref{eq:bose}, first suppose that $\mathcal{T}=0$, or $\mathrm{Tri}(\P)=\emptyset$. Then we have to check that 
$\mathcal{S}\geq 2$. This follows immediately from the previous paragraph if $\P$ has a part with rank $2$.
Otherwise every part of $\P$ is singular. Thus again  $\mathcal{S}\geq 2$, since $\P$ is nontrivial by  assumption. 
So we may suppose that $\mathrm{Tri}(\P)\neq\emptyset$. Now
let  $A\subset\mathrm{Tri}(\P)$ be an arbitrary subset of finite cardinality $|A|$, and set 
$$
\mathcal{T}_A:=\sum_{a\in A}\big(\rank(a)-2\big).
$$
 It suffices to show  that 
 $$
 \mathcal{S}\geq \mathcal{T}_A+2.
 $$
  Let us say that a component of $A':=\S^1-A$ is \emph{prime}, provided that it coincides with a component of 
 $a':=\S^1-\{a\}$ for some   $a\in A$. Let $\mathcal{S}_A$ denote the number of prime components of $A'$. Then
 $$
 \mathcal{S}\geq \mathcal{S}_A,
 $$
 because each prime component of $A'$ contains a singular element of $\P$ by the previous paragraph. Thus it suffices to show that
 \begin{equation}\label{eq:partition}
 \mathcal{S}_A= \mathcal{T}_A+2, 
\end{equation}
 which will be done by induction on $|A|$.
  First note that when $|A|=1$, or $A=\{a\}$,  every component of $A'$ is a component of $a'$ and thus is prime. So $\mathcal{S}_A=\rank(a)=\mathcal{T}_A+2$, as desired.

Next, to perform the inductive step on \eqref{eq:partition}, we need the following observation.
Let us say that an element $a\in A$ is \emph{prime} provided that with at most one exception each component of $a'$ is a component of $A'$. 
We claim that $A$ always contains at least one prime element. To see this let $a_1$ be any element of $A$. If $a_1$ is prime, we are done. Otherwise let $a_2\in A$ be an element which lies in a component of $a_1'$.  Continuing this process inductively, if $a_{i-1}$ is  not prime, we let $a_{i}\in A$ be an element in a component of $a_{i-1}'$ which does not contain $a_{i-2}$. Since $\P$ is nested, $a_1,\dots, a_{i}$ are all distinct. Thus, since $|A|$ is finite,  this chain terminates after $k\leq |A|$ steps. Then $a_k$ is the desired prime element.

  Now assume that \eqref{eq:partition} holds whenever $|A|=n$, and suppose that we have a set $A$ with $|A|=n+1$. 
  Let $a_1$ be a prime element of $A$. Since now $A$ has more than one element, $a_1'$ has $\rank(a_1)-1$ components which are components of $A'$. Thus
  $$
  \mathcal{S}_{A-\{a_1\}}=\mathcal{S}_A-\(\rank(a_1)-1\)+1=\mathcal{S}_A-\rank(a_1)+2.
  $$
  On the other hand, by the inductive hypothesis,
  $$
  \mathcal{S}_{A-\{a_1\}}=\mathcal{T}_{A-\{a_1\}}+2=\sum_{i=2}^n\(\rank(a_i)-2\)+2.
  $$
Setting the right hand sides of the last two displayed expressions equal to each other, we obtain
  $$
  \mathcal{S}_A=\sum_{i=2}^n(\rank(a_i)-2)+\rank(a_1)=\sum_{i=1}^n(\rank(a_i)-2)+2=\mathcal{T}_A+2,
  $$
 as desired.
\end{proof}

\begin{note}\label{note:tree}
Another way to prove  \eqref{eq:partition}, which reveals the deeper combinatorial reason behind the Bose formula \eqref{eq:bose}, is by observing that any finite set $A\subset \mathrm{Tri}(\P)$ gives rise to a tree $T$ as follows. 
Let the vertices of $T$ be the elements of $A$ and prime components of $A'$. We declare each prime component of $A'$ to be \emph{adjacent} to the corresponding element of $A$. Further, we stipulate that 
$a_i\in A$, $i=1$, $2$ are  \emph{adjacent} provided that
there are disjoint connected sets $I_i\subset \S^1$ such that $a_i\subset I_i$, $\p I_i\subset a_i$, and there is no element of $A$ which separates $I_i$, i.e., has points in both components of $\S^1-\{I_1\cup I_2\}$.

Since $A$ is nested, it follows that $T$ is a \emph{tree}, i.e., $T$ is connected, and contains no closed paths. Prime components of $A'$ form the \emph{leaves} of $T$, or the vertices of degree $1$, while  elements  of $A$ are vertices of $T$ with degree  $\geq 3$. Thus \eqref{eq:partition} corresponds to
\begin{equation}\label{eq:bosetree}
\#\mathrm{Leaves}(T)=\sum_{\mathrm{deg}(v)\geq 3}(\mathrm{deg}(v)-2)+2,
\end{equation}
which is a general property of all trees.
To establish this equality let $V$, $|V|$, $|V_i|$, and $|E|$ denote, respectively, the set of vertices, the number of vertices, the number of vertices of degree $i$, and the number of edges of $T$. Then we have
$$
|E|=|V|-1=\sum_{i}|V_i|-1,\quad \text{and} \quad |E|=\frac{1}{2}\sum_{v\in V} \mathrm{deg}(v)=\frac{1}{2}\sum_{i} i|V_i|.
$$
The first equality is a basic characteristic property of trees, and the second one is the ``hand shaking lemma", which holds for all graphs. Setting the right hand sides of these expressions equal to each other yields 
$$
|V_1|=\sum_{i\geq 3}(i-2)|V_i|+2,
$$
which is equivalent to \eqref{eq:bosetree}.
\end{note}

\subsection{Proof of the classical $4$ vertex theorems}
Let $f\colon M\to\S^2$ be a topological immersion of a disk into the sphere, and $\P$ be the partition induced on $\p M$ by the maximal circles of $M$, as described in Example \ref{ex:partition}. We say that a maximal circle $C$ is \emph{singular} if $C\cap\p M$ is connected. Let $\mathcal{S}$ denote the number of singular maximal circles of $M$, and $\mathcal{T}$ be the number of maximal circles  $C$ such that $C\cap\p M$ has at least $3$ components.
As we discussed in Example \ref{ex:partition}, $\P$ is nested. Thus Lemma \ref{lem:partition} immediately yields the following observation, which is a special case of our Theorem \ref{thm:singularcap}.

\begin{cor}\label{cor:classic}
 Suppose that $\p M$ is not a circle, and let $C\subset M$ be a maximal circle.
Then each component of $\p M-C$ contains a point  where the corresponding maximal circle  is singular. In particular  
\begin{equation}\label{eq:boseclassic}
 \mathcal{S}\geq \mathcal{T}+2.
\end{equation}
 \end{cor}

Since the stereographic projection preserves circles, the above observation has a direct analogue for planar curves. An example of this phenomenon is illustrated in Figure \ref{fig:ellipse} which depicts an ellipse with a pair of singular maximal circles on either side of a maximal circle of rank $2$.

\begin{figure}[h]
\begin{overpic}[height=0.9in]{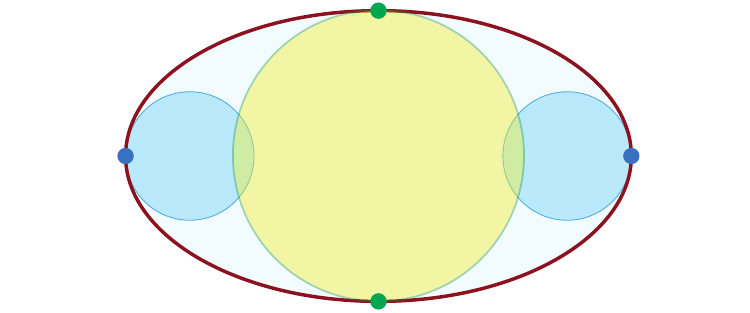}
\end{overpic}
\caption{}
\label{fig:ellipse}
\end{figure}

If $C_p$ is a singular maximal circle of $M$ at $p$,  then it is well-known that $f(C_p)$ is the osculating circle of $\f$ at $p$; this follows for instance by applying Taylor's theorem to a stereographic projection of $\f$ into $\R^2$. Let $\kappa$ denote the geodesic curvature of $\f$ with respect to the normal vector field which points locally into $f(M)$. Then, since $f(C_p)$ lies on one side of $\f$ near $p$, it follows that  $\kappa$  has a local minimum at $p$.
So, since $\kappa$ must have a local maximum between every pair of local minima, we have
\begin{equation}\label{eq:V2S}
\mathcal{V}\geq 2\mathcal{S}\geq 2(\mathcal{T}+2),
\end{equation}
where $\mathcal{V}$ is the number of local extrema of $\kappa$, or vertices of $\f$.  
In particular  $\mathcal{V}\geq 4$, which  yields the 1912 theorem of  A. Kneser \cite{kneser} for simple closed planar curves. Earlier,  in 1909, Mukhopadhyaya \cite{mukhopadhyaya} had established this inequality for convex planar curves. The improved inequality \eqref{eq:V2S} was first established by Bose \cite{bose} for convex  planar curves in 1932, and was extended  to all simple closed planar curves by Haupt \cite{haupt:bose} in 1969. Prior to Bose, H. Kneser \cite{hkneser}, who was A. Kneser's son, had shown that $\mathcal{S}\geq 2$ in 1922. The  above corollary also improves the theorem of Osserman \cite{osserman:vertex} and dual results of Jackson \cite{jackson:bulletin} on the relation between the vertices and the inscribed/circumscribed circles of planar curves. Further, \eqref{eq:V2S} yields the simply connected case of a theorem of Pinkall \cite{pinkall} who in 1987 had obtained the first version of the $4$ vertex theorem for nonsimple curves which bound immersed surfaces in $\S^2$ (or $\R^2$), see also \cite{cairns2,umehara}.

\begin{note}\label{note:umehara}
As we mentioned above, 
in \cite{umehara2} Umehara introduces a partition of $\S^1$, called an \emph{intrinsic circle system}, which is more restrictive than the nested partitions discussed above. An intrinsic circle system is a nested partition where each part is closed, and the parts satisfy a  lower semicontinuity property: if $p_i\in \S^1$  converge to $p$, and $q_i\in[p_i]$ converge to $q$, then $q\in [p]$. Subject to these additional restraints, Umehara shows that equality holds in \eqref{eq:bose}, when $\mathcal{S}<\infty$. See also Note \ref{note:umehara2} below.
\end{note}

\begin{note}
Another way to gain insight into the Bose formula  in Corollary \ref{cor:classic}, is by considering the locus $T$ of the centers of  maximal circles in $M$. It is well-known \cite{mather} that when $\p M$, or more precisely $f\colon \p M\to\S^2$, is generic, $T$ is a graph. 
The leaves of $T$ correspond to centers of singular circles, and, more generally, its vertices of degree $n$ correspond to circles of rank $n$. Further, since $M$ is simply connected,  $T$ is a tree. Thus  the Bose formula \eqref{eq:boseclassic} follows from the basic combinatorial property of trees described by \eqref{eq:bosetree}. Also note that $T$ is the set of singularities of the distance function from $\p M$, which is well defined in any dimension. These sets have been studied by Mather \cite{mather} who showed that generically they have a stratified structure in the sense of Whitney. Further they have been called ``the medial axis", ``Maxwell graph", or ``the central set", in various parts of the literature, and may even arise as the cut-locus of Riemannian manifolds without conjugate points. See the papers of Damon \cite{damon}, Houston and van Manen \cite{houston&vanmanen}, and references therein. 
\end{note}

\subsection{Outline of the proof of the main theorem}\label{subsec:outline}
The proof of Theorem \ref{thm:main}, or more precisely that of its refinement, Theorem \ref{thm:main2} below,  is modeled on the proof of Corollary \ref{cor:classic} and the subsequent inequality \eqref{eq:V2S} described above. Indeed Theorem \ref{thm:main2} generalizes  \eqref{eq:V2S} to all  locally convex  immersions $f\colon M\to\R^3$ satisfying the hypothesis of Theorem \ref{thm:main}. The key idea here is that the role of maximal circles  in the above argument may  be played by \emph{maximal caps} in our  locally convex surface:
$$
\text{circles\; $\longrightarrow$\; caps}.
$$
 Roughly speaking, caps are convex pieces of a locally convex surface which are cut off by a plane, or lie on a plane. For instance every circle in $\S^2$ bounds a cap, but we will also regard single points of $\S^2$ as (degenerate) caps. In general, degenerate caps may  take the form of a line segment or a flat disk as well. The  precise definition of  caps will be presented in Section \ref{subsec:capdef}. 

 \begin{figure}[h]
\begin{overpic}[height=1.1in]{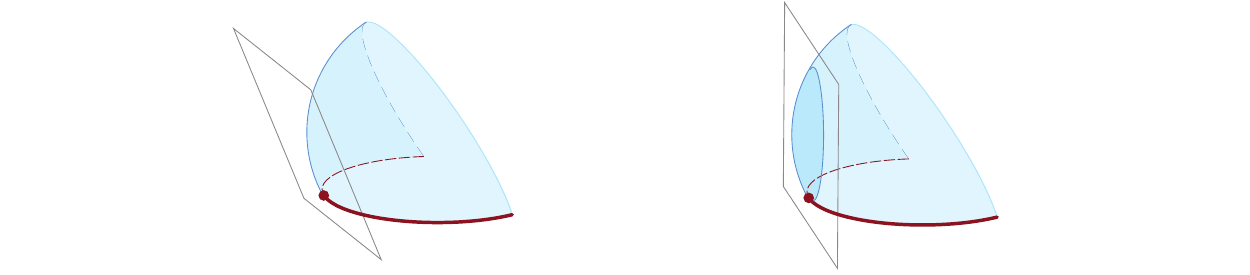}
\end{overpic}
\caption{}
\label{fig:caps}
\end{figure}

Sections \ref{sec:degenerate}, \ref{sec:basic}, and \ref{sec:maxcap} are devoted to proving the existence of a maximal cap at each point $p\in \p M$.
This will be done via a constructive procedure which  is easy to describe  when $f$ is smooth, and has positive curvature near $\p M$:  start with a plane $H$ which is tangent to $f$ at $p$, see Figure \ref{fig:caps}, and rotate it about the tangent line passing through $p$ so as to cut off a small cap from the surface;  continue to rotate $H$ in the same direction as far as possible while the interior of the cut off region remains disjoint from $\p M$. The final cut off region will  be the desired cap.

To carry out the above construction in the general  case we need to 
address a number of technical issues. The first major complication is that  maximal caps may be degenerate. So first we characterize these in Section \ref{sec:degenerate}. Subsequently some basic properties of general caps, including their convergence, may be further developed in Section \ref{sec:basic}, which culminates in a comprehensive structure theorem in Section \ref{sec:slicing}. Next we need to describe what we mean by a ``tangent plane" when $f$ is not regular in the interior of $M$. This will require studying  the tangent cones of $f$ along $\p M$ in Section \ref{sec:tancone}. 
Using this tool, we then show, in Section \ref{sec:boundarycaps}, that there exists a cap at each boundary point of $M$, and subsequently prove the existence of a maximal cap $C_p$ for each point $p\in\p M$ in Section \ref{subsec:maxcap}. 

Once the existence of maximal caps has been established, we let $[p]:=C_p\cap\p M$ be the partition that  these caps induce on
$\p M\simeq\S^1$, just as in Example \ref{ex:partition}. Next we will show in Section \ref{sec:singularcap} that this partition has the nested property,  which  yields the estimate $\mathcal{S}\geq\mathcal{T}+2$ via Lemma \ref{lem:partition}.
It remains then to establish the estimate $\mathcal{V}\geq2\mathcal{S}$, where $\mathcal{V}$ denotes the number of times the torsion $\tau$ of $\f$ changes sign. To this end we will show, in Section \ref{subsec:osculateplane}, that  the plane $H$ of a maximal cap $C_p$ is the osculating plane of $\f$ at $p$. Consequently, since $\f$ lies locally on one side of $H$, it follows that $\tau$  vanishes at $p$. Moreover it will be shown in Section \ref{subsec:orientexpress} that  $\p M$ may be oriented so that $\f$ lies locally \emph{below} $H$, i.e., in the half-space opposite to where the binormal vector $B(p)$ points. This will ensure, as we will show in Section \ref{subsec:torsion}, that $\tau$ changes sign from negative to positive as $\f$ crosses $p$. Thus we obtain the desired inequality $\mathcal{V}\geq2\mathcal{S}$,  which completes the proof.

\subsection{Notation and terminology} Here we gather the frequently used terms in this paper for easy reference. Most importantly, we will define  \emph{caps}, which are the central objects of study in this work. 

\subsubsection{General terms}
$\R^n$ stands for the $n$-dimensional Euclidean space with standard inner product $\l\cdot,\cdot\r$ and corresponding norm $\|\cdot\|$. Further $\S^{n-1}$ denotes the unit sphere centered at the origin of $\R^n$. The abbreviations  \emph{int} and \emph{cl} will stand respectively for the interior and closure, and $\p$ will denote the boundary of a manifold or a subspace. A \emph{segment} is a topological space which is homeomorphic to the interval $[0,1]\subset\R$. Finally, for any set $X\subset\R^n$, $\conv(X)$ will denote the \emph{convex hull} of $X$, i.e., the intersection of all closed half-spaces which contain $X$.

\subsubsection{Support planes and vectors}
Hyperplanes of $\R^n$ will be denoted by $H$, and by a \emph{side} of $H$ we shall mean one of the closed half-spaces $H^-$ or $H^+$, which are determined by $H$. When $H$ is oriented, i.e., it comes with a preferred choice of a unit normal $n$, we say that $n$ is the \emph{outward normal} of $H$, and assume $H^+$ is the side of $H$ into which $n$ points. We say that $H$ \emph{supports} a set $A\subset\R^n$, if $H$ passes through a point $p$ of $A$, and $A$ lies on one side of $H$, which unless specified otherwise, we assume to be $H^-$. In this case, $n$ will be called a \emph{suppor vector} of $A$ at $p$. If, furthermore, $H$ intersects $A$ only at $p$ we say that $H$ is a \emph{strictly supporting} plane of $A$ at $p$;  accordingly, $n$ will  be called a \emph{strict support vector} of $A$ at $p$. 

\subsubsection{The local nonflatness assumption}
This refers simply to the assumption in Theorem \ref{thm:main}, and throughout the paper, that no point $p\in\partial M$ has a neighborhood $U$ such that $f(U)$ lies in a plane.

\subsubsection{$f$ and $M$}
Unless noted otherwise, we shall assume that $f$ and $M$ are as in Theorem \ref{thm:main}, i.e., $M$ is a topological disk, $f\colon M\to\R^3$ is a locally convex immersion which satisfies the local nonflatness assumption along $\p M$,
and  $\f$ is $\C^3$ and has no \emph{inflections}, i.e., points where the  curvature vanishes. Most of the  statements which are established below, however,  often hold  in wider contexts. For instance, the $\C^3$ regularity of $\f$ is not needed before Section \ref{subsec:orientexpress}; until that point $\f$ may be $\C^2$, and  often even $\C^1$. Further, the simply connected assumption on $M$ is not required until Section \ref{sec:singularcap}, i.e., $M$ may be any compact connected 2-manifold with boundary up to that point. Finally, the statement and proofs of many of the following results hold nearly word by word in $\R^n$,  but for simplicity we will not distinguish these from other results.

\subsubsection{Convex neighborhoods}\label{subsubsec:convngbhd}
A convex body $K\subset\R^n$ is a compact convex set with interior points in $\R^n$. Recall that  $f\colon M\to\R^3$ is a \emph{locally convex immersion} provided that $M$ is a $2$-dimensional topological manifold, and $f$ is a locally one-to-one continuous map which sends a neighborhood $U$ of each point $p\in M$ into the boundary of a convex body $K\subset\R^3$. Since $f$ is locally-one-to-one, we may assume that $U$ is so small that $f$ is one-to-one on $\cl(U)$ and consequently $f\colon\cl (U)\to f(\cl(U))\subset\partial K$ is a homeomorphism (any one-to-one map from a compact space into a Hausdorff space is a homeomorphism onto its image).  Further we may assume that $\cl(U)$ is a topological disk and $\partial U$ is a simple closed curve. When $p\in\partial M$, we will also assume that $\partial U\cap\partial M$ is connected, and $K=\conv(f(U))$ by the local nonflatness assumption. When all these conditions hold, we say that $U$ is a \emph{convex neighborhood} of $p$ with \emph{associated body} $K$.

\subsubsection{Caps}\label{subsec:capdef}
We use the term \emph{cap}, which is short for \emph{convex cap}, to refer  both to certain subsets of $\R^3$ as well as corresponding subsets of $M$ which are mapped homeomorphically onto them. Further, in each category we distinguish two varieties of caps: nondegenerate and degenerate, as described below.

A \emph{nondegenerate} cap $C$ in $\R^3$ is a topological disk which lies embedded on the boundary of a convex body $K$, and meets a plane $H$ precisely along its boundary $\partial C$. Then $H$ which will be called the \emph{plane} of $C$. By the Jordan curve theorem, $C$ lies on one side of $H$, which we  designate by $H^+$, and refer to as the \emph{half-space} of $C$. In particular, $C=\cl(\partial K\cap\inte H^+)$. Conversely, if $K\subset\R^3$ is a convex body, $H$ is a plane such that $H\cap K$ has interior points in $H$, and $H^+$ is a side of $H$ which contains an interior point of $K$, then $\cl(\partial K\cap\inte H^+)$ is a nondegenerate cap.

A \emph{degenerate} cap in $\R^3$, on the other hand, is a compact convex subset $C$ of a plane $H$, i.e.,  $C$ is either a point, a line segment, or a convex  disk in $H$. In this case we again say that $H$ is a \emph{plane} of $C$; however, note that this plane  is not unique when $C$ is not a disk.

 A  closed connected set $C\subset M$ will be called a \emph{cap} provided that (i) $f$ maps $C$ injectively onto a cap $f(C)$ in $\R^3$, and (ii) if $f(C)$ is degenerate, then there exists a neighborhood $U$ of $C$ in $M$ such $f(U-C)$ is disjoint from a half-space $H^+$ of $f(C)$. Note that since $M$ is compact, $C$ is compact, and thus $f\colon C\to f(C)$ is a homeomorphism.  So a cap in $M$ is either homeomorphic to a disk, a line segment, or a point. 
 
 We say that a cap $C$ in $M$ is \emph{degenerate} (resp. nondegenerate) when $f(C)$ is degenerate (resp. nondegenerate). When $C$ is nondegenerate, then $H$, $H^+$ will be called respectively the \emph{plane} and  \emph{half-space} of $C$, if they are the plane and the half-space of $f(C)$ respectively. When $C$ is degenerate,  we say that $H$ is a plane of $C$ provided that $H$ is a plane of $f(C)$, and furthermore a half-space $H^+$ of $H$ satisfies condition (ii) in the last paragraph. 
 
 A cap $C\subset M$ is called \emph{maximal} if it is contained in no other cap of $M$, and is \emph{singular} if $C\cap\p M$ is connected. The number of components of $C\cap \p M$ is the \emph{rank} of $C$. The partition which maximal caps induce on $\p M$ will be denoted by $\mathcal{P}$, and the corresponding quantities $\mathcal{S}$, $\mathcal{T}$ will be defined as in \eqref{eq:S} and \eqref{eq:T} respectively.
 
 \subsubsection{The uniqueness and extension properties}
 When $C$ is a nondegenerate cap in $M$,  condition (ii) above is met automatically (Proposition \ref{prop:capextend}), which will be referred to as the \emph{extension property}. Another basic, but important, fact is that intersecting caps in $M$ which have a common plane coincide (Proposition \ref{prop:capunique}), which will be called the \emph{uniqueness property}. Since these properties are invoked often, we will not always cite the respective propositions.

\subsubsection{Other terms} The components of the Frenet-Serret frame of a curve in $\R^3$ will be denoted by $T$, $N$, $B$ which stand for the \emph{tangent}, \emph{principal normal}, and \emph{binormal} vectors respectively. Further $\nu$ will denote the \emph{conormal} vector along $f|_{\partial M}$ as defined in Section \ref{sec:tancone}. The \emph{curvature} and \emph{torsion} of a curve will be denoted by $\kappa$ and $\tau$ respectively, and $\mathcal{V}$ will be \emph{the number of sign changes} of $\tau$. For any set $X\subset\R^n$ and point $p\in X$, $T_p X$ will denote the \emph{tangent cone} of $X$ at $p$. Finally the \emph{tangent plane} of $f$ at $p$ (Section \ref{sec:tancone}) will be denoted by $H_p(0)$, and its outward normal will be denoted by $n$ or $n_p(0)$.

\section{Structure of Degenerate Caps}\label{sec:degenerate}
The main result of this section is Theorem \ref{thm:flatcap} below which gives a  characterization for degenerate caps,  and  will be eventually subsumed under Theorem \ref{thm:caps}. To obtain this result, first we need to record a local characterization for degenerate caps, which follows from a classical theorem of Tietze-Nakajima. 

\subsection{Local characterization}
Let $A\subset M$.
We say that \emph{$f$ is locally convex on $A$ relative to a plane $H$} provided that for every point $p$ of $A$ there exists a convex neighborhood $U$ of $p$ in $M$ such that $f(U)$ is a convex subset of $H$. The next lemma is a version of the Tietze-Nakajima characterization for convex sets \cite{valentine:book}, which has been extended in several directions \cite{SackstederValentine, bjorndahl&karshon,BirteaOrtegaRatiu}. In particular, this lemma follows immediately from \cite[Thm. 15]{bjorndahl&karshon}. On the other hand, as has been pointed out in the proofs of \cite[Lem. 3.4]{alexander&ghomi:chp} or \cite[Lem. 1]{sacksteder:convex}, the proof of the original Tietze-Nakajima theorem, as presented for instance in \cite[Thm 4.4 ]{valentine:book}, may be quickly adapted to the present setting.

\begin{lem}[Tietze-Nakajima \cite{bjorndahl&karshon}]\label{lem:TN}
Let $A\subset M$ be a closed connected set. Suppose that $f$ is locally convex on $A$ relative to a plane $H$. Then $f$ is one-to-one on $A$, and $f(A)$ is a convex subset of $H$.
\end{lem}

Next we adapt the above lemma to the precise form  we need here: 

\begin{prop}\label{lem:3conditions}
Let $A\subset M$ be a closed connected set, and suppose that $f(A)$ lies in a plane $H$. Further suppose that each point of $A$ has a convex neighborhood $U$  in $M$ such that 
\begin{enumerate}
\item[(i)]{$f(U\cap A)$ is convex,}
\item[(ii)]{$f(U-A)$ is disjoint from $H$,}
\item[(iii)]{$f(U)$ lies on one side of $H$.}
\end{enumerate}
 Then $A$ is a cap with plane $H$.
\end{prop}
\begin{proof} 
By condition (i), $f$ is locally convex on $A$ relative to $H$. Thus, by Lemma \ref{lem:TN}, $f$ embeds $A$ into a convex set in $H$. It remains to show then that there is a neighborhood $V$ of $A$ in $M$ such that $f(V-A)$ is disjoint from $H$,  and lies on one side of it. Let $V$ be the union of all neighborhoods $U$  in the statement of the proposition. We claim that $V$ is the desired neighborhood.

Condition (ii) quickly yields that $f(V-A)\cap H=\emptyset$, as desired. To show that $f(V-A)$ lies on one side of $H$, it suffices to check that the neighborhoods $U$  all lie  either in $H^+$ or in $H^-$.
Further, we only need  to check this for those  $U$ which intersect $\p A$, for if $U\subset A$, then $f(U)\subset H\subset H^\pm$. So suppose  that there is a point $q\in\p A$ such that $f(U)\subset H^-$ for some  $U$ containing $q$. We claim then that,  for every $U$ intersecting $\p A$, 
\begin{equation}\label{eq:3conditions1}
f(U)\subset H^-.
\end{equation}
To establish this claim, which would finish the proof, first note that by $\p A$ here we mean the subset of $A$ which corresponds to the topological boundary of the convex set $f(A)\subset H$, under the homeomorphism $f\colon A\to f(A)$. Thus $\p A$ is either homeomorphic to a point, a line segment, or a circle. 
In particular, $\partial A$ is connected. 

 Let $\p A^\pm$ be the set of points $q\in \p A$ such that for some $U$ containing $q$, $f(U)\subset H^\pm$ respectively. Then $\p A^\pm$ cover $\p A$ by condition (iii). Further it is clear that $\p A^\pm$ are  open in $\p A$. Next we claim that 
\begin{equation}\label{eq:pApm}
 \p A^+\cap\p A^-=\emptyset.
\end{equation}
  This would show that $\p A^\pm$ are closed. Since by assumption $\p A^-\neq\emptyset$, and $\p A$ is connected, it would follow then  that $\p A=\p A^-$, which would in turn yield  \eqref{eq:3conditions1}, as desired.
To establish \eqref{eq:pApm},  suppose towards a contradiction that there is a point $q\in \p A^+\cap \p A^-$. Then  there are neighborhoods $U^\pm$ of $q$ in $M$ such that $f(U^\pm)\subset H^\pm$ respectively. 
Set $U:=U^+\cap U^-$. 
 Then $f(U)\subset H^+\cap H^-=H$. Consequently  $U\subset A$, since $U\subset V$ and $f(V-A)\cap H=\emptyset$. 
 It follows then that $q\in\partial M$, because $q\in\p A$, and $U$ is a neighborhood of $q$ in $M$. 
 This violates the local nonflatness assumption, and yields the desired contradiction.
 \end{proof}

\subsection{Global characterization}
Here we use the local characterization (Proposition \ref{lem:3conditions}) developed in the last section to obtain the following global result on the structure of degenerate caps.

\begin{thm}\label{thm:flatcap}
Let $p\in\inte(M)$,   $H$ be a local support plane of $f$ at $p$, and $A$ be the component of $f^{-1}(H)$ which contains $p$. Then $A$ is a cap with plane $H$.
\end{thm}

This observation generalizes \cite[Lem. 3.4]{alexander&ghomi:chp}, where it had been assumed that  $A\cap\partial M=\emptyset$, and thus was relatively simple to treat. The general case, however,  is considerably more subtle, as we will explore below.
It is also worth noting that the assumption here that $p\in\inte(M)$ is necessary. For instance let $M$ be the upper hemisphere of $\S^2$, $f$ be the inclusion map into $\R^3$, $H$ be the $xy$-plane, and $p:=(1,0,0)$. Then $f^{-1}(H)=M\cap H$ is a circle,  not a cap. 
The basic strategy for proving the above theorem is to show that $A$ satisfies the local conditions in Proposition \ref{lem:3conditions}; however, this is not directly feasible for points of $A\cap\p M$. Thus we take another route via the following simple observation:

\begin{lem}\label{lem:subset}
Suppose that $A_0\subset A$ is a cap with plane $H$. Then $A_0=A$. 
\end{lem}
\begin{proof}
Since $f(A_0)\subset f(A)\subset H$, $A_0$ is degenerate. So there exists a neighborhood $U$ of $A_0$ in $M$ such that $f(U-A_0)\cap H=\emptyset$, by the definition of degenerate caps. Consequently $(U-A_0)\cap A=\emptyset$. Thus $A\cap A_0=A\cap U$ is open in $A$. But $A\cap A_0$ is also closed in $A$, since $A$ and $A_0$ are both closed. Thus $A\cap A_0=A$, since $A\cap A_0\neq\emptyset$ and $A$ is connected. So $A\subset A_0$, which completes the proof.
\end{proof}

So, to prove Theorem \ref{thm:flatcap}, it suffices to show that there is some subset of $A$ which is a cap with plane $H$. We choose this set to be
$$
A_0:= \text{the component of }\cl(A\cap\inte(M))\text{ which contains } p.
$$
Note that $A_0\neq\emptyset$ since $p\in A_0$. Further, a particularly useful feature of this set is that it satisfies condition (iii) of Proposition \ref{lem:3conditions} fairly quickly, via the theorem on the invariance of domain:

\begin{lem}\label{lem:HA}
$H$ is a local support plane of $f$ at all points of $A_0$.
\end{lem}
\begin{proof}
Let $X\subset A$ consist of all points $q\in A$ such that $H$ locally supports $f$ at $q$. Then $X$ is open in $A$. Indeed, if $q\in X$, then there is a neighborhood $U$ of $q$ in $M$ such that $f(U)$ lies on one side of $H$. So $U\cap A\subset X$.

 We have to show that $A_0\subset X$. First note that  $X\cap A_0\neq\emptyset$, since $p\in X\cap A_0$. It is also immediate that $X\cap A_0$ is open in $A_0$, since $X$ is open in $A$, and $A_0\subset A$. It remains then to show that $X\cap A_0$ is closed in $A_0$, which will finish the proof, since $A_0$ is connected. 
To  this end, let $q_i\in X\cap A_0$ be a sequence   which converges to a point $q\in A_0$:
$$
X\cap A_0\ni \;q_i\longrightarrow q\;\in A_0.
$$
 We have to show that $q\in X$. Let $U$ be a convex neighborhood of $q$ with associated body $K$. Then $q_i\in U$, for a sufficiently large index $i$ which we now fix. 
Since $q_i\in A_0$, there are points of $\inte(M)\cap A$ which converge to $q_i$ and so eventually lie in $U$ as well. Thus there is a point $q'\in \inte(M)\cap A\cap U$. If $q'$ is sufficiently close to $q_i$, then $q'\in X$, since $q_i\in X$, and $X$ is open in $A$. So
$$
q'\in\inte(M)\cap X\cap U.
$$

Since $q'\in X$, there is a neighborhood $V$ of $q'$ in $M$ such that $f(V)$  lies on one side of $H$, and since $q'\in \inte(M)\cap U$, we may assume that $V\subset \inte(M)\cap U$.
 Since $V\subset\inte(M)$,  $f(V)$ is open in $\partial K$, by the invariance of domain. So $H$ is a local support plane of $K$ at $f(q')$. This in turn yields that $H$ is a global support plane of $K$, since $K$ is convex. So $f(U)$ lies on one side of $H$, and thus $q\in X$ as desired.
\end{proof}

 It remains then to show that $A_0$ satisfies conditions (i) and (ii) of Proposition \ref{lem:3conditions} as well. To this end we need the following three lemmas.

\begin{lem}\label{lem:localdisk}
Let $K\subset\R^3$ be a convex body, $p\in\p K$,  $B$ be a  (closed) ball of radius $r$ centered at $p$, and  $D:=B\cap\p K$. Then there exists an $\epsilon>0$ such that for all $0<r\leq\epsilon$, $D$ is a topological disk, and $\p B\cap D=\p D$.
\end{lem}
\begin{proof}
We may assume that $p=o$, the origin of $\R^3$,  there exists a  disk $\Omega$ of radius $\epsilon$ centered at $o$ in $\R^2$, and a nonnegative convex function $f\colon\Omega\to\R$, with $f(o)=0$, such that $f(\Omega)\subset\p K$. 
Then the upper hemisphere  of $\p B$ is the graph of a concave function $g\colon\Omega\to\R$. Let  $\Omega'\subset\Omega$ be the set  where $h:=g-f\geq 0$. Then $D$ is a graph over $\Omega'$ and so is homeomorphic to it.  It suffices then to show  that $\Omega'$ is convex, since $h(o)=1-0>0$, which means that $\Omega'$ has interior points in the $xy$-plane, and thus is a disk.  Let $x$, $y\in \Omega'$ and set $z:=\lambda x+(1-\lambda)y$, for $\lambda\in[0,1]$. Then
$$
f(z)\leq \lambda f(x)+(1-\lambda) f(y),\quad\quad\text{and}\quad\quad g(z)\geq \lambda g(x)+(1-\lambda)g(y),
$$
since $f$ is convex $g$ is concave. Consequently, 
since $h(x)$, $h(y)\geq 0$, 
$$
h(z)=g(z)-f(z)\geq\lambda h(x)+(1-\lambda)h(y)\geq 0.
$$
Hence $z\in \Omega'$, which means that $\Omega'$ is convex as desired. To show that $\p D=\p B\cap D$, note that $h\equiv 0$ on $\p \Omega'$. So $\p D\subset\p B$. It remains to check that $\inte(D)\cap\p B=\emptyset$, or that $h\neq 0$  on $\inte(\Omega')$. To see this let $z\in\inte(\Omega')$. Then $z=\lambda o+(1-\lambda) y$, for some $y\in\partial \Omega'$, and $0<\lambda<1$. Since $h(o)=1$ and $h(y)=0$, it follows that  $h(z)\geq \lambda>0$ which completes the proof.
\end{proof}

Another basic fact which we need to record is the following elementary observation. Again we include the proof for completeness.

\begin{lem}\label{lem:Gamma-L}
Let $\Gamma\subset\R^3$ be a $\C^2$ embedded curve segment, $p\in\inte(\Gamma)$, and $L$ be the tangent line of $\Gamma$ at $p$. Suppose that the curvature of $\Gamma$ at $p$ does not vanish. Then there exists a neighborhood $U$ of $p$ in $\Gamma$ which intersects $L$ only at $p$.
\end{lem}
\begin{proof}
Let $\gamma\colon(-\delta,\delta)\to\Gamma$ be a local unit speed parametrization with $\gamma(0)=p$, and set $h(t):=\l\gamma(t)-\gamma(0),N(0)\r$, where $N:=\gamma''/\|\gamma''\|$ is the principal normal of $\gamma$. Then $\gamma(t)\not\in L$  if $h(t)\neq 0$. By 
the Frenet-Serret formulas, and Taylor's theorem, $h(t)=s^2\kappa(s)/2$ for some $s=s(t)\in (0,t)$, where  $\kappa:=\|\gamma''\|$ is the curvature of $\gamma$ (e.g. see \cite[p. 31]{spivak:v2}). Thus, since $\kappa(0)\neq 0$, $h>0$ on $(-\epsilon,\epsilon)-\{0\}$ for some $0<\epsilon<\delta$. Consequently  $U:=\gamma (-\epsilon, \epsilon)$ is the desired neighborhood.
\end{proof}

Using the last two lemmas, we  establish next the main observation we need in order to prove Theorem \ref{thm:flatcap}. Again the interior points will be relatively easy to treat here, due to the invariance of domain, but boundary points will require more work.

\begin{lem}\label{lem:U-A}
Every point  of $A_0$ has a convex neighborhood $U\subset M$ such that $f(U)\cap H$ is convex.
 \end{lem}
\begin{proof}
Let $q\in A_0$. There are two cases to consider:  $q\in\inte(M)$, and $q\in\p M$. 

\emph{(Case 1)}
First suppose that $q\in\inte(M)$. Then there exists a convex neighborhood $V$ of $q$ with associated body $K$, such that $V\subset\inte(M)$. So $f(V)$ is open in $\partial K$, by the invariance of domain. Further, by Lemma \ref{lem:HA}, $f(V)$ lies on one side of $H$. Thus $H$ is a local support plane of $K$ at $f(q)$, which in turn yields that $H$ is a (global) support plane of $K$, since $K$ is convex. 
Let $B$ be a ball centered at $f(q)$ and set  $D:=\p K\cap B$. Choosing  $B$ sufficiently small, we may
assume that $D$ is a (topological) disk which meets $\p B$ only along its own boundary $\p D$, by Lemma \ref{lem:localdisk}. 
Further, we may assume that $f(\partial V)$ lies outside $B$. Then we claim that $U:=f^{-1}(\inte(B))\cap V$ is the desired neighborhood, that is $f(U)\cap H=\inte(D)\cap H$ is convex. To this end it suffices to check that $H\cap D$ is convex, since $H\cap\inte(D)=H\cap D\cap\inte(B)$, and if any set $X\subset B$ is convex, then so is $X\cap\inte(B)$. 

To show that $H\cap D$ is convex, 
 let $K':=B\cap K$. Note that $D\subset\p K'$, and $\p K'-D$ is an open subset of $\p B$; see Figure \ref{fig:BK}.
 \begin{figure}[h]
\begin{overpic}[height=1.1in]{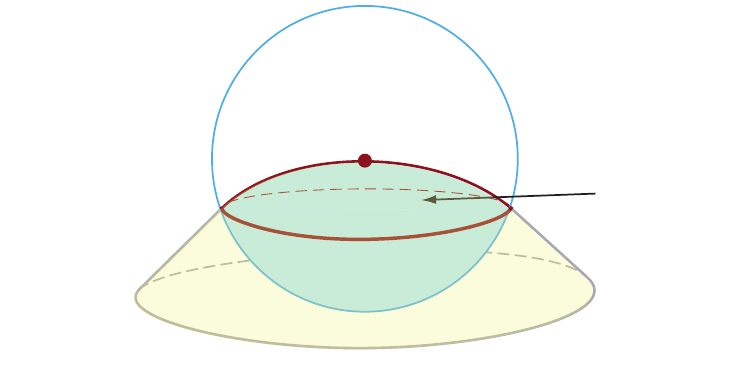}
\put(45,32){\Small$f(q)$}
\put(80,23){\Small$D$}
\put(47,13){\Small$K'$}
\put(24,9){\Small$K$}
\put(27,40){\Small$B$}
\end{overpic}
\caption{}
\label{fig:BK}
\end{figure}
We claim that $H\cap D=H\cap K'$, which is all we need, since $K'$, and therefore $H\cap K'$ is convex. To establish this claim recall that  $H$ is a support plane of $K$ and therefore of $K'$, since $K'\subset K$. So 
$$
H\cap K'=H\cap \p K'=(H\cap D)\cup (H\cap(\p K'-D)).
$$
Note that  if $H$ intersects $\p K' -D$, then $H$ will be a support plane of $\p K'-D$, which is  open in $\p B$. Thus $H$ will be a support plane of $B$. But this is not possible because $f(q)$, which by assumption lies on $H$, is in the interior of $B$. So 
$H\cap(\p K'-D)=\emptyset$. Consequently $H\cap K'=H\cap D$ as claimed. 

\emph{(Case 2)}
It remains to  consider the case where $q\in\p M$.  Again let $V$ be a convex neighborhood of $q$ with associated body $K$. Recall that, since $q\in \p M$, $K=\conv(f(V))$, by definition. Thus, since $f(V)$ lies on one side of $H$,  $K$ lies on one side of $H$ as well.   Let $B$ be a small ball centered at $f(q)$,  and set $D:=B\cap\p K$ as in the previous case. 
Since $\p V$ coincides with a segment of $\p M$ near $q$, and $\f$ is $\C^1$, we may assume that $\p B$ and consequently $\p D$ intersect  $f(\p V)$ at only two points. So $\Gamma:=f(\p V)\cap D$ determines a pair of regions in $D$, one of which is $D':=D\cap f(V)$, see Figure \ref{fig:DandD}. Once again we claim that $U:=f^{-1}(\inte (B))\cap V$ is the desired neighborhood, that is $H\cap f(U)=H\cap D'\cap\inte(B)$ is  convex. To this end,   as we discussed above, it suffices to check that  $H\cap D'$ is convex, which we claim is the case.

\begin{figure}[h]
\begin{overpic}[height=0.9in]{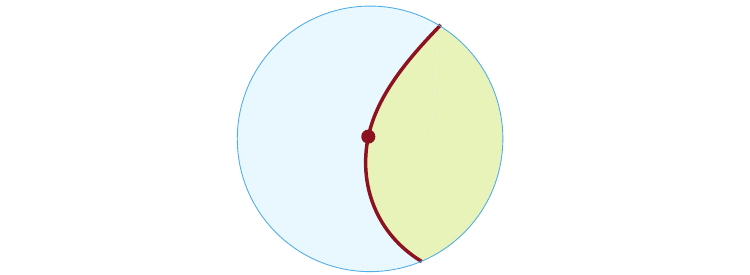}
\put(38,18){\Small$f(q)$}
\put(57,16){\Small$D'$}
\put(26.5,11){\Small$D$}
\put(46.5,4){\Small$\Gamma$}
\end{overpic}
\caption{}
\label{fig:DandD}
\end{figure}

To establish the last claim, first suppose that $H\cap D$ has interior points in $H$. Then we show that $H\cap D'=H\cap D$, which is all we need, since $H\cap D$ is convex as we showed in the previous case. Note that
$H\cap D'\subset H\cap D$, since $D'\subset D$. So it remains to check that $H\cap D\subset H\cap D'$, or more simply that $H\cap D\subset D'$. To this end first note that
if $H\cap D$ has interior points in $H$, 
then it is homeomorphic
  to a disk. Consequently, if $H\cap D$ were to have points in the interiors of both sides of $\Gamma$ in $D$, then the interior of $H\cap D$ would have to intersect $\Gamma$. This would in turn imply that there exists a neighborhood of a point of $\p M$ which is mapped by $f$ into a plane, which would violate the local nonflatness assumption. So $H\cap D$ has to lie on one side of $\Gamma$. To show that this side is $D'$, we just need to check that 
  $H\cap D$ contains an interior point of $D'$, or more simply that $H\cap\inte(D')\neq\emptyset$. This is indeed the case since $q\in U\cap A_0\subset\cl(A\cap\inte(M))$. Thus, since $U$ is open, $U\cap A\cap\inte(M)\neq\emptyset$. Let $q'\in U\cap A\cap\inte(M)=(U-\p M)\cap A$.
  Then $f(q')\in f(U-\partial M)\cap H\subset\inte(D')\cap H$, as desired. 
  
Finally we consider the case where $H\cap D$ has no interior points in $H$. Then $H\cap D$ is either a point or a line segment, since  it is convex. If $H\cap D$ is a point, then so is $H\cap D'$, since $D'\subset D$. In particular $H\cap D'$ is convex  as claimed.
Suppose then that $H\cap D$ is a line segment, say $L$. Then $H\cap D'=H\cap D\cap D'=L\cap D'$. Thus we just need to check that $L\cap D'$ is connected, since a connected subset of a line is convex. 
To see this first note that if $L$ is transversal to $\Gamma$ at $f(q)$, there is a neighborhood of $f(q)$ in $\Gamma$ which intersects $L$ only at $f(q)$. By Lemma \ref{lem:Gamma-L}, the same holds as well when $L$ is tangent to $\Gamma$ at $f(q)$. 
 Thus by choosing $B$ small enough, we can make sure that $L$  meets $\Gamma$ only at $f(q)$. It follows then that the portion of $L$ lying in $D'$ consists either of all of $L$, or one of the two subsegments of it determined by $f(q)$, or just $f(q)$ itself. In all these cases, $L\cap D'$ will be connected as claimed.
\end{proof}

Now we can prove the main result of this section:
 
\begin{proof}[Proof of Theorem \ref{thm:flatcap}]
By Lemma \ref{lem:subset} it suffices to show that $A_0$ is a cap. To this end, we will check that $A_0$ satisfies the three conditions of Proposition \ref{lem:3conditions}. By Lemma \ref{lem:HA}, $A_0$ already satisfies condition (iii). To  check the other conditions, recall that, by Lemma \ref{lem:U-A}, every point $q\in A_0$ has a convex neighborhood $U$ such that $f(U)\cap H$ is convex. We claim that if $U$ is a such a neighborhood, then
\begin{equation}\label{eq:flatcap1}
f(U)\cap H=f(U\cap A_0).
\end{equation}
If this equality holds, then $f(U\cap A_0)$ is convex, which means that condition (i) is met. Furthermore, $f(U-A_0)=f(U-U\cap A_0)=f(U)-f(U\cap A_0)=f(U)-f(U)\cap H=f(U)-H$. So $f(U-A_0)$ is disjoint from $H$, which means that condition (ii) holds as well. Thus it suffices to establish \eqref{eq:flatcap1}. 

To show that \eqref{eq:flatcap1} holds,
let $y\in f(U)\cap H$. Then there is a point $x\in U\cap f^{-1}(H)$ such that $f(x)=y$. Since $f(U)\cap H$ is connected, then so is $U\cap f^{-1}(H)$. It follows then that $U\cap f^{-1}(H)=U\cap A$, because by definition $A$ is a component of $f^{-1}(H)$; and, furthermore, $U\cap A\neq\emptyset$, since $q\in U\cap A_0\subset U\cap A$.
Thus $x\in U\cap A$, and therefore $y=f(x)\in f(U\cap A)$. So  $f(U)\cap H\subset f(U\cap A)$. As the reverse inclusion is trivial, we conclude that 
\begin{equation}\label{eq:flatcap2}
 f(U)\cap H= f(U\cap A).
 \end{equation}
 In particular, $f(U\cap A)$ is convex, since $f(U)\cap H$ is convex.
 It remains to show that $U\cap A=U\cap A_0$. This equality, together with \eqref{eq:flatcap2} would establish 
 \eqref{eq:flatcap1}, which would in turn complete the proof. So our last claim will be
\begin{equation}\label{eq:flatcap3}
 U\cap A=U\cap A_0.
 \end{equation}
 To see this 
 note  that $U\cap A$ is homeomorphic to either a point,  a disk, or a line segment, since $f(U\cap A)$ is convex. 
  If $U\cap A$ is a point, then so is $A$, which yields  that $A=\{p\}=A_0$ and we are done. So we may assume that $U\cap A$ is either a line segment or  a disk.
 
 If $U\cap A$ is a disk, then each point of it is a limit of points of $A$ in the interior of $M$. 
 The same also holds when $U\cap A$ is a segment, say $L$, because $L\cap \partial M$ cannot have any interior points in $\partial M$, since $f|_{\partial M}$ has no inflections by assumption.
 Thus 
 $$
 U\cap A = U\cap \cl(A\cap\inte(M)).
 $$
In particular  $U\cap \cl(A\cap\inte(M))$ is connected, since $U\cap A$ is connected.  This in turn yields that 
 $$
 U\cap \cl(A\cap\inte(M))=U\cap A_0,
 $$
  since $A_0$ has a point $q$ in $U$, and $A_0$ is a component of $\cl(A\cap\inte(M))$ by definition.
 The last two displayed expressions yield \eqref{eq:flatcap3}, which completes the proof.
 \end{proof}

\section{Structure of General Caps}\label{sec:basic}
Here we develop a number of fundamental properties of caps which are required in this work. These observations culminate in a comprehensive characterization, Theorem \ref{thm:caps}, which subsumes Theorem \ref{thm:flatcap}   proved in the last section. 

\subsection{Extension and uniqueness} First we establish a pair of basic  properties of caps which had been mentioned in Section \ref{subsec:capdef}:

 \begin{prop}[The extension property]\label{prop:capextend}
Let $C\subset M$ be a cap, and $H^+$ be a half-space of $C$. Then there exists a neighborhood $U$ of $C$ in $M$ such that 
$$
f(U-C)\subset \inte(H^-).
$$
 In particular, $C$ is a component of $f^{-1}(H^+)$.
\end{prop}

To establish this property we need the following observation:

\begin{lem}\label{lem:capextend}
Suppose that $C$ is nondegenerate. Let $p\in\p C$, $V$ be a convex neighborhood of $p$ with associated body $K$, and suppose that $V\cap\p C\cap\inte(M)\neq\emptyset$.
Then $K$ intersects the interiors of both sides of $H$.
\end{lem}
\begin{proof}
Let $p\in V\cap\p C\cap\inte(M)$. Suppose, towards a contradiction, that  $K$ lies on one side of $H$. Then $H$ is a local support plane of $f$ at $p$. Let $D$ be the component of $f^{-1}(H)$ which contains $p$. Then $D$ is a cap by Theorem \ref{thm:flatcap}, since $p\in\inte(M)$. Further note that $\p C\subset D$, since $p\in\p C$, $\p C$ is connected, and $\p C\subset f^{-1}(H)$. So, since $\p C$ is homeomorphic to $\S^1$, $D$ cannot be a point or a segment and must therefore be a topological disk.
Further recall that $f(\inte(C))\cap H=\emptyset$ by the definition of nondegenerate caps. So $\inte(C)\cap D=\emptyset$.
This in turn yields that $\partial C\subset\partial D$.  Thus $\partial C=\partial D$ since both of these objects are homeomorphic to $\S^1$.   Consequently $C\cup D$ is a  topological sphere. But this is not possible since $M$ is  connected and $\p M\neq\emptyset$. 
\end{proof}

We can now establish the extension property:

\begin{proof}[Proof of Proposition \ref{prop:capextend}]
We may assume that $C$ is nondegenerate, since otherwise  $U$ exists by definition. It suffices to show that every  $p\in\p C$ has a neighborhood $V$ in $M$ such that $f(V-C)$ is disjoint from  $H^+$, because the union of all these neighborhoods, together with $C$, then yield the desired neighborhood $U$. 

 \begin{figure}[h]
\begin{overpic}[height=0.7in]{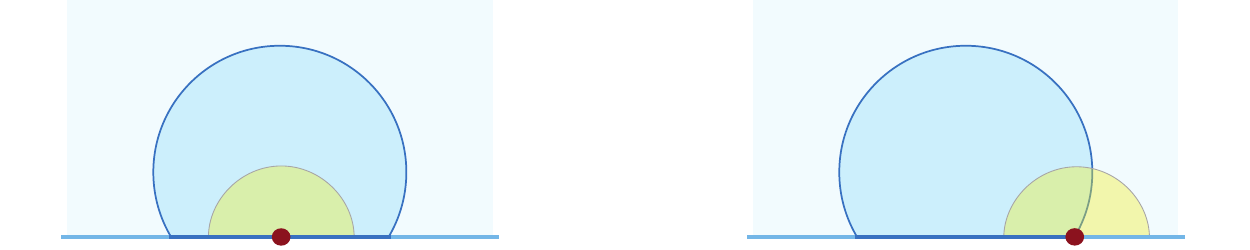}
\put(7,16){\Small$M$}
\put(21.5,-1.5){\Small$p$}
\put(21.5,3){\Small$V$}
\put(21.5,11){\Small$C$}
\put(-2,0){\Small$\p M$}
\put(63,16){\Small$M$}
\put(85,-1.5){\Small$p$}
\put(84.5,3){\Small$V$}
\put(76.5,11){\Small$C$}
\put(53,0){\Small$\p M$}
\end{overpic}
\caption{}
\label{fig:halfdisks}
\end{figure}

Let $V$ be a convex neighborhood of $p$ in $M$ with associated body $K$. If a 
 neighborhood of $p$ in $\p C$ lies in $\p M$, then we may choose $V$ so small that $V\subset C$, see the left diagram in Figure \ref{fig:halfdisks}. 
 Then $f(V-C)\cap H^+=\emptyset\cap H^+=\emptyset$, as desired. So we may assume  that $V\cap\p C\cap\inte(M)\neq\emptyset$, as shown in the right diagram in Figure \ref{fig:halfdisks}. Then $K$ has points in the interiors of both sides  of $H$ by Lemma \ref{lem:capextend}.

Since, by definition, $f$ is one-to-one on $C$,  which is compact, there exists a neighborhood $A$ of $C$ in $M$ such that $f$ is one-to-one on $\cl(A)$. Thus $f\colon A\to f(A)\subset\p K$ is a homeomorphism. For notational convenience  we will identify $A$ with $f(A)$, $C$ with $f(C)$, and suppress $f$ altogether henceforth. Further we may assume that $V\subset A$. What we need to show then is that  
\begin{equation}\label{eq:V-C}
(V-C)\cap H^+=\emptyset.
\end{equation}

Since $K$ intersects the interiors of both sides of $H$, $\partial K\cap H$ is a closed convex planar curve which separates $\partial K$ into a pair of regions including $(\partial K)^+:=\partial K\cap H^+$, see Figure \ref{fig:line}. Let $B$ be a ball centered at $p$. Set  $W:=\inte(B)\cap\partial K$, and $W^+:=W\cap (\partial K)^+$. Then
 \begin{equation}\label{eq:W-W+}
 W-W^+\subset\partial K-(\partial K)^+\subset\inte (H^-).
 \end{equation}
 
 \begin{figure}[h]
\begin{overpic}[height=1.2in]{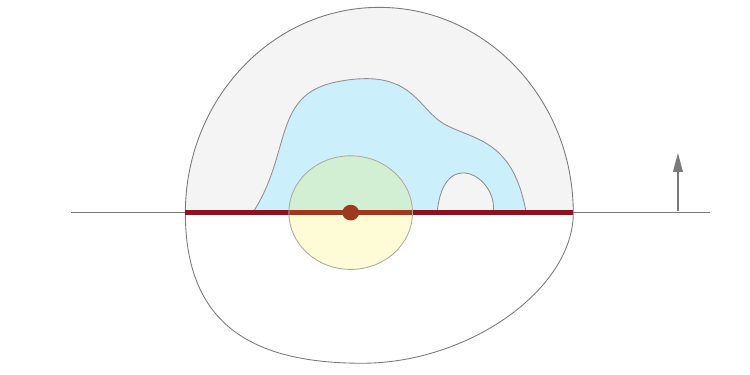}
\put(45,17.5){\Small$p$}
\put(41,32){\Small$V\cap C$}
\put(55,38.5){\Small$\p K^+$}
\put(5,19){\Small$H$}
\put(44,23.5){\SMALL$W^+$}
\put(35,12){\Small$W$}
\put(66,3){\Small$\partial K$}
\put(81,24){\Small$H^+$}
\end{overpic}
\caption{}
\label{fig:line}
\end{figure}

 Note that $V\cap C$ is a neighborhood of $p$ in $(\partial K)^+$; because $V\cap C$ is a neighborhood of $p$ in $C$, $V\cap C\subset\partial K\cap H^+=(\partial K)^+$, and $V\cap\p C\subset\partial K\cap H$. Thus choosing $B$ sufficiently small, we can make sure that $W^+\subset V\cap C$. So $W^+\subset W\cap V\cap C$, since $W^+\subset W$. Conversely, $W\cap V\cap C\subset W\cap C\subset W\cap H^+=W^+$. Thus 
 \begin{equation}\label{eq:WcapVcapC}
 W^+=W\cap V\cap C.
\end{equation}
 Next note that
 $
 W\cap V=\inte(B)\cap\partial K\cap V=\inte(B)\cap V=\inte(B)\cap A\cap V.
 $
 Thus $W\cap V$ is a neighborhood of $p$ in $A$, since $\inte(B)\cap A$ and $V$ are both  neighborhoods of $p$ in $A$. So we may reset $V:=W\cap V$. This together with \eqref{eq:W-W+} and \eqref{eq:WcapVcapC} yields
 $$
 V-C=V-V\cap C\subset W-V\cap C=W-W\cap V\cap C=W-W^+\subset\inte(H^-),
 $$
which in turn immediately yields \eqref{eq:V-C} and thereby completes the proof.
  \end{proof}

Using the extension property established above, we next prove the other property of caps which is frequently used in this work:
 \begin{prop}[The uniqueness property]\label{prop:capunique}
 Let $C_1$,  $C_2\subset M$ be a pair of caps. Suppose that $C_1$, $C_2$ share a plane, and $C_1\cap C_2\neq\emptyset$. Then $C_1=C_2$.
 \end{prop}

\begin{proof}
 Let $H$ be a common plane, and $p$ be a common point of $C_i$, $i=1$, $2$. Further let $H^+$ be the corresponding half-space of $C_1$. If $H^+$ is also a half-space of $C_2$, then, by Proposition \ref{prop:capextend}, $C_i$ are each the component of $f^{-1}(H^+)$ which contains $p$, and thus $C_1=C_2$, as desired. 
 Suppose then, towards a contradiction, that $H^+$ is  not the half-space of $C_2$.   Then $H^-$ must be the half-space of $C_2$. So $C_2$ is the component of $f^{-1}(H^-)$ which contains $p$, again by Proposition \ref{prop:capextend}.

Let $\tilde C_i:= C_i\cap f^{-1}(H)$. Then $\tilde C_i$ is connected. Indeed, if $C_i$ is degenerate, then $\tilde C_i=C_i$, and if $C_i$ is nondegenerate, then $\tilde C_i=\p C_i$. Further, note that $f(p)\in f(C_1\cap C_2)=f(C_1)\cap f(C_2)\subset H^+\cap H^-=H$. Thus $p\in C_i\cap f^{-1}(H)=\tilde C_i$.
So it follows that $\tilde C_1\subset C_2$, since $\tilde C_1$ is  connected, contains $p$, and lies in $f^{-1}(H)\subset f^{-1}(H^-)$, while $C_2$ is the component of $f^{-1}(H^-)$ containing $p$. But we also have $\tilde C_1\subset f^{-1}(H)$. So $\tilde C_1\subset C_2\cap f^{-1}(H)=\tilde C_2$.
By symmetry, we also have $\tilde C_2\subset \tilde C_1$. So  $\tilde C_1=\tilde C_2$.

There are four possible cases to consider. (\emph{Case 1}) If $C_i$ are both degenerate, then $C_1=\tilde C_1=\tilde C_2=C_2$, which is a contradiction, since $C_i$ have different half-spaces.
(\emph{Case 2}) If $C_i$ are both nondegenerate, then $\p C_1=\tilde C_1=\tilde C_2=\p C_2$. On the other hand, $\inte(C_1)\cap\inte (C_2)=\emptyset$, since $f(\inte(C_1))\subset\inte (H^+)$, and $f(\inte(C_2))\subset\inte (H^-)$. Thus $C_1\cup C_2$ is  homeomorphic to $\S^2$, which is not possible, since $M$ is connected and $\p M\neq\emptyset$.
(\emph{Case 3}) If $C_1$ is degenerate and $C_2$ is not, then $C_1=\tilde C_1=\tilde C_2=\p C_2$, which is again a contradiction, since $\p C_2$ is a topological circle, while $C_1$ can only be homeomorphic to a point, an interval, or a disk. 
(\emph{Case 4}) If $C_2$ is degenerate and $C_1$ is not, then, by symmetry, $C_2=\p C_1$, which is not possible, as discussed in the previous case. Thus every case leads to a contradiction as desired.
\end{proof}

 \subsection{Distortion}\label{sec:distortion}

Let $X\subset \R^n$ be a path connected set. For any pair of points $p$, $q\in X$, we may define $d(p,q)$, the intrinsic distance of $p$ and $q$ in $X$, as the infimum of the lengths of all paths in $X$ which join $p$ and $q$. Then the  \emph{distortion} of $X$, in the sense of Gromov \cite[p. 11]{gromov:metric}, is defined as:
$$
\distort (X):=\sup\left\{\frac{d(x,y)}{\|x-y\|}\,\Big|\, x, y\in X, x\neq y\right\}.
$$
See \cite{pardon2,gromov&guth,freedman&krushkal} for recent developments on the distortion.
In this section we obtain an estimate for the distortion of caps $C$ in $\R^n$. Of course if $C$ is degenerate, i.e., it is a convex subset of a hyperplane, then its distortion is $1$, and thus we only need  to consider the case where $C$ is nondegenerate (the definition in Section \ref{subsec:capdef} for nondegenerate caps of $\R^3$ extends directly to $\R^n$). 
Let $o$ be any point in the interior of the region bounded by $\p C$ in $H$.  We define the \emph{inradius} (resp. \emph{outradius}) of $C$ with respect to $o$ as the radius of the largest (resp. smallest) hemisphere in $H^+$ centered at $o$ which is contained in (resp. contains)  $\conv(C)$.

\begin{prop}\label{prop:distort}
Let $C$ be a nondegenerate  cap in $\R^n$, $o$ be a point in the hyperplane  $H$ of $C$ which is contained in the interior of the region bounded by $\p C$, and $r$, $R$ denote, respectively, the inradius and outradius of $C$ with expect to $o$. Then 
$$
\distort(C)\leq 4(2+\pi)\frac{R}{r}.
$$
\end{prop}

To obtain this estimate, we will make use of the following well-known fact. Let $K\subset\R^n$ be a convex body. Then to each point $x\in\R^n-\inte(K)$, there is associated a unique point $\pi(x)\in\p K$ which minimizes the distance between $x$ and points of $\p K$. Thus we obtain a mapping $\pi\colon (\R^n-\inte(K))\to\p K$ which is known as the \emph{nearest point projection}.  Then for every pair of points $x$, $y\in\R^n-\inte(K)$, $\|\pi(x)-\pi(y)\|\leq \|x-y\|$, see \cite[Thm. 1.2.2]{schneider:book}. In particular it follows that for any rectifiable curve segment $\Gamma\subset \R^n-\inte(K)$, $\length(\pi(\Gamma))\leq \length(\Gamma)$. So we may say that $\pi$ is \emph{distance reducing} or is a ``short map". For future reference, let us record that:

\begin{lem}[\cite{schneider:book}]\label{lem:short}
For any convex body $K\subset\R^n$, the nearest point projection $\pi\colon (\R^n-\inte(K))\to\p K$ is distance reducing.
\end{lem}

Now we are ready to establish the main result of this section:

\begin{proof}[Proof of Proposition \ref{prop:distort}]
Note that $\distort(C)=\distort(\inte(C))$ since  the ratio $d(x,y)/\|x-y\|$ depends continuously on $x$ and $y$. 
Let $x$, $y\in\inte(C)$ be distinct points. Since $C$ is star-shaped with respect to $o$, the points $x$, $y$, and $o$ are not colinear. Thus there exists a unique  ($2$-dimensional) plane $\Pi$ passing through $x$, $y$, and $o$. Further note that $\Pi\not\subset H$, since $x$, $y\in\inte(C)$. Consequently 
$
\Gamma:=\Pi\cap C
$
 is a convex arc (as opposed to a closed curve). In particular there exists a unique segment of $\Gamma$ bounded by $x$, $y$ which we denote by $\overset{\frown}{xy}$. Note that $d(x,y)\leq\length(\overset{\frown}{xy})$. Thus it suffices to show that
 $$
 \frac{\length(\overset{\frown}{xy})}{\|xy\|}\leq 4(2+\pi)\frac{R}{r},
 $$
 where $xy$ denotes the line segment connecting $x$ and $y$. Let $B_r$, $B_R$ denote the half-balls centered at $o$ such that $B_r\subset\conv(C)\subset B_R$, then $D_r:=\Pi\cap B_r$, and $D_R:=\Pi\cap B_R$ are half-disks centered at $o$ such that 
$$
D_r\subset\conv(\Gamma)\subset D_R.
$$ 
Then Lemma \ref{lem:short} implies that
$$
\length(\Gamma)\leq \length (\p\conv(\Gamma))\leq\length(\p D_R)= (2+\pi) R.
$$
Thus if $\|xy\|\geq r/4$, then
$$\frac{\length(\overset{\frown}{xy})}{\|xy\|}\leq\frac{\length(\Gamma)}{r/4}\leq\frac{(2+\pi)R}{r/4}=4\(2+\pi\)\frac{R}{r},$$
as desired.
So we may assume that
$$
\|xy\|< \frac{r}{4}.
$$

Let $L$ be the line passing through $x$ and $y$, see Figure \ref{fig:annulus}, which depicts the configuration in the plane $\Pi$. 
If $o$ lies on $L$, then it has to be between $x$ and $y$, because  $\Gamma$ is star-shaped with respect to $o$. Consequently, $\|xy\|=\|x-o\|+\|o-y\|\geq 2r$, which is not possible.
Thus $o$ lies in  the interior of one of the sides of $L$ in $\Pi$, say 
$
o\in\inte(L^-).
$
 \begin{figure}[h]
\begin{overpic}[height=1.25in]{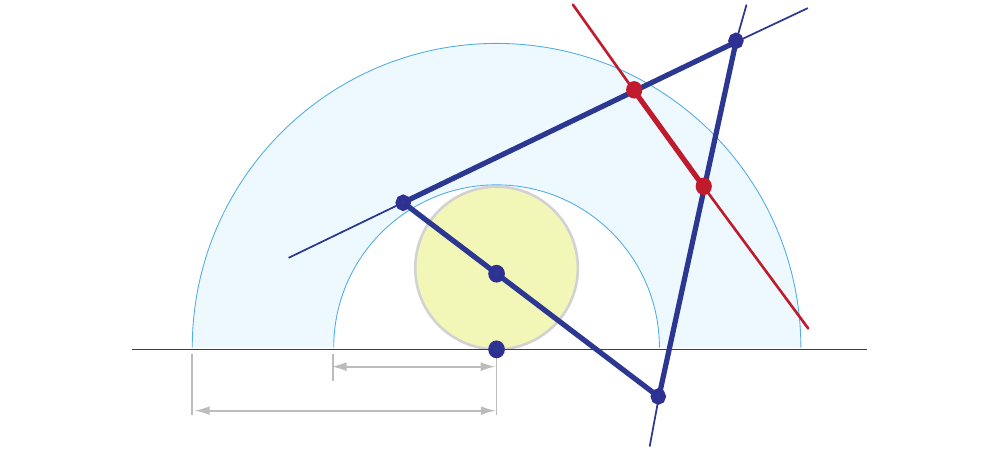}
\put(40,5){\small$r$}
\put(33,0){\small$R$}
\put(49,11.5){\small$o$}
\put(49,19.5){\small$o'$}
\put(8,9){\small$H$}
\put(44,14){\small$S$}
\put(58,36){\small$x$}
\put(67,24){\small$y$}
\put(75,38.5){\small$o''$}
\put(53.5,43.5){\small$L$}
\put(38,26){\small$x'$}
\put(67,3){\small$y'$}
\end{overpic}
\caption{}
\label{fig:annulus}
\end{figure}

We may assume that $\overset{\frown}{xy}$ is not a line segment for otherwise $\length(\overset{\frown}{xy})/\|xy\|=1$ and there is nothing to prove. Thus, since $\overset{\frown}{xy}$ is convex arc, $\inte(\overset{\frown}{xy})$ lies in the interior of one side of $L$.
If $\inte(\overset{\frown}{xy})\subset\inte(L^-)$, then, since $\Gamma$ is convex,  the rest of $\Gamma$, including its end points $a$, $b$ lie in $L^+$, which is not possible, since $o\in ab$.
  So we may assume that 
 $
\inte(\overset{\frown}{xy})\subset\inte(L^+).
 $ 

 Next let  $o'$ be the center of the  circle $S$ of diameter $r$ contained in $D_r$.  We claim that 
 $
 o'\in\inte(L^-).
 $
  Suppose, towards a contradiction, that 
 $o'\in L^+.$
 Then  $L\cap oo'\neq\emptyset$,  since  $o\in L^-$. Let $z \in L\cap oo'$. Then $x$ and $y$  lie on the same side of $z$ in $L$, since $\|xy\|<r$. We may suppose that $x\in\inte(zy)$.
 Let $Q$ be a support line of $\Gamma$ at $x$. Then $z$ and $y$  lie on the same side of $Q$, since $z\in D_r\subset\conv(\Gamma)$, and $y\in\Gamma$. So $zy\subset Q$.  But $z\in oo'\subset\inte(D_r)\subset\inte(\conv (\Gamma))$. So $Q\cap \inte(\conv (\Gamma))\neq\emptyset$. This  is the desired contradiction since $Q$ supports $\Gamma$.

Now let  $Q_x$,  $Q_y$ be   support lines of $\Gamma$ at $x$, $y$ respectively. Then  $Q_x$,  $Q_y$ also support $S$, since $S\subset D_r\subset \conv(\Gamma)$. So, since $\|xy\|<r$, it follows that $Q_x$, $Q_y$ intersect  at a point $o''$. Further, since  $o'\in\inte(L^-)$, $o''\in L^+$. Finally, since  $\overset{\frown}{xy}\cap\inte(L^+)\neq\emptyset$,  then $Q_x$, $Q_y\neq L$.  So it follows that 
$
o''\in\inte(L^+).
$
This in turn yields that there are points
$x'$, $y'$  on $Q_x$, $Q_y$ respectively such that $\|o''x'\|=\|o''y'\|$ and $o'\in x'y'$. Analyzing the triangle $o''x'y'$ yields our desired estimate as follows.

 \begin{figure}[h]
\begin{overpic}[height=1.0in]{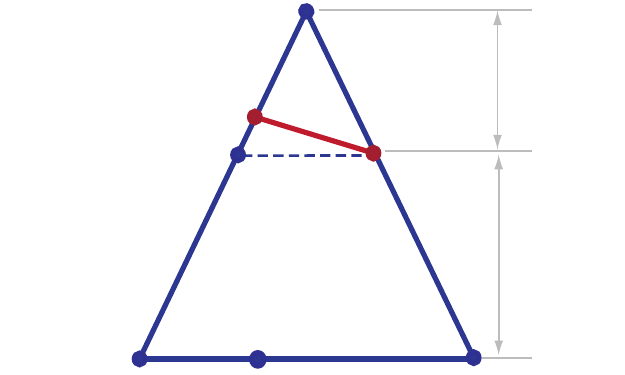}
\put(35.5,42){\small$x$}
\put(29,33){\small$x''$}
\put(63,34){\small$y$}
\put(75,-4){\small$y'$}
\put(18.5,-4.5){\small$x'$}
\put(48,61){\small$o''$}
\put(47,31){\small$a$}
\put(47,4.5){\small$b$}
\put(39,-5){\small$o'$}
\put(82,17){\small$c$}
\put(82,45){\small$d$}
\end{overpic}
\caption{}
\label{fig:triangle2}
\end{figure}

Suppose that $\|o''x\|\leq\|o''y\|$, and let $x''$ be the point on $o''x'$ so that $x''y$ is parallel to $x'y'$, see Figure \ref{fig:triangle2}.  Further let  $a:=\|x''y\|$, $b:=\|x'y'\|$, $c$ be the distance between $x''y$ and $x'y'$, and $d$ be  the distance between $o''$ and $x''y$. Then we have
$$
a\leq 2\|xy\|\leq r/2, \quad\quad\quad\quad b\geq r, \quad\quad\quad\quad c\leq \sqrt{r^2+R^2}.
$$ 
The first estimate holds because $o''x''y$ is an isosceles triangle. The second estimate holds because  $o'\in x'y'$, but $x'$, $y'$ lie outside $S$. The third estimate holds because $c\leq \|o'x\|$; further, since the angle $\measuredangle o'ox'\leq\pi/2$,  we have $\|o'x'\|^2\leq \|o'o\|^2+\|ox'\|^2\leq r^2+R^2$. 
Next note that, since $d/a=(c+d)/b$, 
$$
d=\frac{ac}{b-a}\leq\frac{2\|xy\|\sqrt{r^2+R^2}}{r/2}=4\sqrt{1+\(\frac{R}{r}\)^2}\,\|xy\|.
$$
Recall that  $\overset{\frown}{xy}$ lies on the same side of $L$ as does $o''$. So $\overset{\frown}{xy}$ is inscribed in the triangle $o''xy$. Consequently, by Lemma \ref{lem:short}, $\length(\overset{\frown}{xy})\leq\|o''x\|+\|o''y\|$. So, since $\|o''x\|\leq\|o''y\|$, we have
$$
\length (\overset{\frown}{xy})\leq 2\|o''y\|=2\sqrt{\frac{a^2}{4}+d^2}\leq 2\sqrt{17+16\(\frac{R}{r}\)^2}\|xy\|\leq 2\sqrt{33}\frac{R}{r}\,\|xy\|,
$$
where in the middle inequality we used the estimate $a\leq 2\|xy\|$, and in the last inequality we employed the assumption  that $R\geq r$.
Now we can  compute that
$$
\frac{\length(\overset{\frown}{xy})}{\|xy\|}\leq 2\sqrt{33}\,\frac{R}{r}\leq 4(2+\pi)\frac Rr,
$$
which completes the proof.
\end{proof}

\subsection{Convergence}\label{sec:convergence}
Here we use the distortion estimate established in the last section to obtain the following  convergence result. A sequence of caps $C_i$  is \emph{nested} provided that $C_i\subset C_{i+1}$.

\begin{prop}\label{prop:converge}
Let $C_i\subset M$ be a nested sequence of caps. Then $\cl(\cup_iC_i)$ is also a cap.
\end{prop}

To establish this result we need to recall some basic concepts from metric geometry.
 If $\mathcal{M}$ is any metric space, then the \emph{Hausdorff distance} between any pair of compact subsets $X$, $Y$ of $\mathcal{M}$ is defined as:
$$
d_H(X,Y):=\inf_{r\geq 0}\{X\subset B_r(Y)\; \text{and} \; Y\subset B_r(X)\},
$$
where $B_r(X)$ denotes the set of all points of $\mathcal{M}$ which are within a distance $r$ of $X$. It is well-known that $d_H$ defines a metric on the space of compact subsets of $\mathcal{M}$ \cite[Sec. 7.3]{bbi:book}. By \emph{convergence} in this paper we shall always mean convergence with respect to this metric, and we write $X_i\to X$ when $X_i$ converge to $X$. The  following observation is a simple exercise \cite[p. 253]{bbi:book}.

\begin{lem}[\cite{bbi:book}]\label{lem:blaschke}
Let $\mathcal{M}$ be a compact metric space and $X_i$ be a sequence of compact subsets of $\mathcal{M}$.
\begin{enumerate}
\item[(i)]{If $X_i\subset X_{i+1}$, then $X_i\to\cl(\cup_i X_i)$.}
\item[(ii)]{If $X_{i+1}\subset X_i$, then $X_i\to\cap_i X_i$.}
\end{enumerate}
\end{lem}

\noindent
When discussing convergence  we will assume that various objects are endowed with their natural metric. For instance,  the boundary   of a convex body carries the metric which is induced on it from its ambient space. Further, the metric associated to $M$ will be the one induced on it by $f$, which is well-defined since $f$ is locally convex.

\begin{lem}\label{lem:limitofcaps}
Let $C_i$ be a sequence of nested caps in the boundary of a convex body $K\subset\R^3$. Then $C_i$ converge either to a cap or to $\p K$. Further, if $C_i\to\p K$, then $\p C_i$ converge either to a point or a line segment.
\end{lem}
\begin{proof}
By Lemma \ref{lem:blaschke}, $C_i\to C:=\cl(\cup_i C_i)$. We need to show that $C$ is a cap or $C=\partial K$. First suppose that $C_i$ is degenerate for some $i$, then it follows from the uniqueness property (Proposition \ref{prop:capunique}) that, for all $j\leq i$, $C_j=C_i$. Thus, if all $C_i$ are degenerate, then all $C_i$ coincide, in which case $C=C_i$ and we are done. Suppose then that $C_i$ is nondegenerate for some $i$. Then, for all $j\geq i$, $C_j$ is also nondegenerate. Consequently, after a truncation of the sequence, we may assume that all $C_i$ are nondegenerate. Let
$H_i$ be the planes of $C_i$, and $H_i^+$ be the corresponding half-spaces. Then 
$$
C_i=\cl(\partial K\cap\inte(H_i^+)).
$$
Since each $H_i$ intersects $K$, $H_i$  converge to a plane $H$ after passing to a subsequence. Further, we may assume that $H_i^+\to H^+$. More precisely, after passing to another subsequence we may assume that the outward normals of $H_i$ converge to a normal of $H$ which we designate as the outward normal of $H$. We claim that 
$$
C=\cl(\p K\cap \inte(H^+)),
$$
 which will finish the proof.
To establish this claim, let $p\in\p K\cap\inte(H^+)$. Then,  for $i$ sufficiently large, $p\in \partial K\cap\inte (H_i^+)$.  Thus $p\in C_i\subset C$. So $\p K\cap\inte(H^+)\subset C$. Since $C$ is compact, it follows that 
$$
\cl(\p K\cap \inte(H^+))\subset C.
$$
It remains then to obtain the reverse inclusion. If $p\in\p K\cap\inte(H^-)$, then there exists a ball $B$ centered at $p$ such that $B \cap H^+=\emptyset$. So, for $i$ sufficiently large, $B \cap H^+_i=\emptyset$, which yields that $B\cap C_i=\emptyset$. Since $C_i$ is nested, it follows then that $B\cap C_i=\emptyset$ for all $i$. Thus  
$p\not\in C$. So $C\subset\p K\cap H^+$. Now if we set $D:=H\cap K$, then we have
 $$
 C\subset\p K\cap H^+=\cl(\p K\cap \inte(H^+))\cup \inte_H(D),
 $$
 where $\inte_H(D)$ here denotes the interior of $D$ in $H$.
Thus if  $C\cap\inte_H(D)=\emptyset$, then  we are done. In particular, if $H$ contains an interior point of $K$, we are finished, since in that case $\inte_H(D)\subset\inte(K)$  which is disjoint from $\partial K\supset C$. So we may assume that $H$  supports $K$. Then $D=H\cap\partial K=H^-\cap \p K$. Now we claim that
\begin{equation}\label{eq:CHD}
C\cap\inte_H(D)=\emptyset,
\end{equation}
 which is all we need. Let $D_i:=\partial K-\inte(C_i)=H_i^-\cap\p K$. Then $D_i \to H^-\cap \p K=D$.
 On the other hand, since $C_i\subset C_{i+1}$,  $D_{i+1}\subset D_i$. Thus $D_i\to\cap_i D_i$ by Lemma \ref{lem:blaschke}. So $D=\cap_i D_i$ by the uniqueness of limits in metric convergence. In particular $D\subset D_i$. So $\inte_H(D)\subset\inte(D)\subset\inte(D_i)=\p K-C_i$. Hence 
 \begin{equation}\label{eq:CiHD}
 C_i\cap\inte_H(D)=\emptyset.
\end{equation}
 Now note that $\inte_H(D)$ is open in $\p K$. Thus,  if $C$ intersects $\inte_H(D)$, then so does $C_i$  for large $i$, which is not possible.
 Hence 
  \eqref{eq:CiHD} implies \eqref{eq:CHD}, which completes the proof of the first claim of the lemma.
  
  To establish the second claim, suppose that $C_i\to\p K$. Then $\inte_H(D)=\emptyset$. Further recall that $D$ is convex, since $D=H\cap K$. So $D$ is either a point or a line segment. Finally since $D_i\to D$, it is easy to see that 
$\p C_i=\p D_i\to\p D=D$, which completes the proof.
 \end{proof}

\begin{lem}\label{lem:Ci}
Let $C_i$ be a sequence of caps in $\R^3$ which converge to a nondegenerate cap $C$. Then the distortion of $C_i$ is uniformly bounded.
\end{lem}
\begin{proof}
After rigid motions we may assume that the planes $H_i$ of $C_i$ coincide with the plane $H$ of $C$.
Let $o$ be a point in $H$  which lies in the interior of the region bounded by $\p C$. Then there exists a pair of concentric hemispheres in $H^+$ centered at $o$ which contain $C$ in between them, and are disjoint from $C$. Consequently, $C_i$ will also lie in between these hemispheres in $H^+$ for $i$ sufficiently large, and thus will have uniformly bounded distortion by Proposition \ref{prop:distort}.
\end{proof}

\begin{lem}\label{lem:Xi}
Let $X_i\subset M$ be a sequence of compact path connected sets which converge to a  set $X$. Suppose that $f$ is injective on each $X_i$ and $\distort(f(X_i))$ is uniformly bounded. Then $f$ is injective on $X$. 
\end{lem}
\begin{proof}
Suppose, towards a contradiction, that there is a pair of distinct points $x$, $y\in X$ such that $f(x)=f(y)$. Since $X_i\to X$, there are sequences of points $x_i$, $y_i\in X_i$ which converge to $x$, $y$ respectively. Let $d$ be the metric which $f$ induces on $M$, and $\tilde d_i$ be the intrinsic distance of $f(X_i)\subset\R^3$. Then $\tilde d_i(f(x_i), f(y_i))\geq d(x_i,y_i)$, since $f$ is injective on $X_i$. So
$$
\distort(f(X_i))
\geq\frac{\tilde d_i(f(x_i),f(y_i))}{\|f(x_i)-f(y_i)\|}
\geq \frac{d(x_i,y_i)}{\|f(x_i)-f(y_i)\|}.
$$
Note that $d(x_i,y_i)\to d(x,y)$ which is positive by assumption, while
$\|f(x_i)-f(y_i)\|\to\|f(x)-f(y)\|=0$. Thus $\distort(f(X_i))$ grows indefinitely, which is the desired contradiction.
\end{proof}

Now we are ready to establish the main result of this section:

\begin{proof}[Proof of Proposition \ref{prop:converge}]
Let $C:=\cl(\cup_i C_i)$. We need to show that $f$ is injective on $C$ and $\ol C:=f(C)$ is a cap in $\R^3$.
By  Lemma \ref{lem:blaschke}, $C_i\to C$. Thus, 
since $f$ is continuous, $\ol C_i:=f(C_i)$ converge to $\ol C:=f(C)$.  For each $\ol C_i$ let $K_i$ be a convex body such 
that $\ol C_i\subset\p K_i$. Since $\ol C_i\to\ol C$, which  is compact, we may assume that  $K_i$  lie in some large ball. Further we may assume that the inradii of $K_i$ are bounded below, since  $\ol C_i$ are nested. Thus, by Blaschke's selection principal for convex bodies \cite[Thm. 1.8.6]{schneider:book}, $K_i$ converge to a convex body $K$, after passing to a subsequence. Further it is well-known that $\p K_i\to\p K$, e.g., see \cite[p. 424]{alexandrov:intrinsic}.
Further note that for $j\geq i$, $\ol C_i\subset \ol C_j\subset \p K_j$. Thus $\ol C_i\subset\p K$, which yields that $\ol C\subset\p K$. So it follows from Lemma \ref{lem:nestedcaps} that $\ol C$ is either a cap or $\ol C=\p K$.
If $\ol C$ is a cap, then, by Lemma \ref{lem:Ci}, $\ol C_i$ have uniformly bounded distortion. Consequently, by Lemma \ref{lem:Xi}, $f$ is injective on $C$, and we are done.

\begin{figure}[h]
\begin{overpic}[height=1.2in]{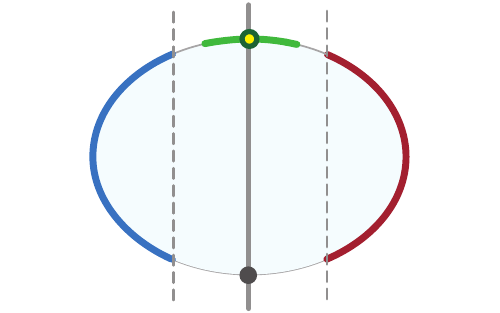}
\put(51.5,10.5){\Small $X$}
\put(47,-5){\Small $H$}
\put(52,48.5){\Small $\ol p$}
\put(56,56){\Small $\ol U$}
\put(84,30){\Small $\ol{C_i^+}$}
\put(7,30){\Small $\ol{C_i^-}$}
\put(39,30){\Small $K$}
\end{overpic}
\caption{}
\label{fig:sandwich}
\end{figure}

It remains then to show that  $\ol C\neq\p K$. Suppose, towards a contradiction, that $\ol C=\p K$. Then $\ol{\p C_i}$  converge  to a point or line segment $X\subset\p K$, by Lemma \ref{lem:limitofcaps}. Let $H$ be a plane which contains $X$, and   an interior point of $K$, see Figure \ref{fig:sandwich}. 
Further let $H^+$ be one of the sides of $H$.  Since   $\ol{\p C_i}\to X\subset H$, there exists a sequence $H_i^+$ of nested half-spaces converging to  $H^+$, such that   $\ol{\p C_i}\subset \inte(H_i^-)$. Let $C_i^+:=f^{-1}(H_i^+)\cap C_i$. Then $C_i^+$ form a nested sequence of caps in 
$M$. Thus, by Lemma \ref{lem:blaschke}, $C_i^+$ converge to a compact set $C^+\subset C$. Note that $\ol{C^+}:=f(C^+)$ cannot coincide with $\partial K$, since $\ol{C^+}\subset H^+$, and $\partial K$ has points in $\inte(H^-)$. So we conclude that $\ol{C^+}$ is a cap in $\R^3$, and therefore, $C^+$ is a cap in $M$, as discussed in the previous paragraph.
By applying the same argument to the other side of $H$, we may also produce a cap $C^-\subset C$ with half-space $H^-$, which is a limit of a nested sequence of caps $C_i^-$.

Now note that $C^+\cap C^-\neq\emptyset$. To see this  let $p\in C$ be a point such that $\ol p:=f(p)\in H-X$.
Then $\ol p\in\inte(\ol C_i)$ for large $i$, since $\ol{\p C_i}\to X$. So $p\in \inte(C_i)$.  Let $U\subset\inte(C_i)$ be a neighborhood of $p$ which is so small that $f$ is one-to-one on $U$. Then $\ol U:=f(U)$ is open in $\p K$ by the invariance of domain. Consequently, $\ol U$ has points in the interiors of both sides of $H$. So $U$ intersects $C_i^\pm$ for large $i$. Thus there are sequences $p_i^\pm\in C_i^\pm\subset C^\pm$ which converge to $p$. Since $C^\pm$ are compact,  $p\in C^\pm$. So $C^+\cap C^-\neq\emptyset$ as claimed.
 Consequently $C^+=C^-$ by the uniqueness property (Proposition \ref{prop:capunique}), since $H$ is the plane of both caps. This is the desired contradiction, since $\ol{C^\pm}$  lie on different sides of $H$.
\end{proof}

\begin{note}
One might suspect that Proposition \ref{prop:converge} hints at a more general phenomenon, namely that the space of caps of $M$ is compact; however, this is not so. For instance let $M$ be the surface obtained by revolving the curve in Figure \ref{fig:capline} about its axis of symmetry. 
\begin{figure}[h]
\begin{overpic}[height=0.7in]{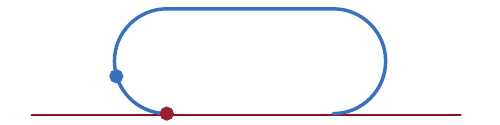}
\put(18,9){\Small $p_i$}
\put(32,-1.5){\Small $p$}
\end{overpic}
\caption{}
\label{fig:capline}
\end{figure}
Then $M$ is a nondegenerate cap. Let  $p_i$ be a sequence of points in the interior of $M$ which converge to a point $p\in \p M$. If $p_i$ are sufficiently close to $p$, then each $p_i$ is a cap, since there is a neighborhood $U$ of $\p M$ in $M$ such that $U-\p M$ is strictly convex. On the other hand $p$ itself is not a cap. Indeed,  by the extension property (Proposition \ref{prop:capextend}), if $p$ were a cap, then it would have to be a locally strictly convex point of $M$, which is not the case.
\end{note}

\subsection{The slicing theorem}\label{sec:slicing}

The main result of this section is the following comprehensive structure theorem for caps:
 
 \begin{thm}\label{thm:caps}
Let $H\subset\R^3$ be a plane, and $C$ be a component of 
$f^{-1}(H^+)$. Suppose that 
\begin{enumerate}
\item[(i)]{$f(C\cap\partial M)\subset H$,}
\item[(ii)]{$C\cap\inte(M)\neq\emptyset$.}
 \end{enumerate}
  Then $C$ is a cap.
\end{thm}
 
 The prove  the above theorem, let us   recall that an immersion, or locally one-to-one continuous map,  $f\colon M\to\R^3$, where $M$ is a manifold without boundary, is \emph{complete} provided that the metric which $f$ induces on $M$ is complete, i.e., all the Cauchy sequences in $M$ converge. We say $f$ is \emph{locally strictly convex} at some point $p\in M$, if there exists a neighborhood $U$  of $p$ 
and a plane $H$ such that $f(U)\cap H=\{f(p)\}$.

 \begin{lem}[van Heijenoort \cite{vH:convex}]\label{lem:vH}
 Let $M$ be a connected $2$ dimensional manifold without boundary, and $f\colon M\to\R^3$ be a complete immersion. Suppose that $f$ is locally convex on $M$, and is locally strictly convex at some point of $M$. Then $M$ is homeomorphic to $\R^2$, $f$ is one-to-one, and $f(M)$ forms the boundary of a convex subset of $\R^3$.
 \end{lem}
 
 The above lemma, via a projective transformation, yields:

 \begin{prop}\label{prop:cap1}
 Let $H^+$ be a  half-space, and  $X$ be a  component of $f^{-1}(\inte (H^+))$. Suppose that $X\cap\p M=\emptyset$. Then $\cl (X)$ is a cap.
 \end{prop}

 \begin{proof}
 First we show that $f$ is locally strictly convex at some point $p$ of $X$. To see this, let $\p X$ denote the topological boundary of $X$ in $M$. Then $\p X$ is compact, since $\p X$ is closed in $M$, and $M$ is compact. So $f(\p X)$ is bounded. Further note that $f(\p X)\subset H$. Consequently it is not hard to show that there exists a ball $B$ in $\R^3$ which contains $f(\p X)$ but not all of $f(X)$, since $f(X)$ has a point outside of $H$ by definition of $X$. Let $o$ be the center of $B$, define $h\colon \cl(X)\to \R$  by $h(x):=\|f(x)-o\|$, and let $p$ be a maximum point of $h$. Since $f(\p X)\subset B$, it follows that $p\in X$. Note that $X$ is an open neighborhood of $p$ in $M$, $f(X)$ is contained in a ball $B'$ of radius $h(p)$ centered at $o$, and $f(p)\in\p B'$. Thus $f$ is locally strictly convex at $p$, as desired.

 After a rigid motion, we may suppose that $H$ coincides with the $xy$-plane and $H^+$ is the upper half-space $z\geq 0$.  Consider the projective transformation
 $$
\inte(H^+)\ni(x,y,z)\overset{\phi}{\longmapsto}\(\frac{x}{z},\frac{y}{z},\frac{1}{z}\)\in\inte(H^+).
 $$ 
 Then $\phi\circ f\colon X\to\R^3$ is a complete locally convex immersion, which is locally strictly convex at $p$.  
 Further note that since $X$ is open in $M$, and $X\cap\p M=\emptyset$, $X$ is a manifold without boundary.
 Thus, by Lemma \ref{lem:vH},  $\phi\circ f$ embeds $X$ onto the boundary of a convex set $A\subset\R^3$. This in turn implies that $f=\phi^{-1}(\phi\circ f)$ maps $X$ injectively into the boundary of the convex body  $K:=\cl(\phi^{-1}(A))$. 
It then follows that
 $$
 f(X)=\p K\cap\inte(H^+).
 $$ 
Now let $H_i^+$ be a sequence of nested half-spaces in $\inte(H^+)$ which converge to $H^+$, 
and  set $C_i:=f^{-1}(H_i^+)\cap X$. Then
$$
X=\cup_i C_i.
$$
Note that $C_i$ is compact and $f$ is one-to-one on $C_i$ since $f$ is one-to-one on $X$. Thus $C_i$ is homeomorphic to $f(C_i)=H_i^+\cap f(X)=H_i^+\cap \p K$. Furthermore, $H_i^+\cap \p K$ is a cap since $K$ has points in the interiors of both sides of $H_i$. So $C_i$ form a nested sequence of caps in $M$. Thus  $\cl(\cup_i C_i)$ is a cap by Proposition \ref{prop:converge}. So $\cl(X)$ is a cap as claimed.
  \end{proof}

We now combine Proposition \ref{prop:cap1} with Theorem \ref{thm:flatcap} to complete the proof of the main result of this section:

\begin{proof}[Proof of Theorem \ref{thm:caps}]
First suppose that $f(C)\subset H$. Then $H$ locally supports $f$ at a point $p\in \inte(M)$ and $C$ is the component of $f^{-1}(H)$ which contains $p$. So $C$ is a cap by Theorem \ref{thm:flatcap}.
We may assume then that there is a point $p\in C$ such that $f(p)\in \inte(H^+)$. Let $C'$ be  the closure of the component of $f^{-1}(\inte(H^+))$ which contains $p$. Then $C'$  is a cap by Proposition \ref{prop:cap1}. We claim that  
$C=C'$, which will finish the proof. That $C'\subset C$, follows immediately from the definition of $C'$ and connectedness of $C$. To establish the reverse inclusion, recall that, by the extension property (Proposition \ref{prop:capextend}),  there is a neighborhood $U$ of $C'$ in $M$ such that $f(U-C')$ is disjoint from $H^+$.   So $U-C'$ is disjoint from $C$. Consequently, since $C$ is connected,  $C\subset C'$. So $C=C'$ as claimed.
\end{proof}

\section{Existence of Maximal Caps}\label{sec:maxcap}
A cap $C\subset M$ is \emph{maximal} provided that it is contained in no other cap of $M$. The main result of this section, Theorem \ref{thm:maxcap}, is that through each point $p\in\partial M$ there passes a maximal cap. The intuitive idea here, as we described in Section \ref{subsec:outline}, is to construct these caps  by rotating the tangent planes of $f$ along its boundary. To this end, first we discuss what we mean by a tangent plane in Section \ref{sec:tancone}. Then we use these planes to show that through each point $p\in \partial M$ there passes some cap. Finally we associate an \emph{angle} to each cap at $p$ in Section \ref{subsec:maxcap} and use this notion, together with Theorem \ref{thm:caps} (the slicing theorem) developed in the last section, to prove the existence of maximal caps.

\subsection{Tangent cones}\label{sec:tancone}
Here we show that the $\C^1$ regularity assumption on $\f$ together with the local convexity of $f$ allow us to assign a plane $H_p(0)$ to each  point $p\in \p M$ which coincides with the standard notion of the tangent plane when $f$ is differentiable at $p$. Furthermore, much like its counterpart in the smooth case, $H_p(0)$ will locally support $f$ at $p$.
To establish these results, we start with a quick review of  the concept  of tangent cones; for more background here see \cite{ghomi&howard:tancones}.

For any set $X\subset\R^n$ and $p\in X$, the \emph{tangent cone} of $X$ at $p$, which we denote by $T_p X$, consists of the limits of all (secant) rays which originate from $p$ and pass through a sequence of points $p_i\in X-\{p\}$ which converge to $p$. A basic fact that we will employ in this work is that if $H\subset \R^n$ is a hyperplane and $\pi\colon\R^n\to H$ denotes the corresponding orthogonal projection, then 
\begin{equation}\label{eq:piTp}
\pi( T_p X)=T_{\pi(p)}\pi(X).
\end{equation}
The other more specific facts that we need about tangent cones of convex bodies are enumerated in the next lemma. A subset $X$ of $\R^n$ is a \emph{cone}, based at a point $p$, provided that for every point $x\in X-\{p\}$ the ray originating from $p$ and passing through $x$ is contained in $X$. If $K\subset\R^n$ is a convex body and $p\in \p K$, then the \emph{support cone}, $S_pK$, of $K$ at $p$ is defined as the intersection of all closed half-spaces which contain $K$ and whose boundary passes through $p$.

\begin{lem}\label{lem:tancone1}
Let $K\subset\R^n$ be a convex body and $p\in\p K$. Then
\begin{enumerate}
\item[(i)]{$T_p K=S_p K$; in particular, $T_pK$ is a closed convex set.}
\item[(ii)]{$\p (T_p K)=T_p\p K$.}
\end{enumerate}
\end{lem}
\begin{proof} 
To see (i), let $H$ be a support plane of $K$ at $p$, and $H^-$ be the side containing $K$.  Then $H^-$ contains every ray which emanates from $p$ and goes through a point of $K$. Furthermore, $H^-$ also contains the limit of any family of such rays since it is closed. So $T_p K\subset S_p K$. Now suppose, towards a contradiction, that $T_p K$ is a proper subset of $S_p K$. There exists then a ray $r$ on the boundary of $T_p K$ which lies in the interior of $S_p K$. Let $H$ be a support plane of $T_p K$ which contains $r$ and $H^-$ be the corresponding half-space which contains $T_p K$. Then $H^-\cap S_p K$ is a proper subset of $S_p K$, which is the desired contradiction, by the definition of $S_p K$ and since $K\subset T_p K\subset H^-$. 
For a proof of item (ii), see  \cite[Lem.\@ 5.6]{ghomi&howard:tancones}. 
\end{proof}

A convex subset of $\R^3$ is a \emph{wedge} provided that it is either a closed half-space, or is the intersection of a pair of closed half-spaces determined by non-parallel planes. Note that a closed convex subset of $\R^3$ which has interior points in $\R^3$ is a wedge if and only if it is bounded by a pair of half-planes with a common boundary line.

\begin{lem}\label{lem:tancone2}
Let $K\subset\R^3$ be a convex body, $\Gamma\subset\p K$ be a $\C^1$-embedded curve segment, $p$ be an interior point of $\Gamma$, and $L$ be the tangent line of $\Gamma$ at $p$. Then  $T_p\p K$ consists of a pair of half planes meeting along $L$. In particular, $T_p K$ is a wedge.
\end{lem}
\begin{proof}
Note that $L\subset T_p K$, since each of the two rays in $L$ originating from $p$ is a limit of secant rays of $\Gamma$. 
Further recall that,  by Lemma \ref{lem:tancone1}, $T_p K$ is convex. Thus it follows that $T_p K$ is cylindrical with respect to $L$, i.e., through each point of $T_pK$ there passes a line parallel to $L$, see \cite[Thm. 8.3]{rockafellar}. Let $\Pi$ be the plane orthogonal to $L$ at $p$ and consider the cross section $T_p K\cap \Pi$.

 Since $T_p K$ is a convex cone, then so is $T_p K\cap \Pi$. Further note that since $\Pi$ meets $\Gamma$ transversely, it cannot be a support plane of $K$. So $K\cap\Pi$ has interior points in $\Pi$, which in turn yields that so does $T_p K\cap\Pi$, since $K\subset T_pK$. Thus $T_p K\cap \Pi$ is a convex planar cone with interior points, which yields that $\p (T_p K\cap\Pi)$ consists of a pair of half-lines meeting at $p$. 
 
 Now note that  $\p (T_pK)$ is an orthogonal cylinder over  $\p (T_p K\cap\Pi)$, due to the cylindricity of $T_p K$. Consequently $\p (T_pK)$ consists of a pair of half-planes meeting along $L$. It only remains to note that $\p (T_p K)=T_p(\p K)$ by Lemma \ref{lem:tancone1}, which completes the proof.
\end{proof}

We  define the \emph{tangent cone} of $f$ a point $p\in M$ as
 $$
 T_p f:=T_{f(p)} f(U),
 $$ 
 where $U$ is a neighborhood of $p$ in $M$.
 Note that $T_p f$ depends only on an arbitrarily small neighborhood of $f(p)$ in $f(U)$ and thus is well defined, i.e., it does not depend on the choice of $U$. Indeed, if $U_1$, $U_2$ are neighborhoods of $p$ in $M$ which are embedded by $f$ into the boundaries of  convex bodies, then $T_{f(p)} f(U_1)=T_{f(p)} f(U_1\cap U_2)=T_{f(p)} f(U_2)$. The main result of this section is as follows:
 
 \begin{prop}\label{prop:halfplane}
 Let $p\in\p M$, $U$ be a convex neighborhood of $p$ in $M$ with associated body $K$, and $L$ be the tangent line of $\f$ at $p$. Then $T_p f$ coincides with one of the half planes of $T_{f(p)}\p K$ determined by $L$. 
  \end{prop}

\begin{proof}
Let $D:=\cl(U)$. Recall that by the definition of convex neighborhoods,  $f\colon D\to f(D)\subset \p K$ is a homeomorphism. Thus we may identify $f(D)$  with $D$,  and drop all references to $f$ for convenience. Then $D$ is an embedded  disk in $\p K$, $p\in\p D$, and $\p D$ is $\C^1$ near $p$. Now we simply need to show that $T_p D$ coincides with one of the half planes of $T_{p}\p K$ determined by $L$.

Let $H$ be a support plane of $T_{p}(\p K)$ at $p$ such that $T_{p}(\p K)$ is a graph over $H$. Then $L$   lies in $H$. Let  $L^+$, $L^-$ denote the sides of $L$ in $H$.
As is well-known, e.g., see \cite{ghomi&howard:tancones}, $T_{p}(\p K)$ is the limit of the homothetic dilations of $\p K$ based at $p$, and thus, $\p K$ is locally a graph over $H$ near $p$. In other words, if $\pi\colon \R^3\to H$ denotes the orthogonal projection, then $\pi$ is one-to-one in a neighborhood  of $p$ in $\p K$. So, we may assume that $\pi\colon D\to H$ is one-to-one. Then $\ol D:=\pi(D)$ is an embedded disk in $H$ whose boundary $\ol{\p D}:=\pi(\p D)$ is $\C^1$ near $p$ and is tangent to $L$ at $p$. 

Since $T_p D\subset T_p\p K$, and $\pi\colon T_p\p K\to H$ is one-to-one, we just need to show that
$\pi(T_pD)$ coincides with one of the sides of $L$ in $H$. But recall that $\pi(T_p D)= T_p(\pi (D))$ by \eqref{eq:piTp}. So we just need to verify that $T_p(\ol D)$ is one of the half-planes $L^+$ or $L^-$. To this end note that since $\ol{\p D}$ is tangent to $L$ at $p$, there is a normal vector $v$ to $L$ at $p$ which points inside $\ol D$.  Let $L^+$ be the side of $L$ to which $v$ points. Then  it follows that $T_p(\ol D)=L^+$.

To establish the last claim let $r$ be a ray in $H$ originating from $p$ which makes an angle of $0\leq\theta\leq\pi$ with $v$. If $\theta<\pi/2$, then $r\cap \ol D$ contains a neighborhood of $p$ in $r$. Thus $r\subset T_p\ol D$, which  yields that 
 $\inte(L^+)\subset T_p(\ol D)$. Of course $L\subset T_p(\ol D)$ as well, since $L$ is the tangent line of $\ol{\p D}$ at $p$. So $L^+\subset T_p(\ol D)$. For the reverse inclusion, suppose that $\theta > \pi/2$. Then $r$ lies in the interior of a cone based at $p$ which does not intersect $\ol D-\{p\}$ near $p$. It follows then that $\inte(L^-)$ is disjoint from  $T_p(\ol D)$, e.g., see \cite[Lem. 2.1]{ghomi&howard:tancones}. So $T_p(\ol D)\subset L^+$ which completes the proof.
\end{proof}

 We may refer to $T_p f$ as the \emph{tangent half-plane} of $f$ at $p$. Accordingly, the (full) \emph{tangent plane} of $f$ at a point $p\in\p M$ may be defined as
 $$
 H_p(0):=\text{the plane containing $T_p f$}.
 $$
 Note that, by the last observation, $H_p (0)$ is a local support plane of $f$ at $p$. Another useful notion is that of the  \emph{conormal vector} of $f$ at  $p\in\p M$, which is defined as
   $$
 \nu(p):=\text{the unit  normal vector of $\f$ at $p$ which points into $T_pf$.}
 $$

\subsection{Caps at the boundary}\label{sec:boundarycaps}
Using the notion of tangent planes $H_p(0)$ and the conormal vector $\nu(p)$ developed above, we next  show that through each point $p\in\partial M$ there passes a cap. For every $p\in\p M$, let 
$$
C_p(0):=\text{the component of $f^{-1} (H_p(0))$ which contains $p$}.
$$
When $C_p(0)$ is a cap, we call it the \emph{tangential cap} of $M$ at $p$. Let $N(p)$ denote the principal normal of $f|_{\partial M}$ at $p$. The main result of this section is as follows:

\begin{prop}\label{prop:boundarycap}
Through each point $p\in \p M$ there passes a  cap. In particular, $C_p(0)$ is a cap
whenever $\nu(p)\neq -N(p)$.
\end{prop}

First we require the following simple observation:

\begin{lem}\label{lem:TpC}
Let $C$ be a nondegenerate cap in $\R^3$ with plane $H$,  $D$ be the region in $H$ bounded by $\partial C$,  and $p\in\partial C$. Then $T_p C\cap\inte(D)=\emptyset$.
\end{lem}
\begin{proof}
Let $q\in\inte(D)$. Then there exists a ball $B\subset\R^3$ centered at $q$ which is disjoint from $C$. Let $A:=\conv(\{p\}\cup B)$,  or the union of all line segments with an end point at $p$ and the other in $B$, see Figure \ref{fig:capcone}. 
\begin{figure}[h]
\begin{overpic}[height=0.9in]{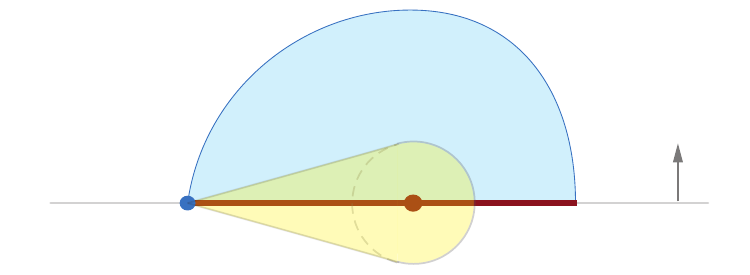}
\put(54,6){\Small $q$}
\put(24,6){\Small $p$}
\put(54,25){\Small $C$}
\put(53.5,13){\Small $B$}
\put(40,11){\Small $A$}
\put(69,5){\Small $D$}
\put(96,8){\Small $H$}
\put(85,18){\Small $H^+$}
\end{overpic}
\caption{}
\label{fig:capcone}
\end{figure}
To show that $q\not\in T_p C$, it suffices to check that 
$
A\cap C=\{p\}, 
$
 as we had mentioned in the proof of Proposition \ref{prop:halfplane}. To this end note that $A\cap\inte(H^-)$  is disjoint from $C$, since  $C\subset H^+$. Further $A\cap \inte(H^+)$ lies in the interior of $K:=\conv(C)$, since $B\cap\inte(H^+)\subset\inte(K)$, and $K$ is convex.
So $A\cap\inte(H^+)$ is disjoint from $C$, since $C\subset\partial K$.
 Finally, $A\cap H-\{p\}$ lies in  $\inte(D)$, since $B\cap D\subset\inte (D)$, and thus is disjoint from $C$ as well. 
\end{proof}

The last lemma yields the following observation. Note that the point $p$ below may be an end point of the segment $I$.

\begin{lem}\label{lem:nuN}
Let $p\in\p M$ and $H=H_p(0)$. Suppose that there exists a segment $I\subset\p M$ containing $p$ such that $f(I)\subset H$. Then
\begin{enumerate}
\item[(i)]{There exists a neighborhood $V$ of $p$ in $M$ such that $f(V-\p M)\cap H=\emptyset$.}
\item[(ii)]{$\nu(p)=-N(p)$}
\end{enumerate}
\end{lem}
\begin{proof}
After replacing $I$ with a subsegment containing $p$, 
we may assume that $I$  lies in a convex neighborhood $U$ of $M$ with associated  body $K$. Since $f\colon\cl(U)\to f(\cl(U))\subset\partial K$ is a homeomorphism, we may identify $\cl(U)$ with its image and suppress $f$ henceforth. Recall that $H$  supports $U$ by Lemma \ref{prop:halfplane}. So $H$ supports $K$, since $K=\conv(U)$. Thus $D:=H\cap \p K=H\cap K$ is  convex. Further $D$ has interior points in $H$ since it contains $I$, which is not a line segment, since by assumption $\f$ has no inflections. Thus $D$ is a convex disk, see Figure \ref{fig:flatbottom}.
\begin{figure}[h]
\begin{overpic}[height=1.0in]{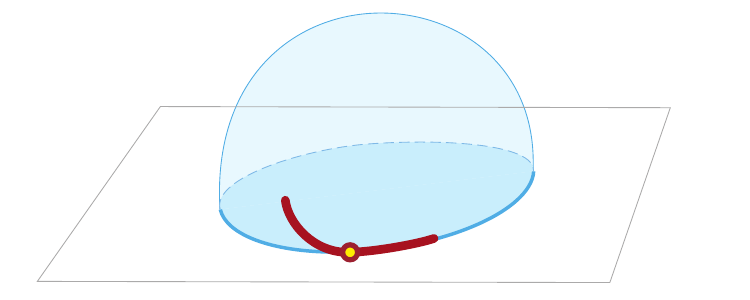}
\put(45,2){\Small $p$}
\put(34,13){\Small $J$}
\put(47,29){\Small $K$}
\put(68,32){\Small $C$}
\put(47,14){\Small $D$}
\put(75,5){\Small $H$}
\end{overpic}
\caption{}
\label{fig:flatbottom}
\end{figure}

Let $H^+$ be the side of $H$ which contains $K$, and set $C:=\cl(\p K\cap\inte(H^+))$. Then $C$ is a convex cap, since $H\cap K=D$ has interior points in $H$. Let 
$J:=\cl(U)\cap\p M$. Then $I\subset J$. Also recall that $J$ is a segment by the definition of $U$. 
 Next note that if there is a point $p'\in J\cap\inte(D)$, then,  $f^{-1}(\inte(D))$ is a neighborhood of $p'$ in $M$ which gets mapped to $D\subset H$,  and thus violates the local nonflatness assumption. So $J\subset C$. 
This, together with the local non flatness assumption, yields that there exists  a neighborhood $V$ of $p$ in $U$ such that $V-\p M=V-J$ is disjoint from $H$. In particular, $V\subset C$.
Consequently, 
$$
\nu(p)\in T_pf=T_{p}V\subset T_{p} C.
$$
Further recall that $\nu(p)\in T_pf\subset H$. 
 So $\nu(p)$ lies in $H$ and points outside $D$,  by Lemma \ref{lem:TpC}. 
 On the other hand, since $I\subset J\subset C$, and $I\subset H$, it follows that $I\subset C\cap H=\p C=\p D$.
 Thus $N(p)$ lies in $H$ and points inside $D$. So we conclude that $\nu(p)=-N(p)$ as desired.
  \end{proof}

Now we can prove the main result of this section:

\begin{proof}[Proof of Proposition \ref{prop:boundarycap}]
By Lemma \ref{prop:halfplane}, $H:=H_p(0)$ locally supports $f$ at $p$. Thus, by Proposition \ref{thm:flatcap}, $C:=C_p(0)$ is a cap as soon as it contains an interior point of $M$, in which case we are done. We may suppose then that 
$
C\subset\p M.
$
 Then $C$ is either a single point or else is a segment containing $p$. In the former case again we are done. So  we may assume  that $C$ is a segment of $\partial M$.  Then, by Lemma \ref{lem:nuN}, 
 $$
 \nu(p)=-N(p)
 $$
  and $f(V-\p M)$ is disjoint from $H$ for a neighborhood $V$ of $p$ in $M$.
In particular we have shown that $C\cap\inte(M)=\emptyset$ only if the above equality holds. Consequently, if $\nu(p)\neq-N(p)$, then $C$ is a cap, as claimed.

Next we are going to show that we may rotate $H$ about the tangent line of $f$ at $p$ so as to produce a nondegenerate cap at $p$ which will finish the proof.
To this end,
let $U\subset V$ be a convex neighborhood of $p$ with associated body $K$. We may assume that $I\subset U$, after replacing $I$ with a subsegment containing $p$. Further, as in the proof of  Lemma \ref{lem:nuN}, we may identify $\cl(U)$ with $f(\cl(U))$ and suppress $f$ henceforth.
 We will show that $U$ contains a cap  passing through $p$ as follows.

 Let $\Gamma:=\cl(U)\cap \p M$, and  $\tilde \Gamma$ denote the projection of $\Gamma$ into $H$. Further let $L^+$ be the half plane of $H$ determined by $L$  into which $\nu(p)$ points. Then, since $N(p)=-\nu(p)$, $N(p)$ points into the opposite half plane $L^-$, see Figure \ref{fig:halfplanes}. 
  \begin{figure}[h]
\begin{overpic}[height=1.0in]{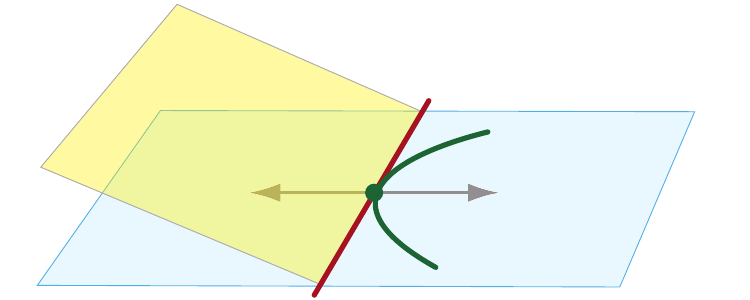}
\put(24,13){\Small $\nu(p)$}
\put(68,13){\Small $N(p)$}
\put(66,20){\Small $\tilde\Gamma$}
\put(23,32){\Small $\tilde L^+$}
\put(47,17){\Small $p$}
\put(76,4){\Small $L^-$}
\put(10,4){\Small $L^+$}
\put(39,-3){\Small $L$}
\end{overpic}
\caption{}
\label{fig:halfplanes}
\end{figure}
 Consequently, $\tilde \Gamma\subset L^-$ and $\tilde \Gamma\cap L^+=\{p\}$, assuming that $U$  is sufficiently small. This yields in particular that $\Gamma\cap L^+=\{p\}$. So it follows that
$$
L^+\cap\p U=\{p\},
$$
since  $\p U-\Gamma\subset U-\p M\subset V-\p M$  is disjoint from $H$. In particular  $\Gamma':=\cl(\p U-\Gamma)$  is disjoint from $L^+$, since $p\in \inte(\Gamma)$. Since $\Gamma'$ is compact, we may then rotate $L^+$ about $L$ into the side of $H$ containing $U$ so as to obtain a half plane $\tilde L^+$ bounded by $L$  such that $\tilde L^+\cap\Gamma'=\emptyset$. Further, since $\tilde \Gamma\subset L^-$, $\Gamma$ lies in the half-space over $L^-$. Thus  we may assume that $\tilde L^+\cap \Gamma=\{p\}$. So
 $$
 \tilde L^+\cap\p U=(\tilde L^+\cap \Gamma)\cup(\tilde L^+\cap\Gamma')=\{p\}\cup\emptyset=\{p\}.
 $$
Consequently, $\p U$ is contained in the wedge formed by $\tilde L^+$ and $L^-$.
In particular, if $\tilde H$ is the plane containing $\tilde L^+$, then $\p U$ lies on one side of $\tilde H$, say 
$$
\p U\subset\tilde H^-.
$$
 In that case $\nu(p)$ points into $\inte(\tilde H^+)$. Thus
 $$
U\cap\inte(\tilde H^+)\neq\emptyset.
 $$  
So $\cl(U\cap \inte(\tilde H^+))=\cl(\p K\cap \inte(\tilde H^+))$ is the desired cap.
\end{proof}

\subsection{Angle of caps}\label{subsec:maxcap}

Here we will prove the main result of Section \ref{sec:maxcap}: 

\begin{thm}\label{thm:maxcap}
Through each point $p\in \p M$ there passes a maximal cap.
\end{thm}

First let us record the following elementary observation:

 \begin{lem}\label{lem:nN}
Let $H$ be a local support plane of $f$ at  $p\in \p M$, $n$ be the outward normal to $H$, and $N$ be the principal normal of $\f$ at $p$. Then $\l n, N\r \leq 0$.
\end{lem}
\begin{proof}
Let $\phi\colon (-\epsilon, \epsilon)\to\p M$ be a local unit speed parametrization, with $\phi(0)=p$, and set $\gamma:=f\circ\phi$. Further let
 $h(t):=\l\gamma(t)-f(p),n\r$. Then  $h$ has a local maximum  at $0$, since $n$ is the outward normal of $H$. Consequently $0\geq h''(0)=\l N,n\r\|\gamma''(0)\|$, which completes the proof.
\end{proof}

Using the last observation, we next show that no degenerate cap of $M$ at $p$ can be larger than the tangential cap $C_p(0)$:

\begin{lem}\label{lem:C0}
Suppose that $M$ has a degenerate cap $C$ at $p\in\p M$. Then $C_p(0)$ is also a cap  and $C\subset C_p(0)$.
\end{lem}
\begin{proof}
First suppose that $\nu(p)=-N(p)$. Let $H$ be a local support plane of $f$ at $p$ with outward normal $n$. Then, by Lemma \ref{lem:nN}, $\l n,N(p)\r\leq 0$. On the other hand, by assumption $f(U)\subset H^-$, for some  neighborhood $U$ of $p$ in $M$. Consequently,  $\nu(p)\in T_p f\subset H^-$. Thus 
$
0\geq\l n,\nu(p)\r=-\l n,N(p)\r\geq 0.
$
 So $\l n,\nu(p)\r=0$, which yields that $H=H_p(0)$. So $H_p(0)$ is the only local support plane of $f$ at $p$. Now recall that the plane of any degenerate cap at $p$ locally supports $f$ at $p$, by the extension property. Thus $H_p(0)$ is the plane of $C$.  Consequently $C=C_p(0)$ by the uniqueness property of caps (Proposition \ref{prop:capunique}).
In particular $C_p(0)$ is a cap.

It remains then to consider the case where $\nu(p)\neq -N(p)$. In this case $C_p(0)$ is a cap by Proposition \ref{prop:boundarycap}. So we just need to verify that $C\subset C_p(0)$. To this end, it suffices to show that $f(C)\subset H_p(0)$, because by assumption $C$ is a connected set containing $p$, and $C_p(0)$ is by definition the component of $f^{-1}(H_p(0))$ containing $p$.
To see  that $f(C)\subset H_p(0)$,
let $H$ be a plane of $C$. Then $H$ is a local support plane of $f$ at $p$ (again by the extension property). In particular, it follows that $H$ locally supports $\f$ at $p$ and therefore contains the tangent line $L$ of $\f$ at $p$. So if $f(C)\subset L$, then there is nothing to prove. 

Suppose then that there is a point $q\in C$ such that $f(q)\not\in L$. Let $r$ be the ray which emanates from $f(p)$ and passes through $f(q)$. Since $f(C)$ is convex, the segment of $r$ between $f(p)$ and $f(q)$ lies in $f(C)$. It follows then that if $U$ is any neighborhood of $p$ in $M$, then $f(U)$ contains a subsegment of $r$ with an end point at $f(p)$. Thus $r\in T_p f$, which yields that $f(q)\in T_pf$. So, since $f(q)\not\in L$, it follows that $T_pf\subset H$, and thus $H=H_p(0)$. So $f(C)\subset H_p(0)$ as desired.
\end{proof}

The following lemma collects some basic facts for easy reference:

\begin{lem}\label{lem:TpfH+}
Let $C$ be a nondegenerate cap of $M$ at $p\in\partial M$, $H$ be the plane of $C$,  and $L$ be the tangent line of $\f$ at $p$. Then $L\subset H$,  $T_p f\subset H^+$,  $T_p f\cap H=L$, and $f(\partial C)$ lies on one side of $L$.
\end{lem}
\begin{proof}
It follows from the extension property (Proposition \ref{prop:capextend}) that $H$ locally supports $\f$ at $p$. Thus $L\subset H$. 

Recall that $T_p f$ is a half-plane bounded by $L$.  Thus to show that $T_p f\subset H^+$ it suffices to check that a point of  $T_p f-L$  lies in $H^+$. To  this end note that $T_{f(p)} f(C)\subset T_p f$ by the definition of $T_p f$, and $T_{f(p)} f(C)\subset H^+$ since $f(C)\subset H^+$. So it suffices to check that $T_{f(p)} f(C)\not\subset L$, which is indeed the case since $f(C)$ is nondegenerate. So $T_p f\subset H^+$.

If $T_p f\cap H\not\subset L$, then $T_p f \subset H$. This in turn yields that 
$T_{f(p)} f(C)\subset H$, which again is not the case since $f(C)$ is nondegenerate. So  $T_p f\cap H=L$.

Recall that $H_p(0)$ locally supports $f$ at $p$. Thus $H_p(0)$ locally supports $f(\p C)$, which in turn yields that $H_p(0)$ globally supports $f(\p C)$, since $f(\p C)$ is a convex curve. Further, since $T_p f\cap H=L$,  we have $H_p(0)\cap H=L$. Thus, $L$ supports $f(\partial C)$.
\end{proof}

The last two observations now yield the following basic fact:

\begin{lem}\label{lem:nestedC0}
Suppose that $M$ has a nondegenerate cap $C$ at $p\in\p M$. Then $C$ contains  all degenerate caps of $M$ at $p$.
\end{lem}
\begin{proof}
If $M$ has a degenerate cap at $p$, then, 
by Lemma \ref{lem:C0}, $C':=C_p(0)$  contains all the degenerate caps of $M$ at $p$. Thus it suffices to show that $C'\subset C$. 
To see this let $H^+$ be the half-space of $C$. Note that $T_pf\subset H^+$, by Lemma \ref{lem:TpfH+}. Thus $T_p f(C')\subset H^+$, since $T_p f(C')\subset T_pf$. Further note that since $C'$ is degenerate, $f(C')\subset T_p f(C')$. Thus  $f(C')\subset H^+$. Now recall that, by Proposition \ref{prop:capextend}, $C$ is the component of $f^{-1}(H^+)$ which contains $p$. Thus, since $C'$ is connected, $p\in C'$, and $C'\subset f^{-1}(H^+)$, it follows that $C'\subset C$ as desired.
\end{proof}

 Let $C$ be a nondegenerate cap of $M$ at  $p\in\p M$, $H$ be the plane of $C$ and $L$ be the tangent line of $\f$ at $p$. Then, by Lemma \ref{lem:TpfH+}, $f(\p C)$ lies on one side of $L$ in $H$, which we denote by $L^+$. We  define the \emph{angle} $0\leq \theta\leq\pi$ of $C$ at $p$ as the angle between the half planes $T_p f$ and $L^+$, see Figure \ref{fig:angle}.  Note that one of these half-planes extends to the tangent plane $H_p(0)$, while the other  lies on the plane of $C$. Thus  the angle $\theta$ of a nondegenerate cap uniquely determines its plane. Consequently, the uniqueness property (Proposition \ref{prop:capunique}) immediately yields that
for each angle $0<\theta\leq \pi$ there exists at most one nondegenerate cap of $M$ at $p$.

 \begin{figure}[h]
\begin{overpic}[height=1in]{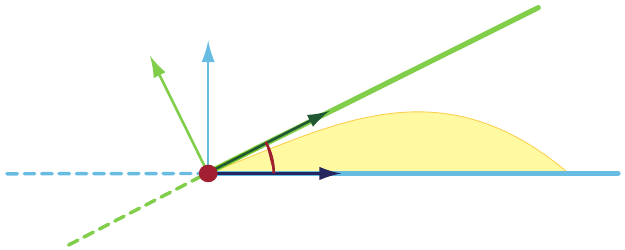}
\put(42,24.5){\Small $\nu_p(0)$}
\put(51,7){\Small $\nu_p(\theta)$}
\put(58,17){\Small $f(C_p(\theta))$}
\put(31,6){\Small $f(p)$}
\put(46,14){\Small $\theta$}
\put(32,35){\Small $n_p(\theta)$}
\put(16,33){\Small $n_p(0)$}
\put(100,11){\Small $L^+$}
\put(88,37){\Small $T_p f$}
\put(-14,11){\Small $H_p(\theta)$}
\put(-4,0){\Small $H_p(0)$}
\end{overpic}
\caption{}
\label{fig:angle}
\end{figure}

 For $0\leq\theta\leq\pi$, let $\nu_p(\theta)$ be the unit normal to $\f$ at $p$ which points into $H_p(0)^-$, and makes an angle of $\theta$ with $\nu(p)=\nu_p(0)$,  the conormal of $f$ at $p$. Next let $H_p(\theta)$ be the plane which contains $L$ and $\nu_p(\theta)$. Further, let $n_p(0)$ 
 be the outward normal of $H_p(0)$, and $n_p(\theta)$ be the normal  to $H_p(\theta)$ at $p$ so that the frame $(\nu_p(\theta), n_p(\theta))$ induces the same orientation on the orthogonal plane $\Pi$ of $\f$ at $p$, as does $(\nu_p(0), n_p(0))$. Then we declare $n_p(\theta)$ to be the outward normal of $H_p(\theta)$. Accordingly, $H_p(\theta)^+$ will denote the side where $n_p(\theta)$ points. Finally, we set
 $$
 C_p(\theta):=\text{the component of $f^{-1}(H_p(\theta)^+)$ which contains $p$}.
 $$
 Note that if $C_p(\theta)$ is a cap, then its angle is $\theta$ by construction. Let us then summarize these observations as follows:
 
 \begin{lem}\label{lem:Cptheta}
 Let $p\in\p M$. Then for every $0\leq\theta\leq\pi$ there exists at most one cap  with angle $\theta$ at $p$. This cap, if it exists, coincides with $C_p(\theta)$.\qed
 \end{lem}
 
 It now quickly follows that:
 
\begin{lem}\label{lem:nestedcaps}
Let $p\in\p M$, and suppose that 
$C_p(\theta_0)$ is a cap, for some $0\leq \theta_0\leq\pi$. Then so is $C_p(\theta)$ for all $0\leq\theta< \theta_0$, and $C_p(\theta)\varsubsetneq C_p(\theta_0)$.
\end{lem}
\begin{proof}
Let us fix $p$ and drop all references to it for convenience. 
If $\theta_0=0$, then there is nothing to prove. So suppose that $\theta_0>0$. Now if $\theta=0$, again we are done by Lemma \ref{lem:nestedC0}, so we may assume that $\theta>0$ as well. 
Then, since $0<\theta <\theta_0$, it follows that $H(\theta)$ intersects the interior of $\conv(f(C(\theta_0)))$. Further
$f(\p C(\theta_0))$ lies in $H(\theta)^-$. Thus $H(\theta)^+\cap f(C(\theta_0))$ is a cap in $\R^3$, which in turn yields that $f^{-1}(H(\theta)^+)\cap C(\theta_0)$  is a cap 
of angle $\theta$ which is contained in $C(\theta_0)$. Since this cap has angle $\theta$,  it must coincide with $C(\theta)$ by Lemma \ref{lem:Cptheta}. Thus $C(\theta)\subset C(\theta_0)$. Finally,  since $\theta<\theta_0$, we have $H(\theta)\neq H(\theta_0)$, i.e., $C(\theta)$ and $C(\theta_0)$ have different planes. So, since nondegenerate caps have  unique planes, it follows  that $C(\theta)\neq C(\theta_0)$, which completes the proof.
\end{proof}

Now we are ready to prove the main result of this section:     

\begin{proof}[Proof of Theorem \ref{thm:maxcap}]
Let $\Theta$  be the set  of all angles $0\leq\theta\leq\pi$ such that $C_p(\theta)$ is a cap. By Proposition \ref{prop:boundarycap}, $\Theta\neq\emptyset$. So the supremum $\ol\theta$ of $\Theta$ exists. 
First suppose that $\ol\theta=0$. Then $\Theta=\{0\}$. 
So $C_p(0)$ is a cap. Furthermore, all caps of $M$ at $p$ must be degenerate, since nondegenerate caps have positive angles. But $C_p(0)$ contains all degenerate caps of $M$ at $p$, by Lemma \ref{lem:C0}. Thus $C_p(0)$ is the desired maximal cap.
Next assume that $\ol\theta>0$. Then, by Lemma \ref{lem:nestedcaps}, $[0,\ol\theta)\subset\Theta$.
Let $\theta_i\in(0,\ol\theta)$ be an increasing sequence  converging to $\ol\theta$. Then $C_i:=C(\theta_i)$ is a nested sequence of caps by Lemma \ref{lem:nestedcaps}. Hence, by Proposition \ref{prop:converge}, there exists a cap $C$ in $M$ which contains all $C_i$. Then, by Lemma \ref{lem:nestedcaps}, the angle of $C$ cannot be smaller than any $\theta_i$, and therefore must be equal to $\ol\theta$. This in turn yields that $C$ contains all caps at $p$, again by Lemma \ref{lem:nestedcaps}, which completes the proof.
\end{proof}

\section{Proof of the Main Theorem}
Here we complete the proof of the main result of this work, Theorem \ref{thm:main2} below, which is a refinement of Theorem \ref{thm:main} mentioned in the introduction. To this end we first show in Section \ref{sec:singularcap} that the maximal caps, whose existence were finally established in the last section, induce a nested partition on $\partial M$. This will ensure the  existence of singular maximal caps, i.e., maximal caps which intersect $\partial M$ along connected sets. Next we show that the planes of these singular caps are osculating planes of $\f$ in Section \ref{subsec:osculateplane}.  Further we show in Section \ref{subsec:orientexpress} that $\f$ may be oriented so that each singular maximal cap lies on the side of the osculating plane where the binormal vector $B(p)$ points. This will  ensure,  as we will show in Section \ref{subsec:torsion}, that   torsion changes sign in a consistent way near the points where each singular maximal cap intersects $\partial M$, which quickly completes the proof.

\subsection{Singular maximal caps}\label{sec:singularcap}
A cap $C$ of $M$ has \emph{rank} $n$ provided that $C\cap\p M$ has $n$ components.
We say that a maximal cap $C\subset M$ is \emph{singular} provided that $\mathrm{rank}(C)=1$.  For each $p\in\p M$, let $C_p$ be the corresponding maximal cap (which  exists by Theorem \ref{thm:maxcap}). Set
$$
 [p]:=C_p\cap\p M,\quad\text{and}\quad\P:=\{[p]\mid p\in \p M\}.
$$
Let $\mathcal{S}$ be the number of singular parts of $\P$, and $\mathcal{T}$ be the triangularity of $\P$ as defined by \eqref{eq:T}.
In this section we establish the following result, which yields a Bose type formula for maximal caps of $M$.

\begin{thm}\label{thm:singularcap}
Suppose that $M$ is not a cap. Then $\P$ is a nontrivial nested partition of $M$. In particular, if $C$ is a maximal cap of $M$, then each component of $\partial M-C$ has a point where the corresponding maximal cap is singular. Thus
\begin{equation}\label{eq:singularcap}
\mathcal{S}\geq \mathcal{T}+2.
\end{equation}
\end{thm}

To prove the above theorem, we need to record a pair of simple observations:

\begin{lem}\label{lem:maxcapintersect}
Distinct maximal caps of $M$  may intersect only in the interior of $M$.
\end{lem}
\begin{proof}
Suppose that $C_p$, $C_q$ are maximal caps which intersect at  $x\in\p M$. Then $C_p$, $C_q\subset C_x$, the maximal cap of $M$ at $x$. In particular $p$, $q\in C_x$, which implies that $C_x\subset C_p$, $C_q$, by the maximality of $C_p$, $C_q$. Thus $C_p=C_x=C_q$.
\end{proof}

Note that by interior in the next observation, we mean the interior of a cap $C$ as a manifold, as opposed to a subset of $M$. Thus for instance if $C$ is a segment, then its interior is homeomorphic to the open interval $(0,1)$.

\begin{lem}\label{lem:C1C2}
Let $C_1$, $C_2$ be  maximal caps of $M$ which are degenerate. Suppose that $\inte(C_1)\cap\inte(C_2)\neq\emptyset$.  Then $C_1=C_2$.
\end{lem}
\begin{proof}
Let $p\in\inte(C_1)\cap\inte(C_2)$. Then $p\in\inte(M)$,
by Lemma \ref{lem:maxcapintersect}. Let $H$ be a local support plane of $f$ at $p$. Then $H$ locally supports $\ol{C_i}:=f(C_i)$  at $\ol p:=f(p)$, for $i=1$, $2$. Further $\ol p\in\inte(\ol{C_i})$, since $f\colon C_i\to\ol C_i$ is a homeomorphism. It follows then that  $\ol{C_i}\subset H$, since $\ol{C_i}$ are planar convex sets. Let $A$ be the component of $f^{-1}(H)$ containing $p$. Then $A$ is a  cap by Theorem \ref{thm:flatcap}. Since $C_i$ are connected, $C_i\subset f^{-1}(H)$, and $p\in C_i$,  it follows that $C_i\subset A$. 
Thus $C_1=A=C_2$, since $C_i$ are maximal.
\end{proof}

We are now ready to prove the main result of this section:

\begin{proof}[Proof of Theorem \ref{thm:singularcap}]
We just need to check that $\P$ is a nested partition; the other claims then follow immediately from Lemma \ref{lem:partition}.
First we check that $\P$ is  a partition, i.e., parts of $\P$ cover $\p M$ and are pairwise disjoint. The former property is clear since $p\in[p]$. To see the latter property, suppose  that $q\in[p]$. Then  $q\in C_p$. Thus, by maximality, $C_p\subset C_q$. In particular, $p\in C_q$. So, again by maximality, $C_q\subset C_p$. Thus $C_q=C_p$, which in turn yields $[q]=[p]$ as desired.

It remains to establish that $\P$  is nested.  Suppose, towards a contradiction, that it is not. Then there are maximal caps $C_i$, $i=1$, $2$,  such $C_2$ intersects different components of $\p M-C_1$. Then  $C_i$  cannot
be a point, and thus  each is either homeomorphic to a disk, or to a line segment.  Further, recall that $C_i$  may not intersect each other on $\p M$, by Lemma \ref{lem:maxcapintersect}. Thus, since $M$ is simply connected,  $\inte(C_i)$ must intersect at a point $p$. 

If $\inte(C_i)$ intersect, then, by  Lemma \ref{lem:C1C2}, $C_i$  cannot both be degenerate.
We may suppose then that $C_1$ is nondegenerate. In particular, $\p C_1$ is a simple closed curve. Then, since $p\in\inte(C_1)$,  it follows that $C_2$ intersects $\p C_1$ at a pair of points $x$, $y$ which are separated in $\p C_1$ by points of $\p M$. There are two cases to consider: either $C_2$ is homeomorphic to a line segments or a disk. For any set $X\subset M$, set $\ol X:=f(X)$.

(\emph{Case 1}) First suppose that $C_2$ is homeomorphic to a line segment. Then $\ol C_2$ is a line segment. Also note that  $\ol C_2$ passes through the points $\ol x$ and $\ol y$ which belong to $\ol{\p C_1}$, and thus lies in the plane $H_1$ of $\ol C_1$. So $\ol C_2\subset H_1$, which in turn yields that $\ol p\in H_1$.  
But  $p\in\inte(C_1)$, and thus $\ol p\in\inte(\ol C_1)$, which is disjoint from $H_1$. So we arrive at a contradiction.

(\emph{Case 2}) Next suppose that $C_2$ is a topological disk. Then $C_2$ has a unique plane $H_2$. If $H_1=H_2$ then we are done by the uniqueness property. So suppose that $H_1\neq H_2$. Note that $H_1$, $H_2$ cannot be parallel, for then it would follow that $\ol C_1\subset \ol C_2$ or $\ol C_2\subset \ol C_2$ which is not possible by maximality. So $H_1\cap H_2= L$, where $L$ is a line. Consequently $\ol{\p C_1}\cap\ol{\p C_2}\subset (\ol{ \p C_1}\cap H_1)\cap H_2=\ol{\p C_1}\cap L$. Since $\ol{\p C_1}$ is a convex planar curve it follows that $\ol{\p C_1}\cap L$  can have at most two components.  Thus $\ol{ \p C_1}\cap\ol{\p C_2}$ can have at most two components, which implies the same for $\p C_1\cap \p C_2$.
On the other hand, $\p C_i$ are simple closed curves in $M$, and by assumption 
$\p C_2$ intersects different components of $\p M-\p C_1$.
 Thus, just as we had argued in Example \ref{ex:partition}, since $M$ is simply connected, 
 $\p C_1\cap \p C_2$ must have at least four components. So again we arrive at a contradiction, which completes the proof.
\end{proof}

\begin{note}\label{note:umehara2}
It is likely the case that the partition $\P$ in Theorem \ref{thm:singularcap} is actually an \emph{intrinsic circle system} as defined by Umehara \cite{umehara2}, see Note \ref{note:umehara}. In the present context, this means that the maximal caps of $M$ satisfy the following lower semicontinuity property: Let $p_i\in\p M$ be a sequence  converging to $p\in\p M$, $C_i:=C_{p_i}$ be the corresponding maximal caps of $M$,  $q_i\in C_i$, and $q$ be a limit point of $q_i$; then $q\in C_{p}$. 
If this is indeed the case, then it would follow that $\P$ is an intrinsic circle system, since we have already established in Theorem \ref{thm:singularcap} that $\P$ is nested. Consequently,  Umehara's generalization of Bose's formula \cite[Thm. 2.7]{umehara2} would yield  that whenever $\mathcal{S}<\infty$, 
equality holds in \eqref{eq:singularcap}.
\end{note}

\subsection{Osculating planes}\label{subsec:osculateplane}
The \emph{osculating plane} of $\f$ at a point $p$ is the plane which is tangent to $\f$ at $p$, and contains the principal normal $N(p)$. In this section we show that if $C$ is a singular maximal cap of $M$ at $p$, then the plane $H$ of $C$ is the osculating plane of $\f$ at $p$. Note that $H$ is unique when $C$ is nondegenerate. Further, by Lemma \ref{lem:C0}, when
 $C$ is  degenerate, $C=C_p(0)$, the tangential cap  defined in Section \ref{sec:boundarycaps}. In particular, $H_p(0)$ is a plane of $C$, which is what we mean by \emph{the} plane $H$ of $C$.  In short, singular maximal caps have a well-defined plane.

\begin{prop}\label{prop:osculate}
Let $p\in\p M$, $C$ be a singular maximal  cap of $M$ at $p$, and $H$ be the plane of $C$. Then $H$ is the osculating plane of $\f$ at $p$. 
\end{prop}

First we need to record the following three basic observations:

\begin{lem}\label{lem:GammaNn}
Let $\Gamma\subset\R^3$ be a compact embedded $\C^2$ curve,  $p$ be an interior point of $\Gamma$, $N_p\Gamma$ be the space of unit normal vectors of $\Gamma$ at $p$, and $A\subset N_p\Gamma$ consist of those normals $n$ such that (i) $n$ strictly supports $\Gamma$ at $p$, and (ii) $\l n, N(p)\r <0$. Then $A$ 
is open in $N_p\Gamma$.
\end{lem}
\begin{proof}
Let $n\in A$, $\epsilon>0$, and $n_i\in N_p\Gamma$, $i=0$, $1$, be  distinct normals such that $\l n, n_i\r=1-\epsilon$. For $\lambda\in[0,1]$, set
$$
n_\lambda:=\frac{\lambda n_1 +(1-\lambda)n_0}{\|\lambda n_1 +(1-\lambda)n_0\|}.
$$  
These vectors generate a segment $I_\epsilon$ of $N_p\Gamma$  centered at $n=n_{1/2}$, see Figure \ref{fig:normals}.
\begin{figure}[h]
\begin{overpic}[height=1in]{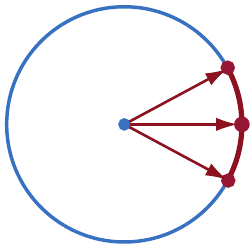}
\put(-7,5){\Small $N_p\Gamma$}
\put(30,35){\Small $f(p)$}
\put(102,49){\Small $n$}
\put(95,72){\Small $n_1$}
\put(95,23){\Small $n_0$}
\end{overpic}
\caption{}
\label{fig:normals}
\end{figure}
Note that $\length(I_\epsilon)\to 0$ as $\epsilon\to 0$. 
We claim that there exists an $\epsilon$ such that $I_\epsilon\subset A$, which will finish the proof.
To this end let $\gamma\colon(-\delta,\delta)\to\Gamma$, $\gamma(0)=p$, be  a local unit speed parametrization, and set 
$$
h_\lambda(t):=\l\gamma(t)-\gamma(0),n_\lambda\r.
$$
 Then $h'_\lambda(0)=\l T(p),n_\lambda\r=0$. Further  $h_\lambda''(0)=\l\kappa(p)N(p),n_\lambda\r$. Thus $h_{1/2}''(0)<0$, since $n_{1/2}=n\in A$, and so $\l N(p),n_{1/2}\r<0$. Consequently, if $\epsilon$ is sufficiently small, then  $h_{i}''(0)<0$.
So, by Taylor's theorem, there exist $0<\delta_i\leq\delta$ such that 
$$
h_i< 0
$$
 on $(-\delta_i,\delta_i)-\{0\}$. Let $\ol\delta:=\min\{\delta_0,\delta_1\}$. Then, for every $t\in (-\ol\delta,\ol\delta)-\{0\}$,
 $$
 h_\lambda(t)=\frac{\lambda h_1(t) +(1-\lambda)h_0(t)}{\|\lambda n_1 +(1-\lambda)n_0\|}<0.
 $$
 So we conclude that  $n_\lambda$ is a strict support vector of $U:=\gamma((-\ol\delta,\ol\delta))$, which is a neighborhood of $p$ in $\Gamma$. Let $H_\lambda$ be the plane which passes through $f(p)$ and is orthogonal to $n_\lambda(p)$.
Then $\Gamma-U$ lies on one side of $H_{1/2}$ and is disjoint from it by assumption. Thus since $\Gamma-U$ is compact, the same holds for $H_\lambda$ provided that $\epsilon$ is sufficiently small. Hence  $n_\lambda$ is a strict support vector of $\Gamma$. In addition, since by assumption $\l n_{1/2},N(p)\r<0$, we can  make sure that $\epsilon$ is so small that $\l n_\lambda,N(p)\r<0$. Thus $I_\epsilon\subset A$, as desired.
\end{proof}

\begin{lem}\label{lem:capPi}
Let $p\in\p M$, $C$ be a nondegenerate cap of $M$ at $p$, and $H$ be the plane of $C$. Suppose that the angle of $C$ at $p$ is $\pi$. Then $H$ is the osculating plane of $\f$ at $p$.
\end{lem}
\begin{proof}
 Since  the angle of $C$ is $\pi$, $T_p f\subset H$, which in turn yields that $H=H_p(0)$. Consequently the outward normal $n$ of $H_p(0)$  is orthogonal to $H$. Furthermore, since $f(C)\subset H^+$,  $n$ points into $H^-$. Consequently  $N(p)$  points into $H^+$, since $\l n, N(p)\r\leq 0$, by Lemma \ref{lem:nN}. On the other hand, since $f(C)\subset H^+$,  $\f$ lies in $H^-$ near $p$, by the extension property. Thus $N(p)$ points into $H^-$. So we conclude that $N(p)$  lies in $H$.  Hence $H$ is the osculating plane.
\end{proof}

\begin{lem}\label{lem:Cp0TpM}
If $C_p(0)$ is a cap of $M$, for some $p\in\p M$, then $f(C_p(0))\subset T_p f$.
\end{lem}
\begin{proof}
Let $C:=C_p(0)$, and $H:=H_p(0)$. By definition $\ol C:=f(C)\subset H$, and $\ol C$ is convex. Further recall that $T_p f$ is one of the half-planes of $H$ determined by the tangent line $L$ of $\f$ at $p$. So we just need to check that a neighborhood of $f(p)$ in $\ol C$ lies on the same side of $L$ where $T_p f$ lies. To see this let $U$ be a convex neighborhood of $p$ in $M$ with body $K$. Then $\ol{U\cap C}:=f(U\cap C)$ is a neighborhood of $f(p)$ in $\ol C$.
Let $W:=T_{f(p)}K$, and recall that $T_pf=W\cap H$. Thus $\ol{U\cap C}=\ol{U\cap C}\cap H\subset K\cap H\subset W\cap H=T_p f$, as desired.   
\end{proof}

Now we are ready to prove the main result of this section:

\begin{proof}[Proof of Proposition \ref{prop:osculate}]
Suppose that $H$ is not the osculating plane. Then we show that $C$ is not maximal. Note that $C\cap \p M$ is connected by assumption. Thus if $C\cap\p M\neq\{p\}$, then there is a segment $I\subset \p M$ containing $p$ such that $f(I)\subset f(\p C)\subset H$. Then $N(p)$ lies in $H$, which is not possible since $H$ is not the osculating plane.
So we may suppose  that 
$$
C\cap\p M=\{p\}.
$$ 
Recall that, by the extension property, $H$ locally supports $\f$ at $p$. Thus, if $n$ is the outward normal of $H$, then $\l n, N(p)\r\leq0$ by Lemma \ref{lem:nN}. Also note that $\l n, N(p)\r\neq 0$, since $H$ is not the osculating plane. So we have
$$
\l n, N(p)\r<0.
$$
By the extension property, there exists a neighborhood $U$ of $C$ in $M$ such that 
$$
f(\cl(U)-C)\subset\inte(H^-).
$$
 After replacing $U$ by a smaller subset,
we may assume that $\Gamma:=\p U$ is a simple closed curve, which coincides with a segment of $\partial M$ near $p$, see Figure \ref{fig:disks}. 
\begin{figure}[h]
\begin{overpic}[height=1.0in]{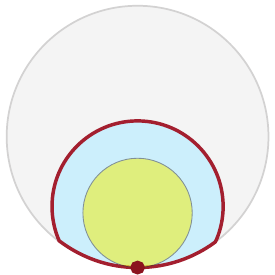}
\put(47,-7){\Small $p$}
\put(45.5,18.5){\Small $C$}
\put(45,44.25){\Small $U$}
\put(13,34){\Small $\Gamma$}
\put(43.75,69){\Small $M$}
\end{overpic}
\caption{}
\label{fig:disks}
\end{figure}
Since $C\cap\p M=\{p\}$, we may assume that $\Gamma\cap C=\{p\}$.
Further, since $f$ is one-to-one on $C$, which is compact, we may suppose that $U$ is so small that $f$ is one-to-one on $\cl(U)$. Consequently $\ol\Gamma:=f(\Gamma)$ will be an embedded curve. 
Now note that, since $\Gamma\cap C=\{p\}$, 
$$
\ol\Gamma-f(p) =f(\Gamma-p)\subset f(\cl(U)-C)\subset \inte(H^-).
$$
Thus, $n$ is a strict support vector of $\ol\Gamma$. Further recall that $\l n, N(p)\r<0$. Thus $n\in A\subset N_p\ol\Gamma$, where $A$ is as in Lemma \ref{lem:GammaNn}.
Consequently, we may rotate $H$ about the tangent line $L$ of $\f$ at $p$ by small angles, in either direction, without compromising the property that $\ol\Gamma\subset H^-$. This will allow us to enlarge $C$ as follows.

(\emph{Case 1}) First suppose that $C$ is nondegenerate. Then the angle $\theta$ of $C$ at $p$ is positive. Further, by Lemma \ref{lem:capPi}, $\theta\neq\pi$. So $0<\theta<\pi$. Recall that $f(\p C)$ lies on one side of $L$, say $L^+$, in $H$ by Lemma \ref{lem:TpfH+}.
 Then $\theta$ is the angle between $L^+$ and $T_{p} M$.  So if  $W$ is the wedge bounded by $L^+$ and $T_p f$, then $W$ is a proper subset of $H^+$. In particular there exists a plane $\tilde H$ 
 which intersects $W$ only along $L$, see Figure \ref{fig:wedge}.
 \begin{figure}[h]
\begin{overpic}[height=1.1in]{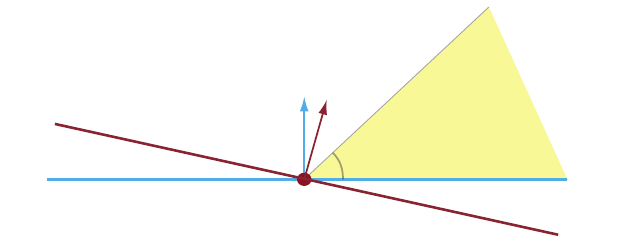}
\put(45,6){\Small $f(p)$}
\put(47,26){\Small $n$}
\put(52,25){\Small $\tilde n$}
\put(55.5,13.5){\Small $\theta$}
\put(4,20){\Small $\tilde H$}
\put(3,10){\Small $H$}
\put(65,16){\Small $W$}
\end{overpic}
\caption{}
\label{fig:wedge}
\end{figure}
 Let $\tilde H^+$ be the side of $\tilde H$ containing $W$, and $\tilde n$ be the unit normal to $\tilde H$ which points into $\tilde H^+$. Then 
 $$
f(C)\subset W\subset \tilde H^+.
 $$
  Further, recall that $\ol\Gamma\subset H^-$. Thus,
by Lemma \ref{lem:GammaNn}, we may assume that  
$$
\ol\Gamma\subset\tilde H^-,
$$
 by choosing $\tilde n$ sufficiently close to $n$.  
Now let 
$$
\tilde  C:=f^{-1}(\tilde H^+)\cap\cl(U).
$$
Then $C\subset \tilde  C$, since $f(C)\subset  \tilde H^+$. Furthermore,  $\tilde  C$ is contained in the region of $M$ bounded by $\Gamma$, and thus $\tilde  C\cap f^{-1}(\inte (\tilde H^+))$ is disjoint from $\p M$. Thus, by  the slicing theorem (Theorem \ref{thm:caps}), $\tilde  C$ is a cap. But note that by construction the angle of $\tilde  C$ at $p$ is strictly bigger than that of $C$. Hence, by Lemma \ref{lem:nestedcaps}, $C$ is a proper subset of $\tilde  C$. So $C$ is not maximal, as claimed.  

(\emph{Case 2}) It remains then to consider the case where $C$ is degenerate. 
Then, by Lemma \ref{lem:C0}, $C=C_p(0)$, since $C$ is maximal.
Let $U$ be a convex neighborhood of $p$ in $M$ with associated body $K$. Recall  that, by Lemma \ref{lem:tancone2}, $W:=T_{f(p)} K$ is a wedge bounded by a pair of half-planes meeting along $L$, one of which is $T_p f$. In particular, the tangent plane $H=H_p(0)$ (which is just the extension of $T_p f$) supports $W$, and $H\cap W=T_p f$. Further, by Lemma \ref{lem:TpfH+}, $f(C)\subset T_p f$.

\begin{figure}[h]
\begin{overpic}[height=1.1in]{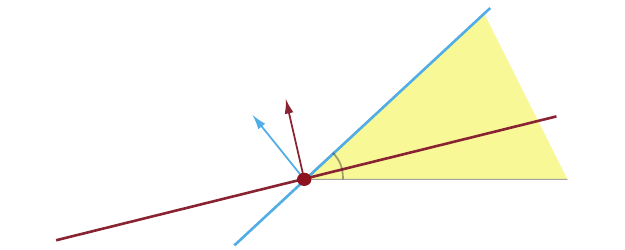}
\put(46,5){\Small $f(p)$}
\put(38,22){\Small $n$}
\put(45,25){\Small $\tilde n$}
\put(31.5,-3){\Small $H_p(0)$}
\put(3,0){\Small $\tilde H$}
\put(66,19){\Small $W$}
\end{overpic}
\caption{}
\label{fig:wedge2}
\end{figure}

Let $\tilde H$ be a plane which contains  $L$ and an interior point of $W$, $\tilde n$ be the unit normal to $\tilde H$ so that $\l n, \tilde n\r>0$, and $\tilde H^+$ be the side of $\tilde H$ into which $\tilde n$ points, see Figure \ref{fig:wedge2}. Then $T_p f\subset \tilde H^+$.
Again, by Lemma \ref{lem:GammaNn}, we may assume that $\ol\Gamma\subset \tilde H^-$, if $\tilde n$ is sufficiently close to $n$. Let $\tilde  C:=f^{-1}(\tilde H^+)\cap \cl(U)$. Then once again $\tilde  C$ will be a cap by the slicing theorem. Further, $C\subset \tilde  C$, since $f(C)\subset T_p f\subset \tilde H^+$, by Lemma \ref{lem:Cp0TpM}. Finally, by construction, the angle of $\tilde  C$  is strictly bigger than that of $C$, which is $0$. Thus, again by Lemma \ref{lem:nestedcaps}, $C$ is a proper subset of $\tilde  C$, and therefore is not maximal.
 \end{proof}

\subsection{Orientation}\label{subsec:orientexpress}
An \emph{orientation} of $\p M$ is the choice of a continuous nonvanishing tangent vector field $p\mapsto u_p$ along $\p M$. Any such choice generates a unit tangent vector field along $\f$ given by $T(p):=df_p(u_p)/\|df_p(u_p)\|$. Accordingly,  the \emph{binormal} vector field $B:=T\times N$ of $\f$ will  be determined as well. Note that $B$ is orthogonal to the osculating planes of $\f$ and thus orients these planes in a continuos way. If $H$ is the osculating plane of $\f$ at $p$, then the side of $H$ into which $B(p)$ points will be called the region \emph{above} $H$, and the other side will be referred to as the region \emph{below} $H$.
Recall that, by Proposition \ref{prop:osculate}, the planes of singular maximal caps of $M$ are osculating planes of $\f$, and thus are oriented by $B$, once $\p M$ has been oriented.  Thus we may talk about whether a singular maximal cap lies above or below its plane, with respect to an orientation of $\p M$.
The main result of this section is as follows:

\begin{prop}\label{prop:orientation}
$\p M$ may be oriented so that each singular maximal cap of $M$ lies above its plane.
\end{prop}

The chief difficulty in establishing this observation is that the conormal vector field $\nu$ is not in general continuous. Note, however, that for each $p\in\p M$, $(\nu(p), n(p))$ determines an orientation for the orthogonal plane of $\f$ at $p$, where $n(p)$ is the outward to the tangent plane $H_p(0)$. The following lemma shows that this orientation depends continuously on $p$.

\begin{lem}\label{lem:nuTn}
$\nu\times n$ is continuous on $\p M$.
\end{lem}
\begin{proof}
Fix an orientation for  $\p M$, and let $T$ be the corresponding tangent vector field along $\f$. 
Then, for every $p \in\p M$, either $\nu\times n=T$ or $\nu\times n=-T$. Let $A^\pm$ consist of points of $\p M$ such that $\nu\times n=\pm T$ respectively. We need to show that either $A^+=\p M$ or $A^-=\p M$. If $A^+$ is open, then by symmetry $A^-$ is also open, which in turn implies that $A^+$ is closed. Thus, since $\p M$ is connected, it suffices to show that $A^+$ is open. To this end we proceed as follows.

First let us record the following observation. Let the \emph{left} side $L_p^+$ of $L_p$ in $H_p(0)$ be defined as the side to which $n(p)\times T(p)$ points. Further note that $\nu(p)=n(p)\times T(p)$ if and only if $T(p)=\nu(p)\times n(p)$
or $p\in A^+$. Thus, since $\nu(p)\in T_p f$, it follows that  $p\in A^+$ if and only if  $T_p f=L_p^+$. 

Now let $p\in A^+$,  and $U$ be a convex neighborhood of $p$ in $M$ with associated body $K$. 
 As usual, we will identify $U$ with its image $f(U)\subset\p K$, and will suppress $f$. 
 We need to show then that $\Gamma:=U\cap \p M\subset A^+$. More explicitly, given that $T_p U=L_p^+$, we claim that, for all $q\in\Gamma$, 
\begin{equation}\label{eq:nuTn1}
T_q U=L_q^+.
\end{equation}

To establish the last claim, let $H$ be a support plane of $K$ at $p$ such that $\p K$ is a graph over $H$ near $p$. Let $\pi\colon\R^3\to H$ denote the orthogonal projection, and for any object $X\subset \R^3$, set $\ol X:=\pi(X)$. Then $\ol{U}$ will be  embedded  in $H$, assuming that $U$ is sufficiently small. Further, for each $q\in\Gamma$, $\pi\colon H_q(0)\to H$ will be injective. Thus, \eqref{eq:nuTn1} is equivalent to
\begin{equation}\label{eq:nuTn2}
\ol{T_q U}=\ol{L_q^+}.
\end{equation}

Define the left side $(\ol{L_q})^+$ of $\ol{L_q}$ in $H$ as the side to which $u\times \ol T$ points, where $u$ is the outward normal of $H$. Since $\l u, n(q)\r>0$, it follows that the left side of $L_q$ projects onto the left side of $\ol{L_q}$, or $\ol{L_q ^+}=(\ol{L_q})^+$. Further $\ol{L_q}$ coincides with the tangent line $L_{\ol q}$ of $\ol\Gamma$ at $\ol q$. 
Thus $(\ol{L_q})^+=L_{\ol q}^+$. Finally $\ol{T_q U}=T_{\ol{q}}\ol U$ by \eqref{eq:piTp}.
So \eqref{eq:nuTn2} holds if and only if
\begin{equation}\label{eq:nuTn3}
T_{\ol{q}}\ol U=L_{\ol q}^+.
\end{equation}

Finally note that $L_{\ol q}^+$ depends continuously on $q$, since it is determined by  $u\times \ol T$. Further $T_{\ol{q}}\ol U$ depends only on the inward normal of $\ol U$ along $\ol\Gamma$,  which  again depends continuously on $\ol q$. Thus \eqref{eq:nuTn3} holds 
 for all $\ol q\in\ol\Gamma$, since by assumption it holds when $\ol q=\ol p$. 
\end{proof}

The last lemma yields  the key fact needed for the proof of Proposition \ref{prop:osculate}:

\begin{lem}\label{lem:nuB}
$\p M$ may be oriented so that   $\l \nu, B\r\geq 0$.
\end{lem}
\begin{proof}
By Lemma \ref{lem:nuTn}, we may orient $\p M$ so that $\nu\times n=T$.
Since $T$ is orthogonal to both $\nu$ and $B$,    
$$
\l \nu, B\r=\l T\times\nu, T\times B\r=\l n, -N\r =-\l n, N\r\geq 0,
$$
where the last inequality holds by Lemma \ref{lem:nN}.
\end{proof}

Now we are ready to prove the main result of this section:

\begin{proof}[Proof of Proposition \ref{prop:orientation}]
Let $p\in\p M$, and $C$ be a singular maximal cap of $M$ at $p$.
If $M=C$, then it is clear that the proposition holds. So assume that $M\neq C$.
Orient $\p M$ as in Lemma \ref{lem:nuB}, and
let $H$ be the plane of $C$, which  is the osculating plane of $\f$ at $p$ by Proposition \ref{prop:osculate}. We need to show that $B(p)$ points into the half-space $H^+$ of $C$. 

Recall that $\nu(p)\in T_pf\subset H^+$, by Lemma \ref{lem:TpfH+}. 
If $\nu(p)$ points into $\inte(H^+)$, then $B(p)$  points into $H^+$ as well, since $\l B,\nu\r\geq 0$ by Lemma \ref{lem:nuB}, and $B$ is orthogonal to $H$. 
So we may suppose that  $\nu(p)$ lies in $H$. Then $H=H_p(0)$, the tangent plane of $\f$ at $p$, and so $B(p)=\pm n(p)$. Note that here $n(p)$ points into $H^-$ because it is the outward normal of the tangent plane by definition and $f(C)\subset H^+$. Thus
if $B(p)=-n(p)$, then we are done. So  we may suppose  that $B(p)=n(p)$. We will show then that $C=M$, which is a contradiction and thereby completes the proof.

First  we claim that $H$ is  tangent to  $\f$ all along $A:=C\cap\partial M$.
If $A=\{p\}$ then there is nothing to prove. So suppose that there is a point $q\in A$ different from $p$. Then, since $A$ is connected by assumption, there exists a segment $I\subset\p M$ joining $p$ and $q$. Consequently,  $f(I)\subset f(\partial C)\subset H$. So $B$ is constant along $I$. In particular, $B(q)=B(p)=n(p)$. Thus $\l  \nu(q),n(p)\r=\l\nu(q), B(q)\r\geq 0$, by Lemma \ref{lem:nuB}. On the other hand,  $\nu(q)$ points into $H^+$, while $n(p)$ points into $H^-$. So $\l \nu(q), n(p)\r\leq 0$. Hence $\l \nu(q),n(p)\r=0$, which means that $\nu(q)$ lies in $H$. So $H$ is the tangent plane of $\f$ at $q$ as claimed.

Now it follows that $H$ locally supports $f$  along $A$. Consequently there is a neighborhood $V$ of $A$ in $M$ such that 
$$
f(V)\subset H^+.
$$
 On the other hand, by the extension property, there is a neighborhood $U$ of $C$ in $M$ such that $f(U-C)\subset \inte(H^-)$. Note that $U$ is also a neighborhood of $A$ since $A\subset C$. So we may assume that $V\subset U$. Then it follows that 
 $$
 f(V-C)\subset \inte(H^-).
 $$
  So  $V-C=\emptyset$, or $V\subset C$. This in turn yields that $V\cap\partial C\subset\partial M$, since $V$ is open in $M$. Consequently $A$ is open in $\partial M$. But $A$ is compact by definition and therefore is closed in $\partial M$ as well. Thus $A=\partial M$. This yields that $C=M$, which is the desired contradiction.
\end{proof}

\subsection{Sign of torsion; Finishing the proof}\label{subsec:torsion}
It is well-known that if a curve lies locally on one side of its osculating plane then its torsion vanishes, which is precisely how we establish the existence of vertices in this work. Indeed Theorem \ref{thm:singularcap} together with Proposition \ref{prop:osculate} have already ensured the existence of at least two vertices along $\f$.
In order to obtain two more vertices, we need to study  how the torsion  changes sign when the curve lies on one side of its osculating plane. Our first observation below may be regarded as a version of the maximum principle for torsion:

\begin{lem}\label{lem:signchange}
Let $\gamma\colon[0,b]\to\R^3$ be a $\C^3$ immersed curve without inflections. Suppose that $\tau\geq 0$ on $[0,b]$. Then there exists an $\epsilon>0$ such that $\gamma([0,\epsilon])$ lies above the osculating plane of $\gamma$ at $0$.
\end{lem}
\begin{proof}
After a rigid motion, we may suppose that $\gamma(0)=(0,0,0)$, and the Frenet frame $(T,N, B)$ of $\gamma$ at $t=0$ coincides with the standard basis $(i,j,k)$ of $\R^3$. In particular, the tangent line of $\gamma$ at $t=0$ coincides with the $x$-axis. Consequently, for $\epsilon$ sufficiently small, the projection of $\gamma([0,\epsilon])$ into the $xy$-plane, which coincides with the osculating plane of $\gamma$ at $0$, is a graph over the $x$-axis. So after a reparametrization, we may assume that 
$$
\gamma(t)=(t,y(t),z(t)),
$$
 for $t\in[0,\epsilon]$, assuming $\epsilon$ is sufficiently small. We need to show that, for small $\epsilon$,  $z\geq 0$ on $[0,\epsilon]$. To this end, first note that
 $$
 (0,1,0)=N(0)=\frac{\gamma''(0)-\l\gamma''(0),T(0)\r T(0)}{\|\gamma''(0)-\l\gamma''(0),T(0)\r T(0)\|}=\frac{(0,y''(0),z''(0)}{\sqrt{(y''(0))^2+(z''(0))^2}}.
 $$
 Thus  $y''(0)>0$ and $z''(0)=0$. In particular we may choose $\epsilon$ so small that $y''>0$ on $[0,\epsilon]$. 
 Now recall that $\tau=\det(\gamma',\gamma'',\gamma''')/(\kappa^2 \|\gamma'\|^3)$, where $\kappa$ is the curvature of $\gamma$. Thus
 $$
\kappa^2\|\gamma'\|^3\tau=
\left|
\begin{array}{ccc}
1 & y' & z'\\
0 & y'' & z''\\
0 & y'''& z'''\\
\end{array}
\right|
=z'''y''-z''y'''=\(\frac{z''}{y''}\)'(y'')^2.
$$
Consequently, since $z''(0)=0$, it follows that, for any $t\in[0,\epsilon]$, 
$$
\frac{z''(t)}{y''(t)}=\frac{z''(t)}{y''(t)}-\frac{z''(0)}{y''(0)}=\int_0^t \(\frac{z''}{y''}\)'ds=\int_0^t \frac{\kappa^2\|\gamma'\|^3\tau}{(y'')^2}\,ds\geq 0,
$$
 since by assumption $\tau\geq 0$.
Thus $z''\geq 0$ on $[0,\epsilon]$, since $y''>0$ on $[0,\epsilon]$. Further recall that $z(0)=z'(0)=0$,  since by assumption  $\gamma$ passes through the origin and is tangent to the $x$-axis at $t=0$.
So for every $t\in[0,\epsilon]$, there exists $\ol t\in [0,t]$ such that $z(t)=z''(\ol t)\ol t^2/2$, by Taylor's theorem. Hence $z\geq 0$ on $[0,\epsilon]$, as desired. 
\end{proof}

The last observation quickly yields:

\begin{lem}\label{lem:signchange2}
Let $a<0<b$, $\gamma\colon[a,b]\to\R^3$ be a $\C^3$ immersed curve without inflections,  and $H$ be the osculating plane of $\gamma$ at $0$. Suppose that $\gamma([a,b])$ lies below $H$, and $\gamma(a)$, $\gamma(b)$ are disjoint from $H$. Then there are points $\tilde a\in [a,0)$ and $\tilde b\in (0, b]$ such that $\tau(\tilde a)>0$ and $\tau(\tilde b)<0$.
\end{lem}
\begin{proof}
Existence of $\tilde a$ follows  from that of $\tilde b$, since replacing $t$ by $-t$ in $\gamma$ switches the sign of $\tau$.  To find $\tilde b$ let $o$ be the supremum of $\gamma^{-1}(H)$ in $[a,b]$. Then $o\in [0,b)$, since $\gamma(b)\not\in H$. After a translation of  $[a,b]$ we may assume that $o=0$ for convenience. Suppose now, towards a contradiction, that $\tau\geq 0$ on $[0,b]$.  Then, by Lemma \ref{lem:signchange}, there exists $\epsilon>0$ such that  $\gamma([0,\epsilon])$ lies above  $H$. But $\gamma([0,\epsilon])$ also lies below $H$ by assumption. Thus  $\gamma([0,\epsilon])\subset H$. In particular $0$ is not the supremum of $\gamma^{-1}(H)$, which is the desired contradiction.
\end{proof}

We need to record only one more lemma:

\begin{lem}\label{lem:CcapdM}
Let $C$ be a cap of $M$,  $H$ be a plane of $C$, and set $A:=C\cap\p M$. Suppose that  $A\neq \p M$.
Then  there exists a segment $I\subset \p M$ such that $A\subset \inte(I)$, $f(I)\subset H^-$ and $f(I-A)\cap H=\emptyset$.
\end{lem}
\begin{proof}
By the extension property (Proposition \ref{prop:capextend}), there exists  
a neighborhood $U$ of $C$ in $\p M$ such that $f(U-C)\subset\inte(H^-)$.
Note that $U\cap\p M$ is a neighborhood  of $A$. Thus, since $A\neq\partial M$,  and $A$ is closed, there exists a segment $I\subset U\cap \p M$ such that $A\subset \inte(I)$. This is the desired segment, because
$I=A\cup (I-A)$,  $f(A)\subset f(\p C)\subset H$, and  $f(I-A)\subset f(U-C)\subset \inte(H^-)$.
\end{proof}

The above lemmas now yield our penultimate result. Recall that $\mathcal{S}$ denotes the number of singular maximal caps of $M$, and let $\mathcal{V}$ be the number of times the torsion $\tau$ of $\f$ changes sign.

\begin{prop}\label{prop:penultimate}
Let $C$ be a singular maximal cap of $M$, and set $A:=C\cap\partial M$. Suppose that $M$ is not a cap, and $\p M$ is oriented as in  Proposition \ref{prop:orientation}. Then every neighborhood of $A$ contains a segment $I$ containing $A$ such that $\tau<0$ at the initial point of $I$ and $\tau>0$ at the final point of $I$. In particular, 
$$
\mathcal{V}\geq 2\mathcal{S}.
$$
\end{prop}
\begin{proof}
Since $M$ is not a cap,  $\mathcal{S}\geq 2$ by Theorem \ref{thm:singularcap}.
Let $C_i$, $i=1, \dots, k\leq\mathcal{S}$ be distinct singular maximal caps of $M$ for some finite integer $k$. Recall that $A_i:=C_i\cap\p M$ are disjoint, since, as we verified in Theorem \ref{thm:singularcap}, the sets where the maximal caps of $M$ intersect $\p M$ form a partition of $\p M$.  Furthermore, $A_i$ are proper subsets of $\p M$,  since $M$ is not a cap by assumption. Thus we may let $I_i\subset\p M$ be intervals containing $A_i$ as in 
Lemma \ref{lem:CcapdM}. Since $A_i$ are disjoint, and $k$ is finite, we may suppose that $I_i$ are disjoint as well.

\begin{figure}[h]
\begin{overpic}[height=1.1in]{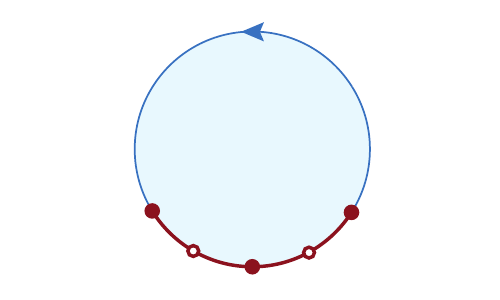}
\put(48,1){\Small$p_i$}
\put(47,28){\Small$M$}
\put(23,14){\Small$a_i$}
\put(34,4){\Small$\tilde a_i$}
\put(73,14){\Small$b_i$}
\put(64,3){\Small$\tilde b_i$}
\put(37.5,12.5){\Small$-$}
\put(59,12.5){\Small$+$}
\end{overpic}
\caption{}
\label{fig:circle}
\end{figure}

Identify each $I_i$ with an interval $[a_i,b_i]\subset\R$, where $a_i<0<b_i$, via an orientation preserving parametrization which we suppress for simplicity, see Figure \ref{fig:circle}.
 By Proposition \ref{prop:osculate}, the planes $H_i$ of $C_i$ are osculating planes of $\f$ at some points $p_i\in A_i\subset\inte(I_i)$, which we may identify with $0\in(a_i, b_i)$. By Proposition \ref{prop:orientation},  $f(C_i)$ lie above $H_i$. Thus, by Lemma \ref{lem:CcapdM}, $f([a_i, b_i])$ lie below $H_i$, while $f(a_i)$, $f(b_i)$ are disjoint from $H_i$.
Consequently, by Lemma \ref{lem:signchange2}, there are points $\tilde a_i\in[a_i,p_i)$ and $\tilde b_i\in(p_i,b_i]$ such that 
 $\tau(\tilde a_i)<0$ and $\tau(\tilde b_i)>0$. Since $I_i$ are disjoint, these points  have the cyclic ordering $[\tilde a_1, \tilde b_1, \dots, \tilde a_k, \tilde b_k]$ in $\p M$. Thus $\mathcal{V}\geq 2k$, which completes the proof.
\end{proof}

The last proposition quickly completes the proof of Theorem \ref{thm:main}, and  yields a refinement of it via Theorem \ref{thm:singularcap} as follows:

\begin{thm}[Main Theorem, Full Version]\label{thm:main2}
 Suppose that $\tau\not\equiv 0$, and let $C\subset M$ be a maximal cap. Then each component $U$ of $\p M-C$ contains a point where the corresponding maximal cap is singular.  Thus $\tau$ changes sign in $U$, and 
$$
\mathcal{V}\geq 2\mathcal{S}\geq 2(\mathcal{T}+2).
$$
\end{thm}
\begin{proof}
If $C$ is singular, then $\tau$ changes sign in $U$ by Proposition \ref{prop:penultimate}. So we may suppose that $\p M-C$ has more than one component. Then $\cl(U)$ is a segment of $\p M$.
Let $\p M$ be oriented as in Proposition \ref{prop:orientation}, and $p_1$, $p_2$ be the initial and final points of $\cl(U)$ respectively. Each $p_i$ belongs to a component $A_i$ of $\p M\cap C$. Let $I_i$ be intervals containing $A_i$ as in Proposition \ref{prop:penultimate}. Then $\tau(a_i)<0$ and $\tau(b_i)>0$ where $a_i$, $b_i$ are the initial and final points of $I_i$ respectively.
Choosing $I_i$ sufficiently small, we may ensure that  $b_1$, $a_2 \in U$. So $\tau$ changes sign in $U$.
Finally, $\mathcal{V}\geq 2\mathcal{S}$, by Proposition \ref{prop:penultimate},  and $\mathcal{S}\geq \mathcal{T}+2$, by Theorem \ref{thm:singularcap}, which complete the proof.
\end{proof}

\section{Examples}
Here we discuss a number of examples which establish  the sharpness of our main results in various respects.

\begin{example}[Local flatness on the boundary]\label{ex:dumbbell}
There exists a $\C^\infty$ embedded disk in $\R^3$ which is locally convex (and therefore has $4$ boundary vertices by Theorem \ref{thm:main}); however, its boundary torsion changes sign only twice. Thus the local nonflatness assumption in Theorem \ref{thm:main} is essential for ensuring that  torsion either changes sign (at least) $4$ times, or else is identically zero. First we describe the boundary curve. Let $h\colon[0,\infty)\to\R$ be the convex function defined by
$$
h(t):=
\left\{
\begin{array}{lrl}
\exp(\frac{8}{1-4t}),   &   \frac14   &  \hspace{-6pt}< t,\\
 0,   &      0   &  \hspace{-6pt}\leq t\leq \frac14.\\
\end{array}
\right.
$$
Extend $h$ to an odd function on $\R$ by setting $h(t):=-h(-t)$. Next  let $\gamma\colon [0,2\pi]\to\R^3$ be the curve given by
$$
\gamma(t):=\(\,\cos(t), \,\(5/4 - \sin^2(t)\)\sin(t),\, 5\, h\(\cos(t)\)\,\),
$$
and let $\Gamma$ denote the trace of $\gamma$.
Note that  $\Gamma$ lies on the cylindrical surface $S\subset\R^3$ given by $z=5h(x)$. Further, $\Gamma$ is a simple closed curve on $S$ as can be seen from its projection into the $xy$-plane, see Figure \ref{fig:dumbbell2}. Let $M$ be the compact region bounded by $\Gamma$ in $S$. Then $M$ is locally convex with respect to the inclusion map $f\colon M\to\R^3$; however, a calculation shows that the torsion of $\Gamma$ changes sign only twice. The graph of torsion appears on the right hand side of Figure \ref{fig:dumbbell2}.

  \begin{figure}[h]
\begin{overpic}[height=1.25in]{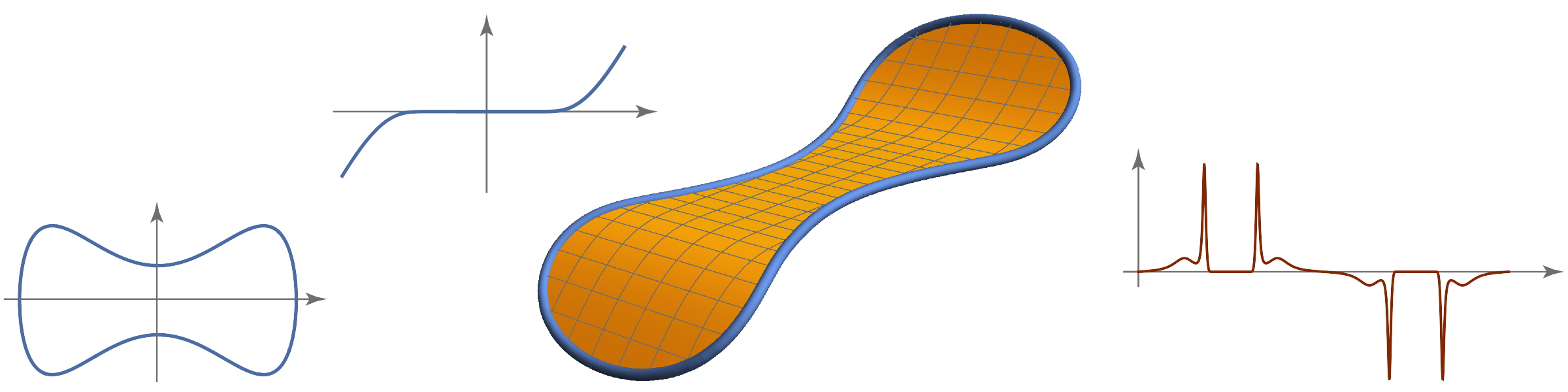}
\put(21.5,5){\Small $x$}
\put(10.5,12.5){\Small $y$}
\put(42.5,17){\Small $x$}
\put(31.5,24){\Small $z$}
\put(100,8){\Small $t$}
\put(73,15.5){\Small $\tau$}
\put(72,5){\Small $0$}
\put(83.5,5){\Small $\pi$}
\put(95,5){\Small $2\pi$}
\put(73,15.5){\Small $\tau$}
\end{overpic}
\caption{}
\label{fig:dumbbell2}
\end{figure}

\end{example}

\begin{example}[Zero curvature on the boundary]\label{ex:dumbbell2}
There exists a $\C^\infty$ embedded disk  in $\R^3$ which has nonnegative curvature and only two boundary vertices. Thus the condition in Corollary \ref{cor:main} that $f$ have positive curvature on $\p M$ is necessary. This surface is another cylindrical disk which is quite similar to the construction in Example \ref{ex:dumbbell} with the  exception that here the boundary has a different third coordinate:
$$
\gamma(t):=\(\,\cos(t), \,\(5/4 - \sin^2(t)\)\sin(t), \,\cos^3(t)/2\,\).
$$
So $\Gamma$ will have the same dumbbell shaped projection into the $xy$-plane as in the previous example, but it will lie on a different cylindrical surface $S\subset\R^3$ given by $z=x^3/2$, see Figure \ref{fig:dumbbell}. Once  again we let $M$ be the compact region bounded by $\Gamma$ in $S$. As the graph on the right hand side of Figure \ref{fig:dumbbell} shows, the torsion of $\Gamma$ vanishes only twice in this example. A curve with similar properties  is  discussed in a paper of R{\o}gen \cite[Ex. 13]{rogen:flat}, where it is also shown that the boundary of a disk of zero curvature in $\R^3$  always has at least two vertices.

 \begin{figure}[h]
\begin{overpic}[height=1.25in]{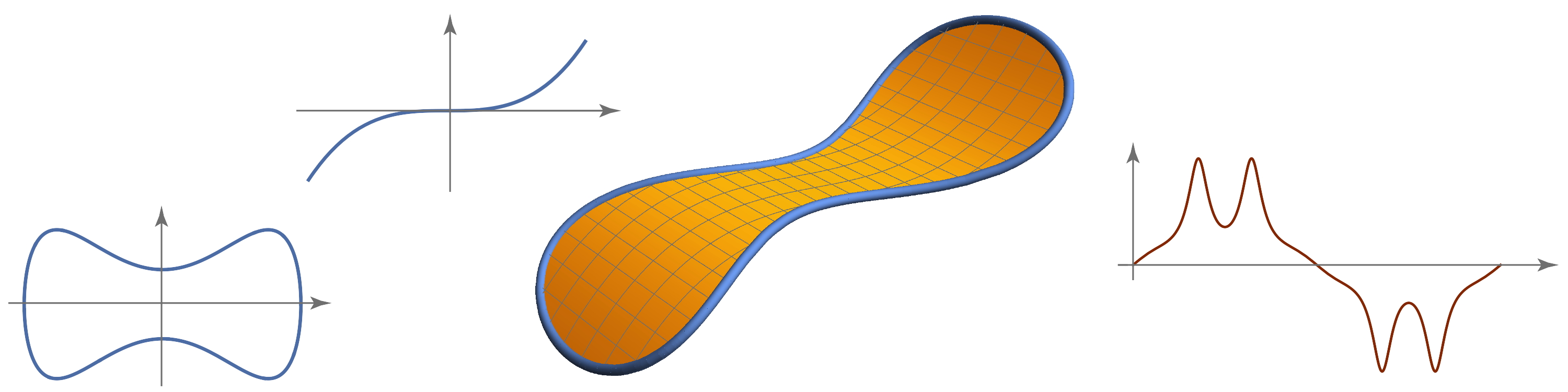}
\put(21.5,4.5){\Small $x$}
\put(10.75,12){\Small $y$}
\put(40,17){\Small $x$}
\put(29.2,24){\Small $z$}
\put(100,8){\Small $t$}
\put(71.75,5.5){\Small $0$}
\put(83,5.5){\Small $\pi$}
\put(95,5.5){\Small $2\pi$}
\put(73,15.5){\Small $\tau$}
\end{overpic}
\caption{}
\label{fig:dumbbell}
\end{figure}
\end{example}

\begin{example}[Negative curvature in the interior]\label{ex:negative}
There exists a $\C^\infty$ simple closed curve $\Gamma\subset\R^3$ which has non vanishing torsion, but bounds a smooth immersion of a disk with positive curvature along $\Gamma$. Thus the condition in Corollary \ref{cor:main} that the curvature of $f$ be nonnegative everywhere is essential. More generally, Theorems \ref{thm:main} or  \ref{thm:main2} do not hold if $f$ is  locally convex only near $\p M$. 

As we mentioned in the introduction, Rosenberg \cite{rosenberg:constant} has shown that a necessary condition for $\Gamma$ to bound a surface with positive curvature near $\Gamma$ is that the self-linking number  $\mathrm{SL}(\Gamma)$ be even. Here we show that this condition is also sufficient.

First note that there are $\C^\infty$ simple closed curves $\Gamma\subset \R^3$ with non vanishing torsion and even self-linking number. For instance, we may let $\Gamma$ be a  torus knot of type $(1,2m)$ for large $m$. Indeed, the self-linking number of any torus knot of type $(1,q)$ is $q$. Thus,
$$
\mathrm{SL}(\Gamma)=2m.
$$
Furthermore, as has been studied by Costa \cite{costa:twisted}, the torsion of $\Gamma$ does not vanish when $m$ is sufficiently large.  Alternatively, we may use a result of Aicardi \cite{aicardi} who has constructed closed curves of non vanishing torsion with any self-linking number. 

Next note that, as has been observed by Gluck and Pan \cite{gluck&pan}, see also  \cite{ghomi:stconvex},   we may extend $\Gamma$, or any closed curve without inflections, to a smooth embedded positively curved annulus $A$. More precisely, there exists a $\C^\infty$ embedding $f\colon \S^1\times [0,1]\to\R^3$ such that $f(\S^1\times\{0\})=\Gamma$, and $f$ has positive curvature everywhere. Next recall that, by Lemma \ref{lem:nN}, $\l n, N\r\leq 0$ along $\Gamma$, where $N$ is the principal normal of $\Gamma$ and $n$ is the outward normal of $A$. Set $\tilde n:=-n$. Then
$$
\mathrm{SL}(\Gamma):=\mathrm{L}(\Gamma,\Gamma+\epsilon N)=\mathrm{L}(\Gamma,\Gamma+\epsilon \tilde n),
$$
where $\mathrm{L}$ stands for linking number, and $\Gamma+\epsilon N$, $\Gamma+\epsilon \tilde n$ denote, respectively, small perturbations of $\Gamma$ along the vector fields $N$, $\tilde n$. The second inequality above holds because, for small $\epsilon$, and all $\theta\in[0,\pi/2]$, $h(\theta):=\Gamma+\epsilon(\cos(\theta)N+\sin(\theta)\tilde n)$ is disjoint from $\Gamma$, since $\l N,\tilde n\r\geq 0$.   Thus $h(\theta)$ yields an isotopy between $\Gamma+\epsilon N$, and $\Gamma+\epsilon \tilde n$ in the complement of $\Gamma$.

So $\mathrm{L}(\Gamma,\Gamma+\epsilon \tilde n)=2m$, which  indicates that $A$ ``twists around" $\Gamma$ an even number of times. Consequently, by a deep result of  Whitney \cite[Thm. 7]{whitney:2n-1}, a neighborhood of $\Gamma$ in $A$ may be extended  to a smooth immersion of a disk, which yields our desired surface. More generally, for any compact connected orientable surface $M$ with connected boundary $\p M$, Whitney constructs a smooth immersion $f\colon M\to\R^3$ such that $f(\p M)=\Gamma$, $f$ is one-to-one of $\p M$, and $f(U)\subset A$ for some neighborhood $U$ of $\p M$ in $M$. 
\end{example}

\begin{example}[The complete case]
There exists a $\C^\infty$ complete noncompact positively curved surface in $\R^3$ bounded by a closed curve with only two vertices, where by ``complete" we mean that the induced metric on the surface is complete in the sense of convergence of Cauchy sequences. Thus this example shows that Theorem \ref{thm:main} cannot be extended to complete surfaces which are  noncompact. To generate this surface we  employ the well-known fact that the vertices of a spherical curve coincide with critical points of its geodesic curvature. It follows then that the curve in Figure \ref{fig:capswirl}(a) has only two vertices. Note that this curve, which we call $\Gamma$, is contained in a hemisphere. By moving $\Gamma$ parallel to itself outward, we may generate an immersed annulus $A$ in the northern hemisphere bounded by $\Gamma$ and a double covering of the equator. Then the projective transformation $(x,y,z)\mapsto(x/z,y/z,1/z)$ turns $A$ into the desired complete surface.

 \begin{figure}[h]
\begin{overpic}[height=1.2in]{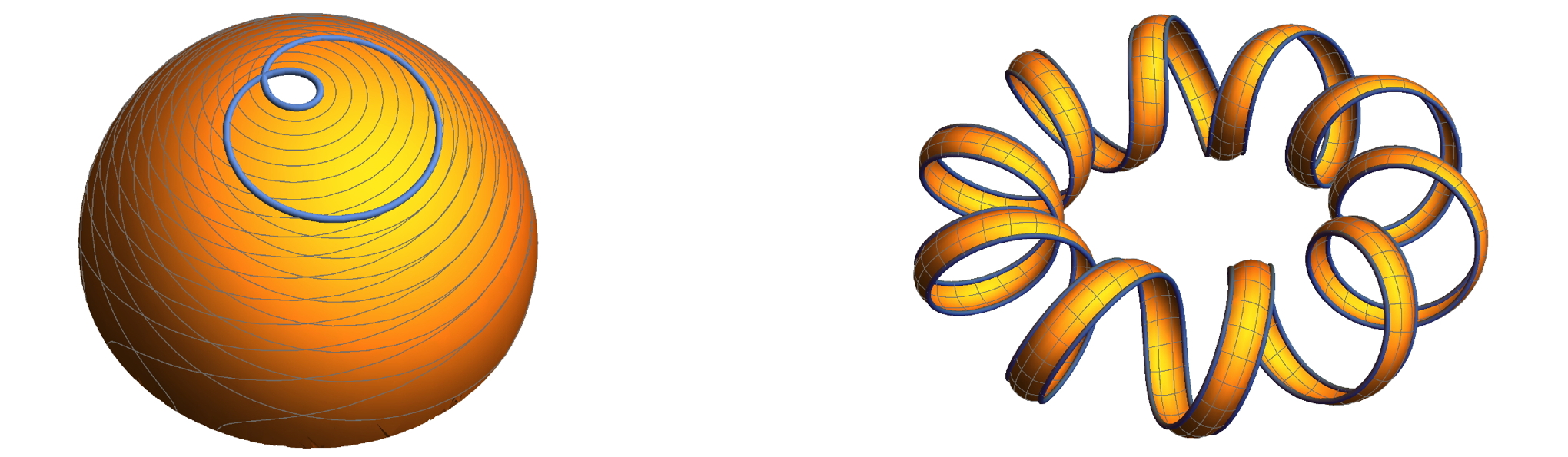}
\put(18,-4){$(a)$}
\put(74,-3){$(b)$}
\end{overpic}
\caption{}
\label{fig:capswirl}
\end{figure}
\end{example}

\begin{example}[The higher genus case]
There exists a $\C^\infty$ embedded annulus with positive curvature in $\R^3$ whose boundaries have no vertices. 
To construct this surface recall that, as we pointed out in Example \ref{ex:negative}, there are plenty of torus knots with non vanishing torsion. For instance a curve of type $(1,10)$ on a thin torus is one such example, see Figure \ref{fig:capswirl} (b). By the discussion in Example \ref{ex:negative}, we may extend this curve, say $\Gamma$, to a positively curved $\C^\infty$ annulus $A$. Let $M\subset A$ be a sub-annulus containing $\Gamma$ whose boundary components are
 $\C^3$-close to $\Gamma$. Then $\p M$ will have nonvanishing torsion.
Thus the simply connected assumption on $M$ in Theorem \ref{thm:main} is not superfluous. On the other hand, we do not know of any  non simply connected surface $M$ with connected boundary which violates Theorem \ref{thm:main}.
\end{example}

\addtocontents{toc}{\setcounter{tocdepth}{-1}}
\section*{Acknowledgments}
\addtocontents{toc}{\setcounter{tocdepth}{2}}

The author thanks Harold Rosenberg  for his interesting question \cite{rosenberg:constant} which prompted this work. Further he is indebted to the papers of Masaaki Umehara \cite{umehara2,thorbergsson&umehara}, and Gudlaugur Thorbergsson  \cite{thorbergsson&umehara} for developing the notion of intrinsic circle systems. 

\bibliographystyle{abbrv}
\bibliography{references}

\begin{thebibliography}{10}

\bibitem{aicardi}
F.~Aicardi.
\newblock Self-linking of space curves without inflections and its
  applications.
\newblock {\em Funct. Anal. Appl.}, 34(2):79--85, 2000.

\bibitem{alexander&ghomi:chp}
S.~Alexander and M.~Ghomi.
\newblock The convex hull property and topology of hypersurfaces with
  nonnegative curvature.
\newblock {\em Adv. Math.}, 180(1):324--354, 2003.

\bibitem{agw}
S.~Alexander, M.~Ghomi, and J.~Wong.
\newblock Topology of {R}iemannian submanifolds with prescribed boundary.
\newblock {\em Duke Math. J.}, 152(3):533--565, 2010.

\bibitem{alexandrov:polyhedra}
A.~D. Alexandrov.
\newblock {\em Convex polyhedra}.
\newblock Springer Monographs in Mathematics. Springer-Verlag, Berlin, 2005.
\newblock Translated from the 1950 Russian edition by N. S. Dairbekov, S. S.
  Kutateladze and A. B. Sossinsky, With comments and bibliography by V. A.
  Zalgaller and appendices by L. A. Shor and Yu. A. Volkov.

\bibitem{alexandrov:intrinsic}
A.~D. Alexandrov.
\newblock {\em A. {D}. {A}lexandrov selected works. {P}art {II}}.
\newblock Chapman \& Hall/CRC, Boca Raton, FL, 2006.
\newblock Intrinsic geometry of convex surfaces, Edited by S. S. Kutateladze,
  Translated from the Russian by S. Vakhrameyev.

\bibitem{angenent:inflection}
S.~Angenent.
\newblock Inflection points, extatic points and curve shortening.
\newblock In {\em Hamiltonian systems with three or more degrees of freedom
  ({S}'{A}gar\'o, 1995)}, volume 533 of {\em NATO Adv. Sci. Inst. Ser. C Math.
  Phys. Sci.}, pages 3--10. Kluwer Acad. Publ., Dordrecht, 1999.

\bibitem{arnold:plane}
V.~I. Arnold.
\newblock {\em Topological invariants of plane curves and caustics}, volume~5
  of {\em University Lecture Series}.
\newblock American Mathematical Society, Providence, RI, 1994.
\newblock Dean Jacqueline B. Lewis Memorial Lectures presented at Rutgers
  University, New Brunswick, New Jersey.

\bibitem{arnold:sphere}
V.~I. Arnold.
\newblock The geometry of spherical curves and quaternion algebra.
\newblock {\em Uspekhi Mat. Nauk}, 50(1(301)):3--68, 1995.

\bibitem{barner&flohr}
M.~Barner and F.~Flohr.
\newblock Der vierscheitesatz und seine verallgemeinrugen,.
\newblock {\em Der Mathe- matikunterricht}, 44:43--73, 1958.

\bibitem{BirteaOrtegaRatiu}
P.~Birtea, J.-P. Ortega, and T.~S. Ratiu.
\newblock Openness and convexity for momentum maps.
\newblock {\em Trans. Amer. Math. Soc.}, 361(2):603--630, 2009.

\bibitem{bisztriczky}
T.~Bisztriczky.
\newblock On the four-vertex theorem for space curves.
\newblock {\em J. Geom.}, 27(2):166--174, 1986.

\bibitem{bjorndahl&karshon}
C.~Bjorndahl and Y.~Karshon.
\newblock Revisiting {T}ietze-{N}akajima: local and global convexity for maps.
\newblock {\em Canad. J. Math.}, 62(5):975--993, 2010.

\bibitem{bose}
R.~C. Bose.
\newblock On the number of circles of curvature perfectly enclosing or
  perfectly enclosed by a closed convex oval.
\newblock {\em Math. Z.}, 35(1):16--24, 1932.

\bibitem{bray&jauregui}
H.~L. Bray and J.~L. Jauregui.
\newblock On curves with nonnegative torsion.
\newblock {\em arXiv:1312.5171}, 2014.

\bibitem{bbi:book}
D.~Burago, Y.~Burago, and S.~Ivanov.
\newblock {\em A course in metric geometry}, volume~33 of {\em Graduate Studies
  in Mathematics}.
\newblock American Mathematical Society, Providence, RI, 2001.

\bibitem{cairns2}
G.~Cairns, M.~McIntyre, and M.~{\"O}zdemir.
\newblock A six-vertex theorem for bounding normal planar curves.
\newblock {\em Bull. London Math. Soc.}, 25(2):169--176, 1993.

\bibitem{damon}
J.~Damon.
\newblock The global medial structure of regions in {$\Bbb R^3$}.
\newblock {\em Geom. Topol.}, 10:2385--2429, 2006.

\bibitem{gluck:notices}
D.~DeTurck, H.~Gluck, D.~Pomerleano, and D.~S. Vick.
\newblock The four vertex theorem and its converse.
\newblock {\em Notices Amer. Math. Soc.}, 54(2):192--207, 2007.

\bibitem{federer:book}
H.~Federer.
\newblock {\em Geometric measure theory}.
\newblock Springer-Verlag New York Inc., New York, 1969.
\newblock Die Grundlehren der mathematischen Wissenschaften, Band 153.

\bibitem{freedman&krushkal}
M.~Freedman and V.~Krushkal.
\newblock Geometric complexity of embeddings in {$\Bbb{R}^d$}.
\newblock {\em Geom. Funct. Anal.}, 24(5):1406--1430, 2014.

\bibitem{ghomi:stconvex}
M.~Ghomi.
\newblock Strictly convex submanifolds and hypersurfaces of positive curvature.
\newblock {\em J. Differential Geom.}, 57(2):239--271, 2001.

\bibitem{ghomi:verticesA}
M.~Ghomi.
\newblock A {R}iemannian four vertex theorem for surfaces with boundary.
\newblock {\em Proc. Amer. Math. Soc.}, 139(1):293--303, 2011.

\bibitem{ghomi:verticesC}
M.~Ghomi.
\newblock Tangent lines, inflections, and vertices of closed curves.
\newblock {\em Duke Math. J.}, 162(14):2691--2730, 2013.

\bibitem{ghomi:verticesB}
M.~Ghomi.
\newblock Vertices of closed curves in {R}iemannian surfaces.
\newblock {\em Comment. Math. Helv.}, 88(2):427--448, 2013.

\bibitem{ghomi&howard:tancones}
M.~Ghomi and R.~Howard.
\newblock Tangent cones and regularity of real hypersurfaces.
\newblock {\em J. Reine Angew. Math.}, 697:221--247, 2014.

\bibitem{ghys&tabachnikov}
{\'E}.~Ghys, S.~Tabachnikov, and V.~Timorin.
\newblock Osculating curves: around the {T}ait-{K}neser theorem.
\newblock {\em Math. Intelligencer}, 35(1):61--66, 2013.

\bibitem{gluck&pan}
H.~Gluck and L.-H. Pan.
\newblock Embedding and knotting of positive curvature surfaces in $3$-space.
\newblock {\em Topology}, 37(4):851--873, 1998.

\bibitem{greene&wu:rigidity}
R.~E. Greene and H.~Wu.
\newblock On the rigidity of punctured ovaloids.
\newblock {\em Ann. of Math. (2)}, 94:1--20, 1971.

\bibitem{greene&wu:rigidityII}
R.~E. Greene and H.~Wu.
\newblock On the rigidity of punctured ovaloids. {I}{I}.
\newblock {\em J. Differential Geometry}, 6:459--472, 1972.
\newblock Collection of articles dedicated to S. S. Chern and D. C. Spencer on
  their sixtieth birthdays.

\bibitem{gromov:metric}
M.~Gromov.
\newblock {\em Metric structures for {R}iemannian and non-{R}iemannian spaces}.
\newblock Birkh\"auser Boston Inc., Boston, MA, 1999.

\bibitem{gromov&guth}
M.~Gromov and L.~Guth.
\newblock Generalizations of the {K}olmogorov-{B}arzdin embedding estimates.
\newblock {\em Duke Math. J.}, 161(13):2549--2603, 2012.

\bibitem{guan&spruck:chp}
B.~Guan and J.~Spruck.
\newblock Locally convex hypersurfaces of constant curvature with boundary.
\newblock {\em Comm. Pure Appl. Math.}, 57(10):1311--1331, 2004.

\bibitem{haupt:bose}
O.~Haupt.
\newblock Verallgemeinerung eines {S}atzes von {R}. {C}. {B}ose \"uber die
  {A}nzahl der {S}chmiegkreise eines {O}vals, die vom {O}val umschlossen werden
  oder das {O}val umschlie\ss en.
\newblock {\em J. Reine Angew. Math.}, 239/240:339--352, 1969.

\bibitem{houston&vanmanen}
K.~Houston and M.~van Manen.
\newblock A {B}ose type formula for the internal medial axis of an embedded
  manifold.
\newblock {\em Differential Geom. Appl.}, 27(2):320--328, 2009.

\bibitem{ivanisvili2}
P.~Ivanisvili.
\newblock Bellman {VS} {B}eurling: sharp estimates of uniform convexity for
  ${L}^p$ spaces.
\newblock {\em arXiv:1405.6229}, 2014.

\bibitem{ivanisvili1}
P.~Ivanisvili.
\newblock Burkholder's {M}artingale transform.
\newblock {\em arXiv:1402.4751}, 2014.

\bibitem{jackson:bulletin}
S.~B. Jackson.
\newblock Vertices for plane curves.
\newblock {\em Bull. Amer. Math. Soc.}, 50:564--478, 1944.

\bibitem{klein}
F.~Klein.
\newblock Eine neue {R}elation zwischen den {S}ingularit\"aten einer
  algebraischen {C}urve.
\newblock {\em Math. Ann.}, 10(2):199--209, 1876.

\bibitem{kneser}
A.~Kneser.
\newblock Bemerkungen \"{u}ber die anzahl der extrema des kr\"{u}mmung auf
  geschlossenen kurven und \"{u}ber verwandte fragen in einer night
  eucklidischen geometrie.
\newblock In {\em Festschrift Heinrich Weber}, pages 170--180. Teubner, 1912.

\bibitem{hkneser}
H.~Kneser.
\newblock Neuer beweis des vierscheitelsatzes.
\newblock {\em Christian Huygens}, 2:315---318, 1922.

\bibitem{mather}
J.~N. Mather.
\newblock Distance from a submanifold in {E}uclidean space.
\newblock In {\em Singularities, {P}art 2 ({A}rcata, {C}alif., 1981)},
  volume~40 of {\em Proc. Sympos. Pure Math.}, pages 199--216. Amer. Math.
  Soc., Providence, RI, 1983.

\bibitem{mobius}
A.~F. M\"{o}bius.
\newblock \"{U}ber die grundformen der linien der dritten ordnung.
\newblock In {\em Gesammelte Werke II}, pages 89--176. Verlag von S. Hirzel,
  Leipzig, 1886.

\bibitem{mohrmann}
H.~Mohrmann.
\newblock Die minimalzahl der station\"{a}ren ebenen eines r\"{a}umlichen
  ovals.
\newblock {\em Sitz Ber Kgl Bayerichen Akad. Wiss. Math. Phys. Kl.}, pages
  1--3, 1917.

\bibitem{mukhopadhyaya}
S.~Mukhopadhyaya.
\newblock New methods in the geometry of a plane arc.
\newblock {\em Bull. Calcutta Math. Soc. I}, pages 31--37, 1909.

\bibitem{ballesteros&fuster}
J.~J. Nu{\~n}o~Ballesteros and M.~C. Romero~Fuster.
\newblock A four vertex theorem for strictly convex space curves.
\newblock {\em J. Geom.}, 46(1-2):119--126, 1993.

\bibitem{osserman:vertex}
R.~Osserman.
\newblock The four-or-more vertex theorem.
\newblock {\em Amer. Math. Monthly}, 92(5):332--337, 1985.

\bibitem{ovsienko&tabachnikov}
V.~Ovsienko and S.~Tabachnikov.
\newblock {\em Projective differential geometry old and new}, volume 165 of
  {\em Cambridge Tracts in Mathematics}.
\newblock Cambridge University Press, Cambridge, 2005.
\newblock From the Schwarzian derivative to the cohomology of diffeomorphism
  groups.

\bibitem{pak:book}
I.~Pak.
\newblock {\em Lectures on discrete and polyhedral geometry}.
\newblock www.math.umn.edu/$\sim$pak/, 2008.

\bibitem{pardon2}
J.~Pardon.
\newblock On the distortion of knots on embedded surfaces.
\newblock {\em Ann. Math.}, 174(1):637--646, 2011.

\bibitem{pinkall}
U.~Pinkall.
\newblock On the four-vertex theorem.
\newblock {\em Aequationes Math.}, 34(2-3):221--230, 1987.

\bibitem{pogorelov:book}
A.~V. Pogorelov.
\newblock {\em Extrinsic geometry of convex surfaces}.
\newblock American Mathematical Society, Providence, R.I., 1973.
\newblock Translated from the Russian by Israel Program for Scientific
  Translations, Translations of Mathematical Monographs, Vol. 35.

\bibitem{rmbc}
E.~Raphael, J.-M. di~Meglio, M.~Berger, and E.~Calabi.
\newblock Convex particles at interfaces.
\newblock {\em J. Phys. I France}, 2:571--579, 1992.

\bibitem{rockafellar}
R.~T. Rockafellar.
\newblock {\em Convex analysis}.
\newblock Princeton Mathematical Series, No. 28. Princeton University Press,
  Princeton, N.J., 1970.

\bibitem{costa:twisted}
S.~I. Rodrigues~Costa.
\newblock On closed twisted curves.
\newblock {\em Proc. Amer. Math. Soc.}, 109(1):205--214, 1990.

\bibitem{rogen:flat}
P.~R{\o}gen.
\newblock Boundaries of flat compact surfaces in 3-space.
\newblock {\em Ann. Global Anal. Geom.}, 19(4):377--407, 2001.

\bibitem{fuster&sedykh}
M.~C. Romero-Fuster and V.~D. Sedykh.
\newblock A lower estimate for the number of zero-torsion points of a space
  curve.
\newblock {\em Beitr\"age Algebra Geom.}, 38(1):183--192, 1997.

\bibitem{rosenberg:constant}
H.~Rosenberg.
\newblock Hypersurfaces of constant curvature in space forms.
\newblock {\em Bull. Sci. Math.}, 117(2):211--239, 1993.

\bibitem{sacksteder:convex}
R.~Sacksteder.
\newblock On hypersurfaces with no negative sectional curvatures.
\newblock {\em Amer. J. Math.}, 82:609--630, 1960.

\bibitem{SackstederValentine}
R.~Sacksteder, E.~G. Straus, and F.~A. Valentine.
\newblock A generalization of a theorem of {T}ietze and {N}akajima on local
  convexity.
\newblock {\em J. London Math. Soc.}, 36:52--56, 1961.

\bibitem{schneider:book}
R.~Schneider.
\newblock {\em Convex bodies: the {B}runn-{M}inkowski theory}.
\newblock Cambridge University Press, Cambridge, 1993.

\bibitem{sedykh:vertex}
V.~D. Sedykh.
\newblock Four vertices of a convex space curve.
\newblock {\em Bull. London Math. Soc.}, 26(2):177--180, 1994.

\bibitem{smith2012}
G.~Smith.
\newblock Compactness results for immersions of prescribed {G}aussian curvature
  {I} - analytic aspects.
\newblock {\em Adv. Math.}, 229(2):731--769, 2012.

\bibitem{smith:lecturenotes}
G.~Smith.
\newblock The plateau problem for gaussian curvature.
\newblock Lecture notes for a minicourse at XVII Brazilian School of Geometry,
  \emph{arXiv:1206.5544}, 2012.

\bibitem{spivak:v2}
M.~Spivak.
\newblock {\em A comprehensive introduction to differential geometry. {V}ol.
  {II}}.
\newblock Publish or Perish Inc., Wilmington, Del., second edition, 1979.

\bibitem{thorbergsson&umehara}
G.~Thorbergsson and M.~Umehara.
\newblock A unified approach to the four vertex theorems. {II}.
\newblock In {\em Differential and symplectic topology of knots and curves},
  volume 190 of {\em Amer. Math. Soc. Transl. Ser. 2}, pages 229--252. Amer.
  Math. Soc., Providence, RI, 1999.

\bibitem{trudinger&wang}
N.~S. Trudinger and X.-J. Wang.
\newblock On locally convex hypersurfaces with boundary.
\newblock {\em J. Reine Angew. Math.}, 551:11--32, 2002.

\bibitem{umehara}
M.~Umehara.
\newblock {$6$}-vertex theorem for closed planar curve which bounds an immersed
  surface with nonzero genus.
\newblock {\em Nagoya Math. J.}, 134:75--89, 1994.

\bibitem{umehara2}
M.~Umehara.
\newblock A unified approach to the four vertex theorems. {I}.
\newblock In {\em Differential and symplectic topology of knots and curves},
  volume 190 of {\em Amer. Math. Soc. Transl. Ser. 2}, pages 185--228. Amer.
  Math. Soc., Providence, RI, 1999.

\bibitem{valentine:book}
F.~A. Valentine.
\newblock {\em Convex sets}.
\newblock McGraw-Hill Series in Higher Mathematics. McGraw-Hill Book Co., New
  York-Toronto-London, 1964.

\bibitem{vH:convex}
J.~Van~Heijenoort.
\newblock On locally convex manifolds.
\newblock {\em Comm. Pure Appl. Math.}, 5:223--242, 1952.

\bibitem{whitney:2n-1}
H.~Whitney.
\newblock The singularities of a smooth {$n$}-manifold in {$(2n-1)$}-space.
\newblock {\em Ann. of Math. (2)}, 45:247--293, 1944.

\bibitem{yau:problems2}
S.-T. Yau.
\newblock Open problems in geometry.
\newblock In {\em Differential geometry: partial differential equations on
  manifolds (Los Angeles, CA, 1990)}, pages 1--28. Amer. Math. Soc.,
  Providence, RI, 1993.

\end{thebibliography}

\end{document}